\documentclass[reqno]{amsart}

\pdfoutput=1

\usepackage[alphabetic]{amsrefs}
\usepackage{amsmath}
\usepackage{amsfonts}
\usepackage{amssymb}
\usepackage{enumerate}
\usepackage{amstext}
\usepackage{amsbsy}
\usepackage{amsopn}
\usepackage{bbm,amsthm}
\usepackage{amscd}
\usepackage[pdftex]{color}
\usepackage{amsxtra}
\usepackage{upref}
\usepackage{epstopdf}
\usepackage{graphicx,color}
\usepackage{hyperref}
\usepackage{longtable}
\usepackage{graphicx}
\usepackage[francais, english]{babel}

\DeclareFontFamily{OML}{rsfs}{\skewchar\font'177}
\DeclareFontShape{OML}{rsfs}{m}{n}{ <5> <6> rsfs5 <7> <8> <9> rsfs7
  <10> <10.95> <12> <14.4> <17.28> <20.74> <24.88> rsfs10 }{}
\DeclareMathAlphabet{\mathfs}{OML}{rsfs}{m}{n}

\newtheorem{theorem}{Theorem}
\newtheorem{lemma}[theorem]{Lemma}
\newtheorem{proposition}[theorem]{Proposition}
\newtheorem{corollary}[theorem]{Corollary}

\theoremstyle{definition}

\theoremstyle{remark}
\newtheorem{remark}[theorem]{\bf Remark}

\numberwithin{equation}{section}
\numberwithin{theorem}{section}

\newcommand{\intav}[1]{\mathchoice {\mathop{\vrule width 6pt height 3 pt depth  -2.5pt
\kern -8pt \intop}\nolimits_{\kern -6pt#1}} {\mathop{\vrule width
5pt height 3  pt depth -2.6pt \kern -6pt \intop}\nolimits_{#1}}
{\mathop{\vrule width 5pt height 3 pt depth -2.6pt \kern -6pt
\intop}\nolimits_{#1}} {\mathop{\vrule width 5pt height 3 pt depth
-2.6pt \kern -6pt \intop}\nolimits_{#1}}}

\newcommand{\intavl}[1]{\mathchoice {\mathop{\vrule width 6pt height 3 pt depth  -2.5pt
\kern -8pt \intop}\limits_{\kern -6pt#1}} {\mathop{\vrule width 5pt
height 3  pt depth -2.6pt \kern -6pt \intop}\nolimits_{#1}}
{\mathop{\vrule width 5pt height 3 pt depth -2.6pt \kern -6pt
\intop}\nolimits_{#1}} {\mathop{\vrule width 5pt height 3 pt depth
-2.6pt \kern -6pt \intop}\nolimits_{#1}}}

\newcommand{\un}{\underline}

\newcommand{\ve}{\varepsilon}

\newcommand{\R}{\mathbb{R}}
\newcommand{\N}{\mathbb{N}}

\newcommand{\Z}{\mathbb{Z}}

\renewcommand{\exp}[1]{{\rm exp}_{#1}}

\newcommand{\Hol}[1]{{\rm Hol}_{#1}}

\newcommand{\Sas}{d_{\rm Sas}}
\newcommand{\inj}{{\rm inj}}

\newcommand{\colvec}[2][.8]{%
  \scalebox{#1}{%
    \renewcommand{\arraystretch}{.8}%
    $\begin{bmatrix}#2\end{bmatrix}$%
  }
}

\begin{document}

\title[Symbolic dynamics for surface maps with discontinuities]{Symbolic dynamics for non-uniformly \\
hyperbolic surface maps with discontinuities \\ (Dynamique symbolique pour les syst\`emes \\ non-uniform\'ement hyperboliques avec discontinuit\'es)}

\author{Yuri Lima and Carlos Matheus}
\date{\today}
\keywords{Billiards, Markov partition, symbolic dynamics. \\ {\color{white}.}\hspace{0.2cm} \emph{Mots-cl\'es.} Billards, Partition de Markov, dynamique symbolique.}

\address{Laboratoire de Math\'ematiques d'Orsay, Universit\'e Paris-Sud\\ CNRS, Universit\'e
Paris-Saclay, $91405$ Orsay, France}
\email{yurilima@gmail.com}
\address{Universit\'e Paris 13\\ Sorbonne Paris Cit\'e\\ CNRS (UMR 7539), F-93430,
Villetaneuse, France}
\email{matheus.cmss@gmail.com}

\begin{abstract}
This work constructs symbolic dynamics for non-uniformly
hyperbolic surface maps with a set of discontinuities $\mathfs D$. We allow 
the derivative of points nearby $\mathfs D$ to be unbounded,
of the order of a negative power of the distance to $\mathfs D$.
Under natural geometrical assumptions on the underlying space $M$,
we code a set of non-uniformly hyperbolic orbits that do not converge
exponentially fast to $\mathfs D$. The results apply to non-uniformly hyperbolic
planar billiards, e.g. Bunimovich billiards.
\end{abstract}

\maketitle

\vspace{-0.9cm}

\selectlanguage{francais}
\begin{abstract} Nous construisons dynamique symbolique pour les applications non-uniform\'ement hyperboliques d'une surface ayant un ensemble de discontinuit\'es $\mathfs D$. La d\'eriv\'ee de l'application peut ne pas \^etre borne\'e, de l'ordre d'une puissance negative de la distance \`a $\mathfs D$. Sous certaines conditions g\'eom\'etriques naturelles sur l'espace des phases $M$, nous codifions un ensemble d'orbites non-uniform\'ement hyperboliques qui ne s'approchent pas exponentiellement vite de $\mathfs D$. Notre r\'esultat s'applique aux billards planaires non-uniform\'ement hyperboliques tels que les billards de Bunimovich.
\end{abstract}
\selectlanguage{english}

\tableofcontents

\section{Introduction}\label{Section-introduction}

Given a compact domain $T\subset\R^2$ with piecewise smooth boundary,
consider the straight line motion of a particle inside $T$, with specular reflections
in $\partial T$. Let $f:M\to M$ be the {\em billiard map}, where
$M=\partial T\times[-\tfrac{\pi}{2},\tfrac{\pi}{2}]$ with the convention that $(r,\theta)\in M$
represents $r=$ collision position at $\partial T$ and $\theta=$ angle of collision.
The map $f$ has a natural invariant Liouville measure $d\mu=\cos\theta drd\theta$.
Sina{\u\i} proved that dispersing billiards are uniformly hyperbolic systems with
discontinuities \cite{Sinai-billiards}, hence the Liouville measure is ergodic.

\medskip
For a while uniform hyperbolicity was the only mechanism to generate chaotic billiards,
until Bunimovich constructed examples of ergodic nowhere dispersing
billiards \cite{Bunimovich-close-to-scattering,Bunimovich-ergodic-properties,Bunimovich-Nowhere-dispersing}.
These billiards, known as {\em Bunimovich billiards}, are non-uniformly hyperbolic:
$\mu$--almost every point has one positive Lyapunov exponent and one negative Lyapunov exponent,
see \cite[Chapter 8]{Chernov-Markarian}.
In this paper we construct symbolic models for non-uniformly hyperbolic billiard maps
such as Bunimovich billiards. Assume that the billiard table $T$ satisfies
the conditions of \cite[Part V]{Katok-Strelcyn}, and
let $h$ be the Kolmogorov-Sina{\u\i} entropy of $\mu$.

\begin{theorem}\label{Thm-billiard}
If $\mu$ is ergodic and $h>0$ then there exists a topological Markov shift
$(\Sigma,\sigma)$ and a H\"older continuous map $\pi:\Sigma\to M$ s.t.:
\begin{enumerate}[$(1)$]
\item $\pi\circ \sigma=f\circ\pi$.
\item $\pi$ is surjective and finite-to-one on a set of full $\mu$--measure.
\end{enumerate}
\end{theorem}

Other examples of non-uniformly hyperbolic billiard maps are
\cite{Wojtkowski-principles,Bunimovich-lemons}. See section \ref{Section-preliminaries} for
the definition of topological Markov shifts.
%Let $\N=\{1,2,\ldots\}$.

\begin{corollary}\label{corollary-periodic}
Under the above assumptions, $\exists C>0$ and $p\geq 1$ s.t. $f$ has at least
$Ce^{hnp}$ periodic points of period $np$ for all $n\geq 1$. 
\end{corollary}

Corollary \ref{corollary-periodic} is consequence of Theorem \ref{Thm-billiard}
and the work of Gurevi{\v{c}} \cite{Gurevich-Topological-Entropy,Gurevich-Measures-Of-Maximal-Entropy},
as in \cite[Thm. 1.1]{Sarig-JAMS}. It is related to an estimate of Chernov \cite{Chernov-91}.
The integer $p$ is the period of $(\Sigma,\sigma)$, hence $p=1$ iff $(\Sigma,\sigma)$
is topologically mixing. Since $\mu$ is mixing, we expect that the symbolic
coding of Theorem \ref{Thm-billiard}
can be improved to give a topologically mixing $(\Sigma,\sigma)$.
Theorem \ref{Thm-billiard} is consequence of the main result
of this paper, Theorem \ref{Thm-main}, and of an argument of Katok and Strelcyn
\cite[Section I.3]{Katok-Strelcyn}.
The statement of Theorem \ref{Thm-main} is technical, so we first introduce some notation. 
%Theorem \ref{Thm-main} constructs symbolic dynamics
%for surface maps with discontinuities, with respect to hyperbolic measures.
%Its statement
%Given a discrete dynamical system, possibly with singularities, we construct symbolic
%models for a relevant part of the dynamics. More specifically, for a fixed $\chi>0$ we
%code a set of orbits that have a non-uniformly hyperbolic of order at least $\chi$ and
%that do not converge exponentially fast to the set of discontinuities of the system.
%The main application is for non-uniformly hyperbolic billiards, such as the Bunimovich stadium.
%For this particular billiard, the Liouville measure is ergodic and non-uniformly hyperbolic.
%An argument by Katok and Strelcyn implies that almost every point does not converge exponentially
%fast to the set of discontinuities. Hence our main result provides a symbolic model that
%covers almost every point with respect to the Liouville measure.
%For now we restrict ourselves to surface maps,
%since the statements are easier to explain and the tools are already developed in \cite{Sarig-JAMS}.

\medskip
Let $M$ be a smooth Riemannian surface with finite diameter, possibly with boundary.
We assume that the diameter of $M$ is smaller than one\footnote{Just multiply the
metric by a sufficiently small constant.}.
%(this allows us to state the assumptions without multiplicative constants).
Let $\mathfs D^+,\mathfs D^-$ be closed
subsets of $M$.
%Then $M\backslash\mathfs D^+$ and $M\backslash\mathfs D^-$ are open subsets of $M$. 
Fix $f:M\backslash\mathfs D^+\to M$ a diffeomorphism
onto its image, s.t. $f$ has an inverse $f^{-1}:M\backslash\mathfs D^-\to M$ that
is a diffeomorphism onto its image.
%also of class $C^{1+\beta}$.

\medskip
\noindent
{\sc Set of discontinuities $\mathfs D$:} The {\em set of discontinuities of $f$} is
$\mathfs D:=\mathfs D^+\cup \mathfs D^-$. 

\medskip
If $x\not\in\bigcup_{n\in\Z}f^n(\mathfs D)$ then $f^n(x)$ is well-defined for all $n\in\Z$,
and for every $y=f^n(x)$ there is a neighborhood $U\ni y$ s.t. $f\restriction_U,f^{-1}\restriction_{U}$
are diffeomorphisms onto their images. We require some regularity conditions on $M,f$.
The first four assumptions are on the geometry of $M$.
Given $x\in M\backslash\mathfs D$, let $\inj(x)$ denote the {injectivity radius} of $M$ at $x$,
and let $\exp{x}$ be the {\em exponential map} at $x$, wherever it can be defined.
Given $r>0$, let $B_x[r]\subset T_xM$ be the ball with center 0 and radius $r$.
%Write $\rho(x):= d(x,\mathfs D)$.
The Riemannian metric on $M$ induces a Riemannian metric on $TM$, called the {\em Sasaki
metric}, see e.g. \cite[\S2]{Burns-Masur-Wilkinson}. %\marginpar{Sasaki metric on appendix}
Denote the Sasaki metric by $\Sas(\cdot,\cdot)$.
Similarly, we denote the Sasaki metric on $TB_x[r]$ by the same notation, and the context
will be clear in which space we are. For nearby small vectors, the Sasaki metric is
almost a product metric in the following sense. Given a geodesic $\gamma$ joining $y$ to $x$,
let $P_\gamma:T_yM\to T_xM$ be the parallel transport along $\gamma$.
If $v\in T_xM$, $w\in T_yM$ then
$\Sas(v,w)\asymp d(x,y)+\|v-P_\gamma w\|$ as $\Sas(v,w)\to 0$, see e.g.
\cite[Appendix A]{Burns-Masur-Wilkinson}. The rate of convergence depends on the
curvature tensor of the metric on $M$. Here are the first two assumptions on $M$.

\medskip
\noindent
{\sc Regularity of $\exp{x}$:} $\exists a>1$ s.t. for all
$x\in M\backslash\mathfs D$ there is $d(x,\mathfs D)^a<\mathfrak r(x)<1$ 
s.t. for $D_x:=B(x,2\mathfrak r(x))$ the following holds:
\begin{enumerate}[ii]
\item[(A1)] If $y\in D_x$ then $\inj(y)\geq 2\mathfrak r(x)$, $\exp{y}^{-1}:D_x\to T_yM$
is a diffeomorphism onto its image, and
$\tfrac{1}{2}(d(x,y)+\|v-P_{y,x}w\|)\leq \Sas(v,w)\leq 2(d(x,y)+\|v-P_{y,x} w\|)$ for all $y\in D_x$ and
$v\in T_xM,w\in T_yM$ s.t. $\|v\|,\|w\|\leq 2\mathfrak r(x)$, where 	
$P_{y,x}:=P_\gamma$ is the radial geodesic $\gamma$ joining $y$ to $x$.
\item[(A2)] If $y_1,y_2\in D_x$ then
$d(\exp{y_1}v_1,\exp{y_2}v_2)\leq 2\Sas(v_1,v_2)$ for $\|v_1\|$, $\|v_2\|\leq 2\mathfrak r(x)$,
and $\Sas(\exp{y_1}^{-1}z_1,\exp{y_2}^{-1}z_2)\leq 2[d(y_1,y_2)+d(z_1,z_2)]$
for $z_1,z_2\in D_x$ whenever the expression makes sense.
In particular $\|d(\exp{x})_v\|\leq 2$ for $\|v\|\leq 2\mathfrak r(x)$,
and $\|d(\exp{x}^{-1})_y\|\leq 2$ for $y\in D_x$.
\end{enumerate}

\medskip
The next two assumptions are on the regularity of $d\exp{x}$.
For $x,x'\in\ M\backslash\mathfs D$, let $\mathfs L _{x,x'}:=\{A:T_xM\to T_{x'}M:A\text{ is linear}\}$
and $\mathfs L _x:=\mathfs L_{x,x}$. 
Then the parallel transport $P_{y,x}$ considered in (A1) is in $\mathfs L_{y,x}$.
Given $y\in D_x,z\in D_{x'}$ and $A\in \mathfs L_{y,z}$,
let $\widetilde{A}\in\mathfs L_{x,x'}$, $\widetilde{A}:=P_{z,x'} \circ A\circ P_{x,y}$.
By definition, $\widetilde{A}$ depends on $x,x'$ but different basepoints define
a map that differs from $\widetilde{A}$ by pre and post composition with isometries.
In particular, $\|\widetilde{A}\|$ does not depend on the choice of $x,x'$.
Similarly, if $A_i\in\mathfs L_{y_i,z_i}$ then $\|\widetilde{A_1}-\widetilde{A_2}\|$ does
not depend on the choice of $x,x'$.
Define the map $\tau=\tau_x:D_x\times D_x\to \mathfs L_x$
by $\tau(y,z)=\widetilde{d(\exp{y}^{-1})_z}$, where we use the identification
$T_v(T_{y}M)\cong T_{y}M$ for all $v\in T_yM$.

%This notion depends on the basepoints $x,x'$, but different basepoints define another linear map
%$y_1,y_2\in D_x$ and $z_1,z_2\in D_{x'}$. For a linear map
%$A_i:T_{y_i}M\to T_{z_i}M$, let $\widetilde{A_i}:T_xM\to T_{x'}M$,
%$\widetilde{A_i}:=P_{z_i,x'} \circ A_i\circ P_{x,y_i}$. Hence, for every two such maps $A_1,A_2$
%the expression $\widetilde{A_1}-\widetilde{A_2}$ is a linear operator. Formally this definition depends
%on $x,x'$, but different definitions differ from one another by pre and post-composition with
%isometries. In particular, $\|\widetilde{A_1}-\widetilde{A_2}\|$ does not depend on the choice of $x,x'$.

\medskip
\noindent
{\sc Regularity of $d\exp{x}$:}
\begin{enumerate}[ii]
\item[(A3)] If $y_1,y_2\in D_x$ then
$
\|\widetilde{d(\exp{y_1})_{v_1}}-\widetilde{d(\exp{y_2})_{v_2}}\|
\leq d(x,\mathfs D)^{-a}\Sas(v_1,v_2)
$
for all $\|v_1\|,\|v_2\|\leq 2\mathfrak r(x)$, and 
$\|\tau(y_1,z_1)-\tau(y_2,z_2)\|\leq d(x,\mathfs D)^{-a}[d(y_1,y_2)+d(z_1,z_2)]$
for all $z_1,z_2\in D_x$.
%$
%\|\widetilde{d(\exp{y_1}^{-1})_{z_1}}-\widetilde{d(\exp{y_2}^{-1})_{z_2}}\|
%\leq d(x,\mathfs D)^{-a}[d(y_1,y_2)+d(z_1,z_2)]$ for all $z_1,z_2\in D_x$.
\item[(A4)] If $y_1,y_2\in D_x$ then the map $\tau(y_1,\cdot)-\tau(y_2,\cdot):D_x\to \mathfs L_x$
has Lipschitz constant $\leq d(x,\mathfs D)^{-a}d(y_1,y_2)$.
\end{enumerate}

\medskip
Conditions (A1)--(A2) guarantee that the exponential maps and their inverses
are well-defined and have uniformly bounded Lipschitz constants in balls
of radii $d(x,\mathfs D)^a$.
Condition (A3) controls the Lipschitz constants of the derivatives of these maps,
and condition (A4) controls the Lipschitz constants of their second derivatives.
Here are some case when (A1)--(A4) are satisfied, in increasing order of generality:
\begin{enumerate}[$\circ$]
\item The curvature tensor $R$ of $M$ is globally bounded, e.g. when $M$ is the
phase space of a billiard map.
\item $R,\nabla R,\nabla^2 R,\nabla^3R$ grow at most polynomially
fast with respect to the distance to $\mathfs D$, e.g. when $M$ is a moduli space
of curves equipped with the Weil-Petersson metric \cite{Burns-Masur-Wilkinson}.
\end{enumerate}
%In (A3) the expression makes sense for $\exp{x}$ when $v,w\in TE_x$,
%and for $\exp{x}^{-1}$ when $v,w\in T[\exp{x}(E_x)]$. 
%Above, the $C^2$ norm of $\exp{x}$ is taken in $B_x[2\mathfrak r(x)]$ and the
%$C^2$ norm of $\exp{x}^{-1}$ in $\exp{x}(B_x[2\mathfrak r(x)])$.
%The $C^2$ norm is the norm defined by the strong Whitney topology.
Now we discuss the assumptions on $f$.

\medskip
\noindent
{\sc Regularity of $f$:} There are constants $0<\beta<1<b$ s.t. for all  $x\in M\backslash\mathfs D$:
\begin{enumerate}[ii]
\item[(A5)] If $y\in D_x$ then $\|df_y^{\pm 1}\|\leq d(x,\mathfs D)^{-b}$.
\item[(A6)] If $y_1,y_2\in D_x$ and $f(y_1),f(y_2)\in D_{x'}$ then
$\|\widetilde{df_{y_1}}-\widetilde{df_{y_2}}\|\leq \mathfrak Kd(y_1,y_2)^\beta$,
and if $y_1,y_2\in D_x$ and $f^{-1}(y_1),f^{-1}(y_2)\in D_{x''}$ then
$\|\widetilde{df_{y_1}^{-1}}-\widetilde{df_{y_2}^{-1}}\|\leq \mathfrak Kd(y_1,y_2)^\beta$.
%$\Hol{\beta}(df\restriction_{B(x, d(x,\mathfs D)/2)})<\mathfrak  d(x,\mathfs D)^{-b}$.
\end{enumerate}

\medskip
%Above $\Hol{\beta}$ stands for the $\beta$--Holder constant, see section \ref{Section-preliminaries}.
Although technical, conditions (A5)--(A6) hold in most cases of interest, e.g.
if $\|df^{\pm 1}\|,\|d^2f^{\pm 1}\|$ grow at most polynomially fast with respect to
the distance to $\mathfs D$. %\marginpar{\textcolor{red}{Check}}
%\footnote{If $\|d^2f_x^{\pm 1}\|< d(x,\mathfs D)^{-b}$, then for $v\in T_yM,w\in T_zM$
%s.t. $ d(y,x), d(z,x)< d(x,\mathfs D)^{b\beta}$
%it holds $\Sas{df^{\pm 1}v}{df^{\pm 1}w}\leq  d(x,\mathfs D)^{-b}$}.
We finally define the measures we code. Fix $\chi>0$.

\medskip
\noindent
{\sc $\chi$--hyperbolic measure:} An $f$--invariant probability measure on $M$ is called
{\em $\chi$--hyperbolic} if $\mu$--a.e. $x\in M$ has one Lyapunov exponent $>\chi$
and another $<-\chi$.

\medskip
\noindent
{\sc $f$--adapted measure:} An $f$--invariant measure on $M$ is called {\em $f$--adapted}
if
$$\int_M \log d(x,\mathfs D)d\mu(x)>-\infty.$$
A fortiori $\mu(\mathfs D)=0$.

%\medskip
%\noindent
%{\sc The set ${\rm NUH}_\chi^\dag$:} It is the set ${\rm NUH}_\chi^\dag:={\rm NUH}_\chi^*\cap\mathfs S$.

\begin{theorem}\label{Thm-main}
Let $M,f$ satisfy conditions {\rm (A1)--(A6)}. For all $\chi>0$,
there exists a topological Markov shift $(\Sigma,\sigma)$
and a H\"older continuous map $\pi:\Sigma\to M$ s.t.:
\begin{enumerate}[$(1)$]
\item $\pi\circ \sigma=f\circ\pi$.
\item  $\pi[\Sigma^\#]$ has full $\mu$--measure for every $f$--adapted $\chi$--hyperbolic measure $\mu$.
\item For all $x\in \pi[\Sigma^\#]$, $\#\{\un v\in\Sigma^\#:\pi(\un v)=x\}<\infty$.
\end{enumerate}
\end{theorem}

Above, $\Sigma^\#$ is the {\em recurrent set} of $\Sigma$, see section \ref{Section-preliminaries}.
Every $\sigma$--invariant measure $\widehat{\mu}$ is carried by $\Sigma^\#$,
hence its projection $\mu=\widehat{\mu}\circ\pi^{-1}$ has the same entropy
as $\widehat{\mu}$ (this follows from the Abramov-Rokhlin formula \cite{Abramov-Rokhlin}).
In particular, the topological entropy of $(\Sigma,\sigma)$ is at most that of $(M,f)$.
On the other direction, every $f$--adapted $\chi$--hyperbolic measure $\mu$ has
a lift $\widehat{\mu}$ with the same entropy. If we know that $\chi$--hyperbolic measures are
$f$--adapted then the topological entropies of $(\Sigma,\sigma)$ and $(M,f)$ coincide,
and their measures of maximal entropy are related. In this case, Corollary \ref{corollary-periodic}
has a potentially stronger statement: for every $\ve>0$, $\exists C>0$ and $p\geq 1$ s.t. $f$ has at
least $Ce^{(H-\ve)np}$ periodic points of period $np$ for all $n\geq 1$,
where $H$ is the topological entropy of $\Sigma$.
At the moment, we are not aware of general results assuring that $\chi$--hyperbolic measures
are $f$--adapted, except when the measure is Liouville \cite[Section I.3]{Katok-Strelcyn}.

\medskip
We now discuss the applicability of Theorem \ref{Thm-billiard}. Let us restrict ourselves to
billiard tables with finitely many boundary components, otherwise many degeneracies
can occur (see e.g. \cite[Part V]{Katok-Strelcyn}). Assumptions (A1)--(A6) are satisfied
if all boundary components are $C^3$. The precise conditions that guarantee 
non-uniform hyperbolicity are unknown, so we mention two classes of billiard tables $T$
whose billiard maps are non-uniformly hyperbolic:
\begin{enumerate}[$\circ$]
\item Sina{\u\i} billiard: every component of $\partial T$ is dispersing.
In this case, the billiard map exhibits uniform hyperbolicity.
\item Bunimovich billiard: $\partial T$ is the union of finitely many segments
and arcs of circles s.t. each of these arcs belongs to a disc contained in $T$.
When this happens, non-uniform hyperbolicity is ensured via a focusing-defocusing mechanism,
see \cite[Chapter 8]{Chernov-Markarian}. See Figure \ref{figure-billiards} for some examples.
\end{enumerate}

\begin{figure}[hbt!]
\centering
\def\svgwidth{12cm}
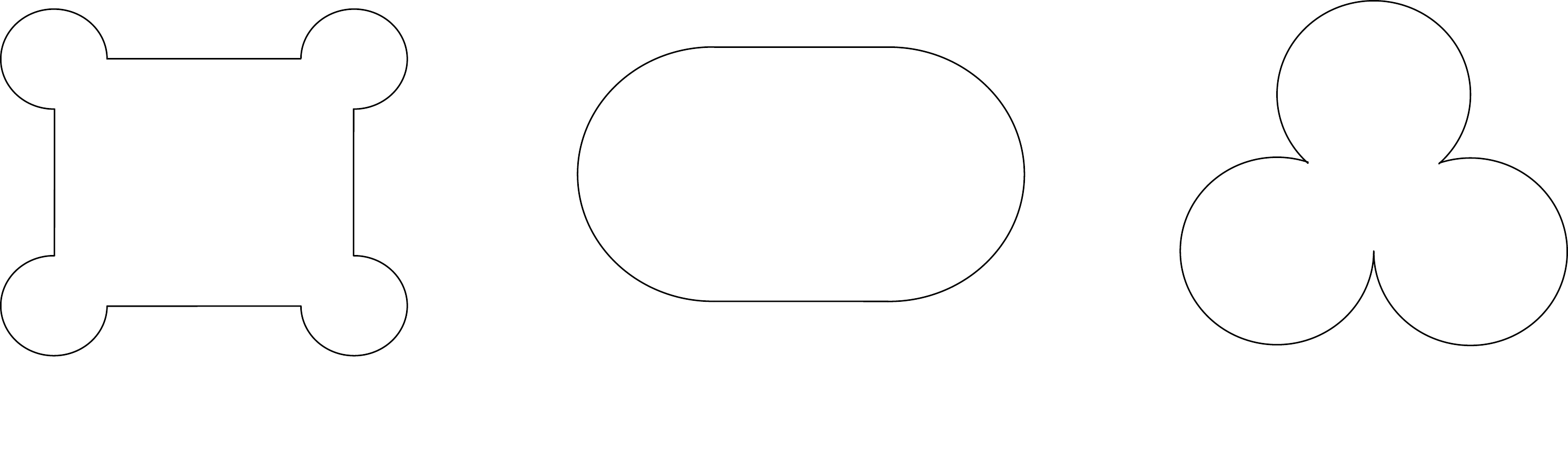\label{figure-billiards}\caption{Examples of Bunimovich billiards:
(a) pool table with pockets, (b) stadium, (c) flower.}
\end{figure}

\subsection{Related literature}

The construction of Markov partitions and symbolic dynamics for uniformly hyperbolic
diffeomorphisms and flows in compact manifolds laid its foundation during the late sixties and
early seventies through the works of Adler \& Weiss
\cite{Adler-Weiss-PNAS,Adler-Weiss-Similarity-Toral-Automorphisms},
Sina{\u\i} \cite{Sinai-Construction-of-MP,Sinai-MP-U-diffeomorphisms},
Bowen \cite{Bowen-MP-Axiom-A,Bowen-Symbolic-Flows}, and
Ratner \cite{Ratner-MP-three-dimensions,Ratner-MP-n-dimensions}.
Below we discuss other contexts.

%Recently there has been
%a growing interest in the dynamics of the Weil-Petersson flow, which is an incomplete flow and hence
%has similar properties to maps with discontinuities. 

\medskip
\noindent
{\sc Billiards:} These are the main examples of maps with discontinuities.
Katok and Strelcyn constructed invariant manifolds for non-uniformly hyperbolic
billiard maps which include Bunimovich billiards \cite{Katok-Strelcyn}.
Bunimovich, Chernov and Sina{\u\i} constructed countable Markov partitions for two-dimensional
dispersing billiard maps \cite{Bunimovich-Chernov-Sinai}.
%Kr\"uger and Troubetzkoy constructed
%countable Markov partitions for non-uniformly hyperbolic billiard maps
%which include Bunimovich billiards \cite{Kruger-Troubetzkoy}.
All these results are for Liouville measures. Up to our knowledge,
our result is the first symbolic coding of uniformly and non-uniformly hyperbolic
billiard maps for general measures.

\medskip
\noindent
{\sc Tower extensions of billiard maps:} Young constructed tower extensions
%(representations of maps as factors of infinite-to-one extension of other maps) 
for certain two-dimensional dispersing billiard maps \cite{Young-towers}.
Contrary to our case, Young's tower extensions provide codings which are usually infinite-to-one,
hence it is unclear that $\chi$--hyperbolic measures can be lifted to the
symbolic space without increasing its entropy. Nevertheless, such tower extensions
guarantee exponential decay of correlations for certain two-dimensional dispersing billiard maps.	  

\medskip
\noindent
{\sc Non-uniformly hyperbolic three-dimensional flows:} The first author and Sarig
constructed symbolic models for non-uniformly hyperbolic three-dimensional flows with positive
speed \cite{Lima-Sarig}. The idea is to take a Poincar\'e section and analyze the Poincar\'e return map $f$.
The Poincar\'e map $f$ has discontinuities, but its derivative is uniformly bounded inside the set
of continuities. Hence the methods of \cite{Sarig-JAMS} apply more easily.

\medskip
\noindent
{\sc Weil-Petersson flow:} Moduli spaces of curves possess natural negatively
curved incomplete K\"ahler metrics, called {\em Weil-Petersson metrics}. The geodesic
flow of one such metric is called the {\em Weil-Petersson flow}, and it preserves a canonical
Liouville measure.
The properties of the Weil-Petersson metric are intimately related to the hyperbolic
geometry of surfaces, and this partly explains the recent interest in the dynamics
of the Weil-Petersson flow. Burns, Masur and Wilkinson proved that the Liouville
measure is hyperbolic \cite{Burns-Masur-Wilkinson}.
For that, they combined results of Wolpert and McMullen to show that the
Weil-Petersson metric explodes at most polynomially fast while approaching
the boundary of the Deligne-Mumford compactification of the moduli space of curves,
hence the Weil-Petersson flow satisfies the assumptions of Katok and Strelcyn \cite{Katok-Strelcyn}.
The construction of symbolic dynamics for the Weil-Petersson flow is still open.

\medskip
As pointed out by Sarig \cite[pp. 346]{Sarig-JAMS}, our main result (Theorem \ref{Thm-main})
can be regarded as a step towards the construction of Markov partitions capturing
measures of maximal entropy for surface maps with discontinuities with positive topological
entropy, such as Bunimovich billiards.
Motivated by this, we ask the following question.

\smallskip

\noindent
{\sc Question:} Let $f$ be a billiard map with topological entropy $H>0$. Does $f$ have a measure
of maximal entropy? If it does, is it $f$--adapted? Is it Bernoulli? 

\medskip
A positive answer to this question would imply that $\exists C>0$ s.t. $f$ has at least 
$Ce^{Hn}$ periodic points of period $n$, for all $n\geq 1$.

\medskip
In \cite[pp. 858]{Burns-Masur-Wilkinson} it was suggested
that one of the assumptions (in their notation, the compactness of $\overline{N}$)
can be relaxed to the assumption that $N$ has finite diameter.
The main reason not to claim this is that they use \cite{Katok-Strelcyn}, whose framework
assumes $\overline N$ to be compact. We only assume finite diameter, hence
our work is a step towards the relaxation of the assumptions of \cite{Katok-Strelcyn}
to the context mentioned in \cite{Burns-Masur-Wilkinson}.

\subsection{Methodology}

The proof of Theorem \ref{Thm-main} is based on \cite{Sarig-JAMS} and \cite{Lima-Sarig},
and it follows the steps below:
\begin{enumerate}[(1)]
\item If $\mu$ is $f$--adapted and $\chi$--hyperbolic, then $\mu$--a.e. $x\in M$
has a Pesin chart $\Psi_x:[-Q_\ve(x),Q_\ve(x)]^2\to M$ s.t.
$\lim_{n\to\infty}\tfrac{1}{n}\log Q_\ve(f^n(x))=0$.
%i.e. the domains of definition of the charts do not decrease exponentially fast along most orbits.
%(Oseledets' theorem and assumptions (A2)--(A3)).
\item Define $\ve$--double charts $\Psi_x^{p^s,p^u}$, the two-sided versions of Pesin charts
that control separately the local forward and local backward hyperbolicity at $x$.
\item Construct a countable collection $\mathfs A$ of $\ve$--double charts that are dense
in the space of all $\ve$--double charts. The notion of denseness is defined in terms of
finitely many parameters of $x$.
\item Define the transition between $\ve$--double charts s.t. $p^s,p^u$
are as maximal as possible. This is important to establish the inverse theorem (Theorem \ref{Thm-inverse}). 
\item Apply a Bowen-Sina{\u\i} refinement (following \cite{Bowen-LNM}).
The resulting partition defines a topological Markov shift $(\Sigma,\sigma)$ and a map
$\pi:\Sigma\to M$ satisfying Theorem \ref{Thm-main}.
\end{enumerate}

\medskip
Contrary to \cite{Sarig-JAMS,Lima-Sarig}, we do not require $M$ to be compact (not even to have bounded curvature)
neither $f$ to have uniformly bounded $C^{1+\beta}$ norm. As a consequence,
we have to control the parameters appearing in the construction more carefully.
In the methodology of proof above,
this is reflected in steps (1), (3), (4). Steps (2) and (5) work almost verbatim as in \cite{Sarig-JAMS}.

\subsection{Preliminaries}\label{Section-preliminaries}

%We now define the main objects in the statements of Theorems \ref{Thm-billiard} and \ref{Thm-main},
%and the notation that will be use throughout this text.
Let $\mathfs G=(V,E)$ be an oriented graph, where $V=$ vertex set and $E=$ edge set.
We denote edges by $v\to w$, and we assume that $V$ is countable.

\medskip
\noindent
{\sc Topological Markov shift (TMS):} A {\em topological Markov shift} (TMS) is a pair $(\Sigma,\sigma)$
where
$$
\Sigma:=\{\text{$\Z$--indexed paths on $\mathfs G$}\}=
\left\{\un{v}=\{v_n\}_{n\in\Z}\in V^{\Z}:v_n\to v_{n+1}, \forall n\in\Z\right\}
$$
%is the space of $\Z$--indexed paths on $\mathfs G$,
and $\sigma:\Sigma\to\Sigma$ is the left shift, $[\sigma(\un v)]_n=v_{n+1}$. 
The {\em recurrent set} of $\Sigma$ is
$$
\Sigma^\#:=\left\{\un v\in\Sigma:\exists v,w\in V\text{ s.t. }\begin{array}{l}v_n=v\text{ for infinitely many }n>0\\
v_n=w\text{ for infinitely many }n<0
\end{array}\right\}.
$$
We endow $\Sigma$ with the distance $d(\un v,\un w):={\rm exp}[-\min\{|n|\in\Z:v_n\neq w_n\}]$.

\medskip
Write $a=e^{\pm\ve}b$ when $e^{-\ve}\leq \frac{a}{b}\leq e^\ve$,
and $a=\pm b$ when $-|b|\leq a\leq |b|$. Given an open set $U\subset \R^n$ and $h:U\to \R^m$,
let $\|h\|_0:=\sup_{x\in U}\|h(x)\|$ denote the $C^0$ norm of $h$. For $0<\beta<1$,
let $\Hol{\beta}(h):=\sup\frac{\|h(x)-h(y)\|}{\|x-y\|^\beta}$ 
where the supremum ranges over distinct elements $x,y\in U$.
If $h$ is differentiable, let
$\|h\|_1:=\|h\|_0+\|dh\|_0$ denote its $C^1$ norm, and
$\|h\|_{1+\beta}:=\|h\|_{C^1}+\Hol{\beta}(dh)$ its $C^{1+\beta}$ norm.
%If $h$ is $C^2$, its $C^2$ norm is
%$$
%\|h\|_{C^2}:=\|h\|_{C^1}+\sup_{x\in U}\|d^2h_x\|=\|h\|_{C^0}+\|dh\|_{C^0}+\|d^2h\|_{C^0}.
%$$
Given $x\in M$, remember that $B_x[r]\subset T_xM$ is the ball with center $0\in T_xM$ and radius $r$.
Also define $R[r]:=[-r,r]^2\subset\R^2$.
%For $x\in M$, let $R_x[r]:=[-r,r]^2\subset T_xM$,
%i.e. $R_x[r]$ is defined with respect to the Riemannian metric.

\medskip
The diameter of $M$ is less than one, hence we can assume that $a=b$: just change
$a,b$ to $\max\{a,b\}$. For symmetry and simplification purposes,
we will sometimes use (A3)--(A5) in the weaker forms below.
Define $\rho(x):=d(\{f^{-1}(x),x,f(x)\},\mathfs D)$, then (A3)--(A5) imply
that for all $x\in M\backslash\mathfs D$:
\begin{enumerate}[ii]
\item[(A3)'] If $y_1,y_2\in D_x$ then
$
\|\widetilde{d(\exp{y_1})_{v_1}}-\widetilde{d(\exp{y_2})_{v_2}}\|
\leq \rho(x)^{-a}\Sas(v_1,v_2)
$
for all $\|v_1\|,\|v_2\|\leq 2\mathfrak r(x)$,
and 
$\|\tau(y_1,z_1)-\tau(y_2,z_2)\|\leq d(x,\mathfs D)^{-a}[d(y_1,y_2)+d(z_1,z_2)]$
for all $z_1,z_2\in D_x$.
%and 
%$
%\|\widetilde{d(\exp{y_1}^{-1})_{z_1}}-\widetilde{d(\exp{y_2}^{-1})_{z_2}}\|
%\leq \rho(x)^{-a}[d(y_1,y_2)+d(z_1,z_2)]$
%for all $z_1,z_2\in D_x$.
\item[(A4)'] If $y_1,y_2\in D_x$ then the map $\tau(y_1,\cdot)-\tau(y_2,\cdot):D_x\to \mathfs L_x$
has Lipschitz constant $\leq \rho(x)^{-a}d(y_1,y_2)$.
\item[(A5)'] If $y\in D_x$ then $\|df_y^{\pm 1}\|\leq \rho(x)^{-a}$.
\end{enumerate}
Here is a consequence of (A5) and the inverse theorem, written in symmetric form:
\begin{enumerate}[(A7)]
\item[(A7)] $\|df^{\pm 1}_x\|\geq m(df^{\pm 1}_x)\geq \rho(x)^a$.
\end{enumerate}
Above, $m(A):=\|A^{-1}\|^{-1}$. For the ease of reference, we collect (A1)--(A7) in Appendix A
in the format we will use in the text.

\medskip
We note that $\mu$ is $f$--adapted iff $\int \log\rho(x)d\mu>-\infty$.
If $\mu$ is $f$--adapted then by $\mu$--invariance the functions
$-\log d(f^{-1}(x),\mathfs D),-\log d(x,\mathfs D),-\log d(f(x),\mathfs D)$ are in $L^1(\mu)$,
hence is also their maximum $-\log\rho(x)$. The reverse implication is proved similarly.

%By the inverse function theorem, condition (A3) is equivalent to the one below:
%\begin{enumerate}[(A6)]
%\item $\|df_x\|>\mathfrak L^{-1}\rho(f(x))^b$ and
%$\|df^{-1}_x\|>\mathfrak L^{-1}\rho(f^{-1}(x))^b$ for all $x\in M\backslash\mathfs D$.
%\end{enumerate}
%Note also that (A3) controls the co-norms of $df_x$ and $df_x^{-1}$:

\section{Linear Pesin theory}

In this section we construct changes of coordinates that make $df$ a hyperbolic matrix.
Since we are dealing with the action of the derivative only, the closeness of $x$ to $\mathfs D$
is irrelevant.

\medskip
Fix $\chi>0$, and let ${\rm NUH}_\chi$ be the set
of $x\in M\backslash \bigcup_{n\in\Z}f^n(\mathfs D)$ for which
there are vectors $\{e^s_{f^n(x)}\}_{n\in\Z}$, $\{e^u_{f^n(x)}\}_{n\in\Z}$ s.t. for every
$y=f^n(x)$, $n\in\Z$, it holds:
\begin{enumerate}[(1)]
\item $e^{s/u}_y\in T_yM$, $\|e^{s/u}_y\|=1$.
\item ${\rm span}(df^m_ye^{s/u}_y)={\rm span}(e^{s/u}_{f^m(y)})$ for all $m\in\Z$.
\item $\lim_{m\to\pm\infty}\tfrac{1}{m}\log\|df^m_y e^s_y\|<-\chi$ and
$\lim_{m\to\pm\infty}\tfrac{1}{m}\log\|df^m_y e^u_y\|>\chi$.
\item $\lim_{m\to\pm\infty}\tfrac{1}{m}\log|\sin\alpha(f^m(y))|=0$, where
$\alpha(f^m(y))=\angle(e^s_{f^m(y)},e^u_{f^m(y)})$.
\end{enumerate}

\subsection{Oseledets-Pesin reduction}

We represent $df_x$ as a hyperbolic matrix.

\medskip
\noindent
{\sc Parameters $s(x),u(x)$:} For $x\in{\rm NUH}_\chi$, define
$$
s(x):=\sqrt{2}\left(\sum_{n\geq 0}e^{2n\chi}\|df^n_xe^s_x\|^2\right)^{1/2} \text{ and }
u(x):=\sqrt{2}\left(\sum_{n\geq 0}e^{2n\chi}\|df^{-n}_xe^u_x\|^2\right)^{1/2}.
$$

\medskip
These numbers are well-defined because $x\in{\rm NUH}_\chi$, and $s(x),u(x)\geq \sqrt{2}$.
Let $e_1=(1,0),e_2=(0,1)$ be the canonical basis of $\R^2$.

\medskip
\noindent
{\sc Linear map $C_\chi(x):$} For $x\in{\rm NUH}_\chi$, let
$C_\chi(x):\R^2\to T_xM$ be the linear map s.t.
$$
C_\chi(x):e_1\mapsto \frac{e^s_x}{s(x)}\ ,\ C_\chi(x): e_2\mapsto \frac{e^u_x}{u(x)}\cdot
$$

\medskip
Given a linear transformation, let $\|\cdot\|$ denote its sup norm and $\|\cdot\|_{\rm Frob}$ its
Frobenius norm\footnote{The Frobenius norm
of a $2\times 2$ matrix $A=\left[\begin{array}{cc}a & b\\ c & d\end{array}\right]$ is
$\|A\|_{\rm Frob}=\sqrt{a^2+b^2+c^2+d^2}$.}. The Frobenius norm is equivalent to the usual sup norm,
with $\|\cdot\|\leq \|\cdot\|_{\rm Frob}\leq \sqrt{2}\|\cdot\|$.
  
\begin{lemma}\label{Lemma-linear-reduction}
For all $x\in{\rm NUH}_\chi$, the following holds:
\begin{enumerate}[{\rm (1)}]
\item $\|C_\chi(x)\|\leq \|C_\chi(x)\|_{\rm Frob}\leq 1$
and $\|C_\chi(x)^{-1}\|_{\rm Frob}=\tfrac{\sqrt{s(x)^2+u(x)^2}}{|\sin\alpha(x)|}$.
\item $C_\chi(f(x))^{-1}\circ df_x\circ C_\chi(x)$ is a diagonal matrix with diagonal entries $A,B\in\R$
s.t. $|A|<e^{-\chi}$ and $|B|>e^\chi$.
\end{enumerate}
\end{lemma}

\begin{proof}
(a) In the basis $\{e_1,e_2\}$ of $\R^2$ and the basis $\{e^s_x,(e^s_x)^\perp\}$ of $T_xM$, $C_\chi(x)$ takes the
form $\left[\begin{array}{cc}\tfrac{1}{s(x)}& \tfrac{\cos\alpha(x)}{u(x)}\\ 0&\tfrac{\sin\alpha(x)}{u(x)}\end{array}\right]$,
hence $\|C_\chi(x)\|_{\rm Frob}^2=\tfrac{1}{s(x)^2}+\tfrac{1}{u(x)^2}\leq 1$. The inverse of
$C_\chi(x)$ is
$\left[\begin{array}{cc}s(x)& -\tfrac{s(x)\cos\alpha(x)}{\sin\alpha(x)}\\ 0&\tfrac{u(x)}{\sin\alpha(x)}\end{array}\right]$,
therefore $\|C_\chi(x)^{-1}\|_{\rm Frob}=\tfrac{\sqrt{s(x)^2+u(x)^2}}{|\sin\alpha(x)|}$.

\medskip
\noindent
(b) It is clear that $e_1,e_2$ are eigenvectors of $C_\chi(f(x))^{-1}\circ df_x\circ C_\chi(x)$.
We calculate the eigenvalue of $e_1$ (the calculation of the eigenvalue of $e_2$ is similar).
Since $df_xe^s_x=\pm\|df_xe^s_x\|e^s_{f(x)}$,
$[df_x\circ C_\chi(x)](e_1)=\pm df_x\left[\tfrac{e^s_x}{s(x)}\right]=\pm\tfrac{\|df_xe^s_x\|}{s(x)}e^s_{f(x)}$, hence
$[C_\chi(f(x))^{-1}\circ df_x\circ C_\chi(x)](e_1)=\pm\|df_xe^s_x\|\tfrac{s(f(x))}{s(x)}e_1$.
Thus $A:=\pm\|dfe^s_x\|\tfrac{s(f(x))}{s(x)}$ is the eigenvalue of $e_1$. Note that
\begin{align*}
s(f(x))^2&=\tfrac{2}{e^{2\chi}\|df_xe^s_x\|^2}\sum_{n\geq 1}e^{2n\chi}\|df^n_xe^s_x\|^2
=\tfrac{s(x)^2-2}{e^{2\chi}\|df_xe^s_x\|^2}<\tfrac{s(x)^2}{e^{2\chi}\|df_xe^s_x\|^2},
\end{align*}
therefore $|A|<e^{-\chi}$.
\end{proof}

\subsection{The set ${\rm NUH}_\chi^*$}

We need to control the exponential rate decay of the distance of
trajectories to the set of discontinuities $\mathfs D$.

\medskip
\noindent
{\sc Regular set:} We define the {\em regular set of $f$} by 
$$
{\rm Reg}:=\left\{x\in M\backslash\mathfs D:\lim_{n\to\pm\infty}
\tfrac{1}{|n|}\log\rho(f^n(x))=0\right\}.
$$
%$$
%\mathfs S:=\left\{x\in M\backslash\mathfs D:\liminf_{n\to\pm\infty}
%\tfrac{1}{|n|}\log\rho(f^n(x))<0\right\}
%$$
%and the {\em regular set of $f$} by

%\begin{remark}
%The term ``regular point'' is often referred to a point satisfying the conclusions of the
%Oseledets theorem. We do not use this terminology here. 
%\end{remark}

%\medskip
%Every $x\in\mathfs {NS}$ the orbit of $x$ is well-defined
%and does not converge exponentially fast to $\mathfs D$.
%For $x\in{\rm NUH}_\chi$, let $\alpha(x)=\angle(e^s_x,e^u_x)$.

\medskip
\noindent
{\sc The set ${\rm NUH}_\chi^*$:} It is the set of $x\in{\rm NUH}_\chi$ with the following properties:
\begin{enumerate}[(1)]
%\item $\lim_{n\to\pm\infty}\tfrac{1}{|n|}\log|\sin\alpha(f^n(x))|=0$.
\item $x\in{\rm Reg}$.
\item There exist sequences $n_k,m_k\to\infty$ s.t.
$C_\chi(f^{n_k}(x)),C_\chi(f^{-m_k}(x))\to C_\chi(x)$. 
\item $\lim_{n\to\pm\infty}\tfrac{1}{|n|}\log\|C_\chi(f^n(x))\|=0$.
\item $\lim_{n\to\pm\infty}\tfrac{1}{|n|}\log\|C_\chi(f^n(x))^{-1}\|=0$.
\end{enumerate}

\medskip
The next lemma shows that relevant measures are carried by ${\rm NUH}_\chi^*$.

%By  Oseledets' theorem and Ruelle's inequality, every ergodic $f$--invariant probability
%measure s.t. $\int \log^+\|df_x\|d\mu<\infty$, $\int \log^+\|df_x^{-1}\|d\mu<\infty$ and
%with metric entropy $>\chi$ is carried by ${\rm NUH}_\chi^*$.

\begin{lemma}\label{Lemma-adaptedness}
If $\mu$ is $f$--adapted and $\chi$--hyperbolic, then $\mu[{\rm NUH}_\chi^*]=1$.
\end{lemma}

%In higher dimension, condition $h_\mu(f)>\chi$ will be changed to the
%$\chi$--hyperbolicity of $\mu$ (the Lyapunov exponents are greater than $\chi$
%in absolute value almost everywhere).

\begin{proof}
%Let $\mu$ be a $f$--adapted probability measure on $M$ with $h_\mu(f)>\chi$.
By (A5) and the $f$--adaptedness of $\mu$, $\int \log^+\|df^{\pm 1}\|d\mu<\infty$
hence the Oseledets theorem applies to the cocycle $df^n$ and
measure $\mu$.
Since $\mu$ is $\chi$--hyperbolic, $\mu[{\rm NUH}_\chi]=1$.
By $f$--adaptedness and the Birkhoff ergodic theorem\footnote{Here we are using that if $\varphi:M\to\R$ satisfies
$\int |\varphi|d\mu<\infty$ then $\liminf_{n\to\pm\infty}\tfrac{1}{n}\varphi(f^n(x))=0$ $\mu$--a.e.
Indeed, by the Birkhoff theorem
$\widetilde\varphi(x)=\lim_{n\to\infty}\tfrac{1}{n}\sum_{i=0}^{n-1}\varphi(f^i(x))$ exists $\mu$--a.e., hence
$\lim_{n\to\infty}\tfrac{1}{n}\varphi(f^n(x))=\lim_{n\to\infty}\left[\tfrac{1}{n}\sum_{i=0}^{n}\varphi(f^i(x))
-\tfrac{1}{n}\sum_{i=0}^{n-1}\varphi(f^i(x))\right]=0$ $\mu$--a.e. The same argument works
for $n\to-\infty$.},
$\mu({\rm Reg})=1$.
By the Poincar\'e recurrence theorem, (2) holds $\mu$--a.e. It remains to check (3)--(4).

\medskip
For $x\in{\rm NUH}_\chi$, let $D_\chi(x):=C_\chi(f(x))^{-1}\circ df_x\circ C_\chi(x)$.
This defines a cocycle $D_\chi^{(n)}$ on ${\rm NUH}_\chi$.
We first show that we can apply the Oseledets theorem for $D_\chi^{(n)}$ and $\mu$.
By lemma \ref{Lemma-linear-reduction} and its proof,
$D_\chi(x)=\left[\begin{array}{cc}A(x) & 0 \\ 0 & B(x) \end{array}\right]$
where $A(x)^2=e^{-2\chi}\tfrac{s(x)^2-2}{s(x)^2}$ and $B(x)^2=e^{2\chi}\tfrac{u(f(x))^2}{u(f(x))^2-2}$.
We have $\|D_\chi(x)\|=|B(x)|$ and $\|D_\chi(x)^{-1}\|=|A(x)|^{-1}$,
therefore we wish to show that
\begin{align*}
\int \log |A(x)|d\mu(x)>-\infty\ \text{ and }\ \int \log |B(x)|d\mu(x)<\infty.
\end{align*}
We prove the first inequality (the second inequality is proved similarly). By (A6),
$s(x)^2\geq 2(1+e^{2\chi}\|df_xe^s_x\|^2)\geq 2(1+e^{2\chi}\rho(x)^{2a})$
hence
$$
A(x)^2=e^{-2\chi}\tfrac{s(x)^2-2}{s(x)^2}=e^{-2\chi}\left(1-\tfrac{2}{s(x)^2}\right)
\geq\tfrac{\rho(x)^{2a}}{1+e^{2\chi}\rho(x)^{2a}}\geq\tfrac{\rho(x)^{2a}}{1+e^{2\chi}}\,\cdot
$$
Therefore
$$
\int \log|A(x)|d\mu(x)\geq a\int\log \rho(x)d\mu(x)-\tfrac{1}{2}\log(1+e^{2\chi})>-\infty.
$$
By a similar reasoning, $\int \log |B(x)|d\mu(x)<\infty$.
Therefore we can apply the Oseledets theorem for $D_\chi^{(n)}$ and $\mu$:
there is an $f$--invariant set $X\subset{\rm NUH}_\chi$ with $\mu(X)=1$ s.t. every
$x\in X$ satisfies (2) and $\lim_{n\to\infty}\tfrac{1}{n}\log\|D_\chi^{(n)}(x)\|$ exists.
We claim that (3)--(4) hold in $X$.

\medskip
We first show that the Lyapunov exponents of $D_\chi^{(n)}$ and $df^n$ coincide in $X$.
Fix $x\in X$, and take $n_k\to\infty$ s.t. $C_\chi(f^{n_k}(x))\to C_\chi(x)$.
%By the Poincar\'e recurrence theorem, this set has full $\mu$--measure.
Since
$\|D_\chi^{(n)}(x)\|\leq \|C_\chi(f^{n}(x))^{-1}\|\|df^n_x\|\|C_\chi(x)\|\leq \|C_\chi(f^{n}(x))^{-1}\|\|df^n_x\|$,
\begin{align*}
&\lim_{n\to\infty}\tfrac{1}{n}\log\|D_\chi^{(n)}(x)\|=\limsup_{k\to\infty}\tfrac{1}{n_k}\log\|D_\chi^{(n_k)}(x)\|\\
&\leq \limsup_{k\to\infty}\tfrac{1}{n_k}\log\|C_\chi(f^{n_k}(x))^{-1}\|+\limsup_{k\to\infty}\tfrac{1}{n_k}\log\|df^{n_k}_x\|
=\lim_{n\to\infty}\tfrac{1}{n}\log\|df^n_x\|.
\end{align*}
Similarly, $\|df^n_x\|\leq \|C_\chi(f^n(x))\|\|D_\chi^{(n)}(x)\|\|C_\chi(x)^{-1}\|\leq\|D_\chi^{(n)}(x)\|\|C_\chi(x)^{-1}\|$,
thus
\begin{align*}
&\lim_{n\to\infty}\tfrac{1}{n}\log\|df^n_x\|=\limsup_{k\to\infty}\tfrac{1}{n_k}\log\|df^{n_k}_x\|
\leq \limsup_{k\to\infty}\tfrac{1}{n_k}\log\|D_\chi^{(n_k)}(x)\|\\
&=\lim_{n\to\infty}\tfrac{1}{n}\log\|D_\chi^{(n)}(x)\|.
\end{align*}
Hence $\lim_{n\to\infty}\tfrac{1}{n}\log\|D_\chi^{(n)}(x)\|=\lim_{n\to\infty}\tfrac{1}{n}\log\|df^n_x\|$.
Applying the same argument along the sequence $m_k\to\infty$ for which
$C_\chi(f^{-m_k}(x))\to C_\chi(x)$, we obtain
\begin{equation}\label{equality-spectra}
\lim_{n\to\pm\infty}\tfrac{1}{|n|}\log\|D_\chi^{(n)}(x)\|=\lim_{n\to\pm\infty}\tfrac{1}{|n|}\log\|df^n_x\|.
\end{equation}

\medskip
Since $\|C_\chi(\cdot)\|\leq 1$, $\limsup_{n\to\pm\infty}\tfrac{1}{|n|}\log\|C_\chi(f^n(x))\|\leq 0$. Reversely,
the inequality $\|df^n_x\|\leq \|C_\chi(f^n(x))\|\|D_\chi^{(n)}(x)\|\|C_\chi(x)^{-1}\|$ implies
$$
\liminf_{n\to\pm\infty}\tfrac{1}{|n|}\log\|C_\chi(f^n(x))\|\geq \lim_{n\to\pm\infty}\tfrac{1}{|n|}\log\|df^n_x\|-
\lim_{n\to\pm\infty}\tfrac{1}{|n|}\log\|D_\chi^{(n)}(x)\|=0.
$$
This proves (3). A similar argument to the proof of (3) does {\em not} give (4). For that, we introduce some normalizing matrices. 
%We now prove (4). Since $\|C_\chi(f^n(x))^{-1}\|\geq 1$ for all $n\in\Z$,
%we obtain that $\liminf_{n\to\pm\infty}\tfrac{1}{|n|}\log\|C_\chi(f^n(x))^{-1}\|\geq 0$.
%If we try to apply the previous method to obtain a reverse inequality, then we see that
%the norm of $(df^n_x)^{-1}=df^{-n}_{f^n(x)}$ appears. To control its growth rate, we introduce
%normalizing matrices.
Let $\lambda_1(x),\lambda_2(x)$ be the Lyapunov exponents of $df^n$ at $x$. By (\ref{equality-spectra}),
$D_\chi^{(n)}$ has the same Lyapunov exponents at $x$. Taking
$\Lambda_\chi(x):=\left[\begin{array}{cc}\lambda_1(x) & 0 \\ 0 & \lambda_2(x)\end{array}\right]$,
we have $\lim_{n\to\pm\infty}\tfrac{1}{|n|}\log \|(D_\chi^{(n)}(x)\Lambda_\chi(x)^{-n})^{\pm 1}\|=0$.

\medskip
Similarly, we can define $\Lambda(x):T_xM\to T_xM$ by $\Lambda(x)e^s_x=\lambda_1(x)e^s_x$
and $\Lambda(x)e^u_x=\lambda_2(x)e^u_x$ and observe that
$\lim_{n\to\pm\infty}\tfrac{1}{|n|}\log\|(df^n_x\Lambda(x)^{-n})^{\pm 1}\|=0$. Since
$\Lambda_\chi(x)=C_\chi(x)^{-1} \Lambda(x) C_\chi(x)$, it follows that
\begin{align*}
&C_\chi(f^n(x))^{-1}=D_\chi^{(n)}(x)C_\chi(x)^{-1}(df^n_x)^{-1}\\
&=[D_\chi^{(n)}(x)\Lambda_\chi(x)^{-n}][\Lambda_\chi(x)^n C_\chi(x)^{-1} \Lambda(x)^{-n}]
[df^n_x\Lambda(x)^{-n}]^{-1}\\
&=[D_\chi^{(n)}(x)\Lambda_\chi(x)^{-n}]C_\chi(x)^{-1} [df^n_x\Lambda(x)^{-n}]^{-1}
\end{align*}
and hence
\begin{align*}
&\limsup_{n\to\pm\infty}\tfrac{1}{|n|}\log\|C_\chi(f^n(x))^{-1}\|\\
&\leq\lim_{n\to\pm\infty}\tfrac{1}{|n|}\log \|D_\chi^{(n)}(x)\Lambda_\chi(x)^{-n}\|+
\lim_{n\to\pm\infty}\tfrac{1}{|n|}\log\|(df^n_x\Lambda(x)^{-n})^{-1}\|=0.
\end{align*}
Since $\liminf_{n\to\pm\infty}\tfrac{1}{|n|}\log\|C_\chi(f^n(x))^{-1}\|\geq 0$, property (4) holds.
Hence $X$ satisfies (2)--(4) and $\mu[X]=1$. Therefore
$X\cap{\rm Reg}\subset {\rm NUH}_\chi^*$ has full $\mu$--measure.
\end{proof}

\section{Non-linear Pesin theory}

We now define charts that make $f$ itself look like a hyperbolic matrix.
%Remember assumptions (A1)--(A2) on the maps $\exp{x}:[-\mathfrak r(x),\mathfrak r(x)]^2\to M$,
%$x\in M\backslash\mathfs D$.
%there is a constant $\mathfrak K>0$ s.t. for every $x\in M\backslash\mathfs S$, the map
%$\exp{x}:[-\mathfrak r(x),\mathfrak r(x)]^2\to M$ satisfies assumptions (A1)--(A2) for some
%$\mathfrak r(x)> d(x,\mathfs D)$.
%We use these maps to define the Pesin charts.

\medskip
\noindent
{\sc Pesin chart $\Psi_x$:} For $x\in{\rm NUH}_\chi$, let
$\Psi_x:R[\mathfrak r(x)]\to M$, $\Psi_x:=\exp{x}\circ C_\chi(x)$.
$\Psi_x$ is called the {\em Pesin chart at $x$}.

\medskip
Given $x\in M\backslash\mathfs D$, let $\iota_x:T_xM\to \R^2$ be an isometry.
If $y\in D_x$ and $A:\R^2\to T_yM$ is a linear map,
we can define $\widetilde{A}:\R^2\to \R^2$, $\widetilde{A}:=\iota_x\circ P_{y,x}\circ A$.
Again, $\widetilde A$ depends on $x$ but $\|\widetilde{A}\|$ does not.

\begin{lemma}\label{Lemma-Pesin-chart}
The Pesin chart $\Psi_x$ is a diffeomorphism onto its image. Moreover:
\begin{enumerate}[{\rm (1)}]
\item $\Psi_x$ is $2$--Lipschitz and $\Psi_x^{-1}$ is $2\|C_\chi(x)^{-1}\|$--Lipschitz.
\item $\|\widetilde{d(\Psi_x)_{v_1}}-\widetilde{d(\Psi_x)_{v_2}}\|\leq d(x,\mathfs D)^{-a}\|v_1-v_2\|$
for all $v_1,v_2\in R[\mathfrak r(x)]$.
\end{enumerate}
%\begin{enumerate}[{\rm (1)}]
%\item $\Psi_x$ is $\mathfrak K$--Lipschitz and $\Psi_x^{-1}$ is $\mathfrak K\|C_\chi(x)^{-1}\|$--Lipschitz.
%\item $\|d^2\Psi_x\|_{C^0}<\mathfrak K$.
%\item $d\Psi_x$ is $\mathfrak K$--Lipschitz and $\|d\Psi_x\|_{C^0}<\mathfrak K+1$.
%\end{enumerate}
\end{lemma}

\begin{proof}
Since $C_\chi(x)$ is a contraction,
$C_\chi(x)R[\mathfrak r(x)]\subset B_x[2\mathfrak r(x)]$
and so $\Psi_x$ is well-defined with inverse $C_\chi(x)^{-1}\circ \exp{x}^{-1}$.
It is a diffeomorphism because $C_\chi(x)$ and $\exp{x}$ are.

\medskip
\noindent
(1) By (A2), $\Psi_x$ is $2$--Lipschitz and
$\Psi_x^{-1}$ is $2\|C_\chi(x)^{-1}\|$--Lipschitz.

\medskip
\noindent
(2) Since $C_\chi(x)v_i\in B_x[2\mathfrak r(x)]$, (A3) implies that
\begin{align*}
&\|\widetilde{d(\Psi_x)_{v_1}}-\widetilde{d(\Psi_x)_{v_2}}\|=
\|\widetilde{d(\exp{x})_{C_\chi(x)v_1}}\circ C_\chi(x)-
\widetilde{d(\exp{x})_{C_\chi(x)v_2}}\circ C_\chi(x)\|\\
&\leq d(x,\mathfs D)^{-a}\|C_\chi(x)v_1-C_\chi(x)v_2\|\leq d(x,\mathfs D)^{-a}\|v_1-v_2\|.
\end{align*}
%, and
%$\|d^2\Psi_x\|_{C^0}\leq \|C_\chi(x)\|\|d^2\exp{x}\|_{C^0}<\mathfrak K$.
%In particular, $d\Psi_x$ is $\mathfrak K$--Lipschitz, and
%$\|d\Psi_x\|_{C^0}\leq \|d(\Psi_x)_0\|+\mathfrak K\|x\|\leq 1+\mathfrak K$.
\end{proof}

\medskip
Given $\ve>0$, let $I_\ve:=\{e^{-\frac{1}{3}\ve n}:n\geq 0\}$.
%For $x\in{\rm NUH}_\chi$,
%define two parameters
%$$
%\widetilde{\mathfrak r}(x):=\tfrac{1}{100\mathfrak K^2}\min\left\{\rho(x)^{2a},\rho(f(x))^{2a}\right\}
%$$
%and
%$$
%Q_{\ve,{\rm hyp}}(x):=\ve^{3/\beta}\min\left\{\|C_\chi(x)\|^{-12/\beta},\|C_\chi(f(x))^{-1}\|\rho(x)^{2a/\beta}\right\}.
%$$

\medskip
\noindent
{\sc Parameter $Q_\ve(x)$:} For $x\in{\rm NUH}_\chi$, let
$Q_\ve(x):=\max\{q\in I_\ve:q\leq \widetilde Q_\ve(x)\}$, where
$$
\widetilde Q_\ve(x)=\ve^{3/\beta}
\min\left\{\|C_\chi(x)^{-1}\|_{\rm Frob}^{-24/\beta},\|C_\chi(f(x))^{-1}\|^{-12/\beta}_{\rm Frob}\rho(x)^{72a/\beta}\right\}.
$$
%$$
%\widetilde Q_\ve(x)=
%\min\left\{\frac{\mathfrak G\rho(x)^a}{20},\frac{\rho(x)}{40\mathfrak K},
%\frac{\rho(x)^b\mathfrak G\rho(f(x))^a}{20 \times 2^b\mathfrak L\mathfrak K^2},
%\ve^{\frac{3}{\beta}}\|C_\chi(x)^{-1}\|^{-12/\beta},
%\ve^{\frac{3}{\beta}}\|C_\chi(f(x))^{-1}\|^{-12/\beta}(\mathfrak L\rho(x)^{-b})^{-2/\beta}\right\}.
%$$

\medskip
The term $\ve^{3/\beta}$ will allow to absorb multiplicative constants.
The choice of $Q_\ve(x)$ guarantees that
the composition $\Psi_{f(x)}^{-1}\circ f\circ \Psi_x$ is well-defined in $R[10Q_\ve(x)]$
and it is close to a linear hyperbolic map (Theorem \ref{Thm-non-linear-Pesin}),
and it allows to compare nearby Pesin charts (Proposition \ref{Lemma-overlap}).
We have the following bounds:
\begin{align*}
&Q_\ve(x)\leq \ve^{3/\beta}, \|C_\chi(x)^{-1}\|Q_\ve(x)^{\beta/24}\leq \ve^{1/8},
\|C_\chi(f(x))^{-1}\|Q_\ve(x)^{\beta/12}\leq \ve^{1/4},\\
&\rho(x)^{-a}Q_\ve(x)^{\beta/72}<\ve^{1/24}.
\end{align*}
%
%\begin{equation}\label{estimates-Q}
%Q_\ve(x)\leq \ve^{3/\beta}, \ \ Q_\ve(x)<\textcolor{red}{\tfrac{\ve^{3/\beta}\rho(x)^{3a}}{100\mathfrak K^2 2^{2a}}}<\tfrac{\rho(x)^{2a}}{100\mathfrak K^2 2^{2a}}<
%\tfrac{\rho(x)^{2a}}{100\mathfrak K^2 2^a}<
%\tfrac{\rho(x)^{a}}{100\mathfrak K^2}<
%\tfrac{\mathfrak r(x)}{100\mathfrak K^2}<\tfrac{\mathfrak r(x)}{40}.
%\end{equation}
%We also have upper bounds for $\|C_\chi(x)^{-1}\|$ and  $\|C_\chi(f(x))^{-1}\|$ in terms of $Q_\ve(x)$:
%\begin{equation}\label{bound-C-and-Q}
%\|C_\chi(x)^{-1}\|,\|C_\chi(f(x))^{-1}\|<\tfrac{\ve^{1/4}}{Q_\ve(x)}\cdot
%\end{equation}
%\footnote{The complicated formula for $Q_\ve(x)$
%reflects the lack of compactness of $M$, the lack of global continuity of $f$, 
%and the (possibly) unboundedness of $df$.}

\begin{lemma}[Temperedness lemma]\label{Lemma-temperedness}
If $x\in{\rm NUH}_\chi^*$, then
$$
\lim_{n\to\pm\infty}\tfrac{1}{|n|}\log Q_\ve(f^n(x))=0.
$$
\end{lemma}

\begin{proof}
Clearly $\limsup_{n\to\pm\infty}\tfrac{1}{|n|}\log Q_\ve(f^n(x))\leq 0$.
Reversely, $x\in{\rm Reg}$ implies that $\lim_{n\to\pm\infty}\tfrac{1}{|n|}\log\rho(f^n(x))=0$.
By property (4) in the definition of ${\rm NUH}_\chi^*$,
$\lim_{n\to\pm\infty}\tfrac{1}{|n|}\log\|C_\chi(f^n(x))^{-1}\|=0$ hence $\liminf_{n\to\pm\infty}\tfrac{1}{|n|}\log Q_\ve(f^n(x))\geq 0$.
%Therefore:
%\begin{enumerate}[$\circ$]
%\item $\liminf_{n\to\pm\infty}\tfrac{1}{n}\log \mathfrak r(f^n(x))\geq
%a\liminf_{n\to\pm\infty}\tfrac{1}{|n|}\log d(f^n(x),\mathfs S)=0$.
%\item $\liminf_{n\to\pm\infty}\tfrac{1}{n}\log  d(f^n(x),\mathfs D)^b \geq
%b\liminf_{n\to\pm\infty}\tfrac{1}{|n|}\log d(f^n(x),\mathfs S)=0$.
%\end{enumerate}
\end{proof}

\subsection{The map $f$ in Pesin charts}

%We express $f$ in Pesin charts.

\begin{theorem}\label{Thm-non-linear-Pesin}
The following holds for all $\ve>0$ small enough: If $x\in{\rm NUH}_\chi$
then $f_x:=\Psi_{f(x)}^{-1}\circ f\circ\Psi_x$ is well-defined on
$R[10Q_\ve(x)]$ and satisfies:
\begin{enumerate}[{\rm (1)}]
\item $d(f_x)_0=C_\chi(f(x))^{-1}\circ df_x\circ C_\chi(x)$.
\item $f_x(v_1,v_2)=(Av_1+h_1(v_1,v_2),Bv_2+h_2(v_1,v_2))$ for $(v_1,v_2)\in R[10Q_\ve(x)]$ where:
\begin{enumerate}[{\rm (a)}]
\item $|A|<e^{-\chi}$ and $|B|>e^\chi$, cf. Lemma \ref{Lemma-linear-reduction}.
\item $h_1(0,0)=h_2(0,0)=0$ and $\nabla h_1(0,0)=\nabla h_2(0,0)=0$.
\item $\|h_1\|_{1+\beta/2}<\ve$ and $\|h_2\|_{1+\beta/2}<\ve$.
\end{enumerate}
\item $\|df_x\|_0<\tfrac{2(1+e^{2\chi})}{\rho(x)^a}$.
\end{enumerate}
The norms above are taken in $R[10Q_\ve(x)]$.
A similar statement holds for $f_x^{-1}:=\Psi_x^{-1}\circ f^{-1}\circ \Psi_{f(x)}$.
\end{theorem}

%Above, $A$ and $B$ are the same constants of Lemma \ref{Lemma-linear-reduction}.

\begin{proof}
The first step is to show that $f_x:R[10Q_\ve(x)]\to\R^2$ is well-defined.
Using that $C_\chi(x)$ is a contraction, $C_\chi(x)R[10Q_\ve(x)]\subset B_x[20Q_\ve(x)]$.
Since $C_\chi(f(x))^{-1}$ is globally defined, it is enough to show that
$$
(f\circ\exp{x})(B_x[20Q_\ve(x)])\subset \exp{f(x)}(B_{f(x)}[2\mathfrak r(f(x))]).
$$
For small $\ve>0$ we have:
\begin{enumerate}[$\circ$]
\item $20Q_\ve(x)<2\mathfrak r(x)\Rightarrow\exp{x}$ is well-defined on $B_x[20Q_\ve(x)]$. By (A2),
$\exp{x}$ maps $B_x[20Q_\ve(x)]$ diffeomorphically into $B(x,40Q_\ve(x))$.
\item $40Q_\ve(x)<2\mathfrak r(x)\Rightarrow B(x,40Q_\ve(x))\subset B(x,2\mathfrak r(x))$. 
%$$
% d(y,\mathfs D)\geq  d(x,\mathfs D)-20\mathfrak KQ_\ve(x)>
% d(x,\mathfs D)-20\mathfrak K\frac{\rho(x)}{40\mathfrak K}=\frac{\rho(x)}{2}.
%$$
By (A5), $f$ maps $B(x,40Q_\ve(x))$ diffeomorphically into $B(f(x),40 \rho(x)^{-a}Q_\ve(x))$.
%Since
%$20\times 2^b\mathfrak L\mathfrak K\rho(x)^{-b}Q_\ve(x)\leq\tfrac{\mathfrak G\rho(f(x))^a}{\mathfrak K}
%<\tfrac{\mathfrak r(f(x))}{\mathfrak K}$, we get
%\begin{center}
%$f$ maps $B(x,20\mathfrak KQ_\ve(x))$ diffeomorphically into $B(f(x),\frac{\mathfrak r(f(x))}{\mathfrak K})$.
%\end{center}
\item $40\rho(x)^{-a}Q_\ve(x)<\tfrac{\mathfrak r(f(x))}{2}\Rightarrow B(f(x),40 \rho(x)^{-a}Q_\ve(x))\subset
B\left(f(x),\frac{\mathfrak r(f(x))}{2}\right)$. By (A2),
$\exp{f(x)}^{-1}$ maps $B\left(f(x),\frac{\mathfrak r(f(x))}{2}\right)$ diffeomorphically into
$B_{f(x)}[\mathfrak r(f(x))]$.
\end{enumerate}
%Since $B(0,\mathfrak r(f(x)))\subset R_{f(x)}[\mathfrak r(f(x))]$
Therefore $f_x:R[10Q_\ve(x)]\to\R^2$ is a diffeomorphism onto its image.

\medskip
We check (1)--(2). Property (1) is clear since
$d(\Psi_x)_0=C_\chi(x)$ and $d(\Psi_{f(x)})_0=C_\chi(f(x))$. By Lemma \ref{Lemma-linear-reduction},
$d(f_x)_0=\left[\begin{array}{cc}A & 0 \\ 0 & B\end{array}\right]$ with $|A|<e^{-\chi}$ and
$|B|>e^\chi$. Define $h_1,h_2:R[10Q_\ve(x)]\to\R $ by 
$f_x(v_1,v_2)=(Av_1+h_1(v_1,v_2),Bv_2+h_2(v_1,v_2))$. Then (a)--(b) are automatically
satisfied. It remains to prove (c).

\medskip
\noindent
{\sc Claim:} $\|d(f_x)_{w_1}-d(f_x)_{w_2}\|\leq \tfrac{\ve}{3}\|w_1-w_2\|^{\beta/2}$
for all $w_1,w_2\in R[10Q_\ve(x)]$.

\medskip
Before proving the claim, let us show how to conclude (c). Let $h=(h_1,h_2)$.
If $\ve>0$ is small enough then $R[10Q_\ve(x)]\subset B_x[1]$. Applying the claim with $w_2=0$,
we get $\|dh_w\|\leq \frac{\ve}{3}\|w\|^{\beta/2}<\tfrac{\ve}{3}$. By the mean value inequality,
$\|h(w)\|\leq \tfrac{\ve}{3}\|w\|<\tfrac{\ve}{3}$, hence $\|h\|_{1+\beta/2}<\ve$.

\begin{proof}[Proof of the claim.]
For $i=1,2$, define
$$
A_i= \widetilde{d(\exp{f(x)}^{-1})_{(f\circ \exp{x})(w_i)}}\,,\
B_i=\widetilde{df_{\exp{x}(w_i)}}\,,\ C_i=\widetilde{d(\exp{x})_{w_i}}.
$$
We first estimate $\|A_1 B_1 C_1-A_2 B_2 C_2\|$.
\begin{enumerate}[$\circ$]
\item By (A2), $\|A_i\|\leq 2$. By (A2), (A3), (A5):
\begin{align*}
&\|A_1-A_2\|\leq d(f(x),\mathfs D)^{-a}d((f\circ \exp{x})(w_1),(f\circ \exp{x})(w_2))\\
&\leq 2d(x,\mathfs D)^{-a}d(f(x),\mathfs D)^{-a}\|w_1-w_2\|\leq 2\rho(x)^{-2a}\|w_1-w_2\|.
\end{align*}
\item By (A5), $\|B_i\|\leq \rho(x)^{-a}$. By (A2) and (A6):
$$
\|B_1-B_2\|\leq \mathfrak K d(\exp{x}(w_1),\exp{x}(w_2))^{\beta}\leq 2\mathfrak K\|w_1-w_2\|^\beta.
$$
\item By (A2), $\|C_i\|\leq 2$. By (A3):
$$
\|C_1-C_2\|\leq d(x,\mathfs D)^{-a}\|w_1-w_2\|\leq \rho(x)^{-a}\|w_1-w_2\|.
$$
\end{enumerate}
By a crude approximation, we get $\|A_1 B_1 C_1-A_2 B_2 C_2\|\leq 24\mathfrak K\rho(x)^{-3a}\|w_1-w_2\|^\beta$.
Now we estimate $\|d(f_x)_{w_1}-d(f_x)_{w_2}\|$:
\begin{align*}
&\|d(f_x)_{w_1}-d(f_x)_{w_2}\|\leq \|C_\chi(f(x))^{-1}\| \|A_1 B_1 C_1-A_2 B_2 C_2\| \|C_\chi(x)\|\\
&\leq 24\mathfrak K\rho(x)^{-3a}\|C_\chi(f(x))^{-1}\|\|w_1-w_2\|^\beta.
\end{align*}
Since $\|w_1-w_2\|<40Q_\ve(x)$, if $\ve>0$ is small enough then
\begin{align*}
&24\mathfrak K\rho(x)^{-3a}\|C_\chi(f(x))^{-1}\|\|w_1-w_2\|^{\beta/2}
\leq 200\mathfrak K\rho(x)^{-3a}\ve^{3/2}\|C_\chi(f(x))^{-1}\|^{-5}\rho(x)^{36a}\\
&\leq 200\mathfrak K\ve^{3/2}<\ve.
\end{align*}
This completes the proof of the claim.
\end{proof}

\medskip
\noindent
(3) In the proof of Lemma \ref{Lemma-adaptedness} we showed that
$\|d(f_x)_0\|=|B(x)|\leq \tfrac{\sqrt{1+e^{2\chi}}}{\rho(x)^a}<\tfrac{1+e^{2\chi}}{\rho(x)^a}$.
By part (2) above, if $w\in R[10Q_\ve(x)]$ then
$\|d(f_x)_w\|\leq \ve\|w\|^{\beta/2}+\tfrac{1+e^{2\chi}}{\rho(x)^a}<\tfrac{2(1+e^{2\chi})}{\rho(x)^a}$,
since $\ve\|w\|^{\beta/2}<1<\tfrac{1+e^{2\chi}}{\rho(x)^a}$ for small $\ve>0$.
\end{proof}

\subsection{The overlap condition}\label{section-overlap}

We now want to change coordinates from $\Psi_x$ to $\Psi_y$ when $x,y$
are ``sufficiently close''. Even when $x$ and $y$ are very close, the behavior of $C_\chi(x)$ and $C_\chi(y)$
might differ, so we need to compare them.
%If $y\in D_x$ and $A:\R^2\to T_yM$ is a linear map,
%we can similarly define $\widetilde{A}:\R^2\to T_xM$, $\widetilde{A}:=P_{y,x}\circ A$.
%Again, $\widetilde A$ depends on $x$ but $\|\widetilde{A}\|$ does not.
%Therefore we need a stronger condition than just proximity in the metric of the manifold.
%
%\medskip
%\noindent
%{\sc Parallelization of $M$:} For each $x\in M$, let $D_x\subset M$ be an open ball with center
%$x$ and radius $< d(x,\mathfs D)/2$ and $\Theta_x:TD_x\to\mathbb R^2$ s.t.:
%\begin{enumerate}[(1)]
%\item $\Theta_x:T_yM\to \R^2$ is a linear isometry for all $y\in D_x$.
%\item If $\vartheta_y:=(\Theta_x\restriction_{T_yM})^{-1}:\R^2\to T_yM$, then
%$$
%(y,u)\in D_x\times B(0, d(x,\mathfs D)^a/2^a)\subset D_x\times\R^2\mapsto \exp{y}(\vartheta_y(u))\in M
%$$ is a smooth map with Lipschitz
%constant $\leq 2\mathfrak K$ with respect to the metric $ d(y,y')+\|u-u'\|$.
%\item The map $\vartheta_y^{-1}\circ\exp{y}^{-1}:D_x\to\R^2$ is well-defined for all $y\in D_x$,
%and $y\mapsto \vartheta_y^{-1}\circ\exp{y}^{-1}$ is a smooth map from $D_x$ to $C^2(D_x,\R^2)$
%with Lipschitz constant $\leq \mathfrak K$ in the $C^2$ norm, i.e.
%\begin{align}\label{equation-parallelization-C^2}
%\|\vartheta_{y_1}^{-1}\circ\exp{y_1}^{-1}-\vartheta_{y_2}^{-1}\circ\exp{y_2}^{-1}\|_{C^2}
%\leq \mathfrak K d(y_1,y_2), \ \forall y_1,y_2\in D_x.
%\end{align}
%\end{enumerate}
We will eventually consider Pesin charts with different domains.

\medskip
\noindent
{\sc Pesin chart $\Psi_x^\eta$:} It is restriction of $\Psi_x$ to $R[\eta]$, where $0<\eta\leq Q_\ve(x)$.

\medskip
\noindent
{\sc $\ve$--overlap:} Two Pesin charts $\Psi_{x_1}^{\eta_1},\Psi_{x_2}^{\eta_2}$ are said to
{\em $\ve$--overlap} if $\tfrac{\eta_1}{\eta_2}=e^{\pm\ve}$ and if there is $x\in M$ s.t.
$x_1,x_2\in D_x$ and $d(x_1,x_2)+\|\widetilde{C_\chi(x_1)}-\widetilde{C_\chi(x_2)}\|<(\eta_1\eta_2)^4$.

\medskip
We write $\Psi_{x_1}^{\eta_1}\overset{\ve}{\approx}\Psi_{x_2}^{\eta_2}$.
We claim that if $\ve>0$ is small enough, then $\Psi_{x_1}^{\eta_1}\overset{\ve}{\approx}\Psi_{x_2}^{\eta_2}$
implies that $\Psi_{x_i}(R[10Q_\ve(x_i)])\subset D_{x_1}\cap D_{x_2}$
(and hence we can apply (A1)--(A3) without mentioning $x$). We prove the inclusion for $i=1$.
Start noting that, since $d(x_1,x_2)<\ve d(x_2,\mathfs D)$,
$d(x_1,\mathfs D)=d(x_2,\mathfs D)\pm d(x_1,x_2)=(1\pm\ve)d(x_2,\mathfs D)$.
By Lemma \ref{Lemma-Pesin-chart}(1),
$\Psi_{x_1}(R[10Q_\ve(x_1)])\subset B(x_1,40Q_\ve(x_1))$. This ball
is contained in $D_{x_1}$ since $40 Q_\ve(x_1)\ll40\ve^{3/\beta}\rho(x_1)^a<\mathfrak r(x_1)$.
We have
$$
\Psi_{x_1}(R[10Q_\ve(x_1)])\subset B(x_1,40 Q_\ve(x_1))\subset
B(x_2,40 Q_\ve(x_1)+d(x_1,x_2)).
$$
Since
$40Q_\ve(x_1)+d(x_1,x_2)\leq 40\ve^{3/\beta}(1+\ve)^ad(x_2,\mathfs D)^a+
d(x_2,\mathfs D)^a<2\mathfrak r(x_2)$ for small $\ve>0$, it follows that
$\Psi_{x_1}(R[10Q_\ve(x_1)])\subset D_{x_2}$.
The next proposition shows that $\ve$--overlap is strong enough to guarantee
that the Pesin charts are close.

\begin{proposition}\label{Lemma-overlap}
The following holds for $\ve>0$ small enough.
If $\Psi_{x_1}^{\eta_1}\overset{\ve}{\approx}\Psi_{x_2}^{\eta_2}$ then:
\begin{enumerate}[{\rm (1)}]
\item {\sc Control of $s,u$:}
$\frac{s(x_1)}{s(x_2)}=e^{\pm(\eta_1\eta_2)^3}$ and $\frac{u(x_1)}{u(x_2)}=e^{\pm(\eta_1\eta_2)^3}$.
\item {\sc Control of $\alpha$:} $\frac{|\sin\alpha(x_1)|}{|\sin\alpha(x_2)|}=e^{\pm(\eta_1\eta_2)^3}$.
\item {\sc Overlap:} $\Psi_{x_i}(R[e^{-2\ve}\eta_i])\subset \Psi_{x_j}(R[\eta_j])$ for $i,j=1,2$.
\item {\sc Change of coordinates:} For $i,j=1,2$, the map $\Psi_{x_i}^{-1}\circ\Psi_{x_j}$
is well-defined in $R[d(x_j,\mathfs D)^a]$,
and $\|\Psi_{x_i}^{-1}\circ\Psi_{x_j}-{\rm Id}\|_{1+\beta/2}<\ve(\eta_1\eta_2)^2$
where the norm is taken in $R[d(x_j, \mathfs{D})^{2a}]$.
%where the norm is taken in $R[ d(x_i,\mathfs D)^a]$.
\end{enumerate}
\end{proposition}

\begin{proof} Assume $x_1,x_2\in D_x$, and let $C_i=\widetilde{C_\chi(x_i)}$.
By assumption, $d(x_1,x_2)+\|C_1-C_2\|<(\eta_1\eta_2)^4$. Note that $\Psi_{x_i}=\exp{x_i}\circ P_{x,x_i}\circ C_i$.

\medskip
\noindent
(1) We prove the estimate for $s$ (the calculation for $u$ is similar).
Since $\ve>0$ is small, it is enough to prove that $\left|\tfrac{s(x_1)}{s(x_2)}-1\right|<\ve^{3/\beta}(\eta_1\eta_2)^3$.
We have $s(x_i)^{-1}=\|C_\chi(x_i)e_1\|=\|C_ie_1\|$, hence
$|s(x_1)^{-1}-s(x_2)^{-1}|=|\|C_1e_1\|-\|C_2e_1\||\leq \|C_1-C_2\|<(\eta_1\eta_2)^4$.
Also
$s(x_1)=\|C_\chi(x_1)e_1\|^{-1}\leq \|C_\chi(x_1)^{-1}\|<\tfrac{\ve^{3/\beta}}{Q_\ve(x_1)}
<\tfrac{\ve^{3/\beta}}{\eta_1\eta_2}$,
therefore
$$
\left|\tfrac{s(x_1)}{s(x_2)}-1\right|=s(x_1)|s(x_1)^{-1}-s(x_2)^{-1}|<\ve^{3/\beta}(\eta_1\eta_2)^3.
$$

\medskip
\noindent
(2) We use the general inequality for an invertible linear transformation $L$:
\begin{equation}\label{gen-ineq-angles}
\frac{1}{\|L\|\|L^{-1}\|}\leq \frac{|\sin\angle(Lv,Lw)|}{|\sin\angle(v,w)|}\leq \|L\|\|L^{-1}\|.
\end{equation}
Apply this to $L=C_1C_2^{-1}$, $v=C_2e_1$, $w=C_2e_2$ to get that
$$
\frac{1}{\|C_1C_2^{-1}\|\|C_2C_1^{-1}\|}\leq \frac{\sin\alpha(x_1)}{\sin\alpha(x_2)}\leq \|C_1C_2^{-1}\|\|C_2C_1^{-1}\|.
$$
We have $\|C_1C_2^{-1}-{\rm Id}\|\leq\|C_1-C_2\|\|C_2^{-1}\|<\ve^{3/\beta}(\eta_1\eta_2)^3$,
and by symmetry $\|C_2C_1^{-1}-{\rm Id}\|<\ve^{3/\beta}(\eta_1\eta_2)^3$, therefore
$\|C_1C_2^{-1}\|\|C_2C_1^{-1}\|<[1+\ve^{3/\beta}(\eta_1\eta_2)^3]^2<e^{2\ve^{3/\beta}(\eta_1\eta_2)^3}<e^{(\eta_1\eta_2)^3}$.
The left hand side estimate is proved similarly.

\medskip
\noindent
(3) We prove that $\Psi_{x_1}(R[e^{-2\ve}\eta_1])\subset \Psi_{x_2}(R[\eta_2])$.
If $v\in R[e^{-2\ve}\eta_1]$ then
$\|C_\chi(x_1)v\|\leq \sqrt{2}e^{-2\ve}\eta_1<2\mathfrak r(x)$,
% \tfrac{ d(x_1,\mathfs D)^a}{2^{2a}}<\tfrac{ d(z,\mathfs D)^a}{2^a}$,
hence by (A1):
$$\Sas(C_\chi(x_1)v,C_\chi(x_2)v)\leq 2(d(x_1,x_2)+\|C_1v-C_2v\|)\leq 2(\eta_1\eta_2)^4.$$
By (A2), 
$d(\Psi_{x_1}(v),\Psi_{x_2}(v))\leq 4(\eta_1\eta_2)^4\Rightarrow\Psi_{x_1}(v)\in B(\Psi_{x_2}(v),4(\eta_1\eta_2)^4)$.
By Lemma \ref{Lemma-Pesin-chart}(1),
$B(\Psi_{x_2}(v),4(\eta_1\eta_2)^4)\subset \Psi_{x_2}(B)$ where
$B\subset \R^2$ is the ball with center $v$ and radius $8\|C_2^{-1}\|(\eta_1\eta_2)^4$,
hence it is enough that $B\subset R[\eta_2]$. If $w\in B$ then
$\|w\|_\infty\leq \|v\|_\infty+8\|C_2^{-1}\|(\eta_1\eta_2)^4\leq (e^{-\ve}+8\ve^{3/\beta})\eta_2<\eta_2$
for $\ve>0$ small enough.

\medskip
\noindent
(4) The proof that $\Psi_{x_2}^{-1}\circ \Psi_{x_1}$ is well-defined in
$R[d(x_1,\mathfs D)^a]$ is similar to the proof of (3). The only difference is in the last estimate:
if $\ve>0$ is small enough then for $w\in B$ it holds
\begin{align*}
&\|w\|\leq \|v\|+8\|C_2^{-1}\|(\eta_1\eta_2)^4\leq \sqrt{2}d(x_1,\mathfs D)^a+8(\eta_1\eta_2)^3\\
&\leq [\sqrt{2}(1+\ve)^a+8\ve^{3/\beta}]d(x_2,\mathfs D)^a<2\mathfrak r(x_2).
\end{align*}
% We include it for completeness.
%If $v\in R[ d(x_1,\mathfs D)^a/\mathfrak K^22^{2a}]$ then
%$\|C_1v\|<
%\tfrac{\sqrt{2}}{2^{2a}}\left(\tfrac{3}{2} d(z,\mathfs D)\right)^a<\tfrac{ d(z,\mathfs D)^a}{2^a}$,
%so
%%by condition (2) in the parallelization of $M$ we have
%%$$
%% d(\exp{x_1}\vartheta_{x_1}C_1v,\exp{x_2}\vartheta_{x_2}C_1v)\leq 2\mathfrak K d(x_1,x_2)
%%<2\mathfrak K(\eta_1\eta_2)^4.
%%$$
%$\Psi_{x_1}(v)\in B(\exp{x_2}\vartheta_{x_2}C_1v,2\mathfrak K(\eta_1\eta_2)^4)$.
%By (A2), it is enough to prove that this latter ball is contained in $B(x_2,\tfrac{\mathfrak r(x_2)}{\mathfrak K})$.
%Since $ d(x_1,\mathfs D)< d(x_2,\mathfs D)+(\eta_1\eta_2)^4<
%(1+\ve^{24/\beta}) d(x_2,\mathfs D)$,
%for small $\ve>0$ we have
%\begin{align*}
%& d(x_2,\exp{x_2}\vartheta_{x_2}C_1v)+2\mathfrak K(\eta_1\eta_2)^4<
%\mathfrak K\sqrt{2}\tfrac{ d(x_1,\mathfs D)^a}{\mathfrak K^2 2^{2a}}+\tfrac{\mathfrak r(x_2)}{50\mathfrak K}\\
%&<\tfrac{\sqrt{2}(1+\ve^{24/\beta})^a}{\mathfrak K2^{2a}} d(x_2,\mathfs D)^a+
%\tfrac{\mathfrak r(x_2)}{50\mathfrak K}<\tfrac{\mathfrak r(x_2)}{2\mathfrak K}+\tfrac{\mathfrak r(x_2)}{50\mathfrak K}<
%\tfrac{\mathfrak r(x_2)}{\mathfrak K}\cdot
%\end{align*}
%This proves that $\Psi_{x_2}^{-1}\circ \Psi_{x_1}$ is well-defined in
%$R[ d(x_1,\mathfs D)^a/\mathfrak K^22^{2a}]$.
Now:
\begin{align*}
&\Psi_{x_2}^{-1}\circ \Psi_{x_1}-{\rm Id}=C_2^{-1}\circ\exp{x_2}^{-1}\circ\exp{x_1}\circ C_1-{\rm Id}\\
&=[C_2^{-1}\circ P_{x_2,x}]\circ[\exp{x_2}^{-1}\circ\exp{x_1}-P_{x_1,x_2}]\circ [P_{x,x_1}\circ C_1]+C_2^{-1}(C_1-C_2)\\
&=[C_2^{-1}\circ P_{x_2,x}]\circ[\exp{x_2}^{-1}-P_{x_1,x_2}\circ\exp{x_1}^{-1}]\circ\Psi_{x_1}+C_2^{-1}(C_1-C_2).
%&=C_2^{-1}C_1+C_2^{-1}\circ[\vartheta_{x_2}^{-1}\circ\exp{x_2}^{-1}-\vartheta_{x_1}^{-1}\circ\exp{x_1}^{-1}]
%\circ\exp{x_1}\circ\vartheta_{x_1}\circ C_1\\
%&={\rm Id}+C_2^{-1}(C_2-C_1)+C_2^{-1}\circ
%[\vartheta_{x_2}^{-1}\circ\exp{x_2}^{-1}-\vartheta_{x_1}^{-1}\circ\exp{x_1}^{-1}]\circ\Psi_{x_1}.
\end{align*}
We calculate the $C^{1+\beta/2}$ norm of $[\exp{x_2}^{-1}-P_{x_1,x_2}\circ\exp{x_1}^{-1}]\circ\Psi_{x_1}$
in the domain $R[d(x_1, \mathfs{D})^{2a}]$.
By Lemma \ref{Lemma-Pesin-chart}(1), $\|d\Psi_{x_1}\|_0\leq 2$
and
$$
\Hol{\beta/2}(d\Psi_{x_1})\leq d(x_1,\mathfs D)^{-a}4d(x_1,\mathfs{D})^{2a(1-\beta/2)}=4d(x_1,\mathfs D)^{a(1-\beta)}
< 4.
$$
%\textcolor{red}{where we used that
%$d(x_1,\mathfs D)^{-a}\leq\rho(x_1)^{-a}\leq \ve^{1/2}Q_\ve(x_1)^{-\beta/6}$.}
Call $\Theta:=\exp{x_2}^{-1}-P_{x_1,x_2}\circ\exp{x_1}^{-1}$. For $\ve>0$ small enough, inside $D_{x_1}$ we have:
\begin{enumerate}[$\circ$]
\item By (A2),
$\|\Theta(v)\|\leq \Sas(\exp{x_2}^{-1}(v),\exp{x_1}^{-1}(v))\leq 2d(x_1,x_2)\leq 2\ve^{6/\beta}(\eta_1\eta_2)^3$
thus $\|\Theta\circ \Psi_{x_1}\|_0<\ve^{2/\beta}(\eta_1\eta_2)^3$.
\item By (A3), $\|d\Theta_v\|=\|\tau(x_2,v)-\tau(x_1,v)\|\leq d(x_1,\mathfs D)^{-a}d(x_1,x_2)
<\ve^{3/\beta}(\eta_1\eta_2)^3$.
Hence $\|d\Theta\|_0<\ve^{3/\beta}(\eta_1\eta_2)^3$ and
$\|d(\Theta\circ\Psi_{x_1})\|_0\leq 2\ve^{3/\beta}(\eta_1\eta_2)^3<\ve^{2/\beta}(\eta_1\eta_2)^3$.
\item By (A4),
\begin{align*}
&\|\widetilde{d\Theta_v}-\widetilde{d\Theta_w}\|=\|[\tau(x_2,v)-\tau(x_1,v)]-[\tau(x_2,w)-\tau(x_1,w)]\|\\
&\leq d(x_1,\mathfs D)^{-a}d(x_1,x_2)\|v-w\|%<40\ve^{1/2}(\eta_1\eta_2)^4\|v-w\|^{\beta/2}.
\end{align*}
hence ${\rm Lip}(d\Theta)\leq d(x_1,\mathfs D)^{-a}d(x_1,x_2)$.
%$\Hol{\beta}(d\tau)\leq 2d(x_1,\mathfs D)^{-a}40Q_\ve(x_1)^{1-\beta}<80\ve^{1/2}$.
\item Using that
$$
\Hol{\beta/2}(d(\Theta_1\circ\Theta_2))\leq \|d\Theta_1\|_0\Hol{\beta/2}(d\Theta_2)+
{\rm Lip}(d\Theta_1)\|d\Theta_2\|_0^{2}4d(x_1,\mathfs{D})^{2a(1-\beta/2)}
$$
for $\Theta_2$ with domain $R[d(x_1,\mathfs{D})^{2a}]$, we get that
\begin{align*}
&\Hol{\beta/2}[d(\Theta\circ\Psi_{x_1})]\leq \|d\Theta\|_0\Hol{\beta/2}(d\Psi_{x_1})+
{\rm Lip}(d\Theta)\|d\Psi_{x_1}\|_0^2 4d(x_1,\mathfs{D})^{2a(1-\beta/2)}\\
&<4\ve^{3/\beta}(\eta_1\eta_2)^3+ 
d(x_1,\mathfs D)^{-a}d(x_1,x_2)16d(x_1,\mathfs{D})^{2a(1-\beta/2)}\\
&<4\ve^{3/\beta}(\eta_1\eta_2)^3+16\ve^{6/\beta}(\eta_1\eta_2)^3<\ve^{2/\beta}(\eta_1\eta_2)^3.
\end{align*}
\end{enumerate}
This implies that $\|\Theta\circ\Psi_{x_1}\|_{1+\beta/2}<3\ve^{2/\beta}(\eta_1\eta_2)^3$, hence
$$
\|C_2^{-1}\circ P_{x_2,x}\circ\Theta\circ\Psi_{x_1}\|_{1+\beta/2}\leq \|C_2^{-1}\|3\ve^{2/\beta}(\eta_1\eta_2)^3
\leq 3\ve^{2/\beta}(\eta_1\eta_2)^2.
$$
Thus
$\|\Psi_{x_2}^{-1}\circ \Psi_{x_1}-{\rm Id}\|_{1+\beta/2}\leq
3\ve^{2/\beta}(\eta_1\eta_2)^2+\|C_2^{-1}\|(\eta_1\eta_2)^4<
3\ve^{2/\beta}(\eta_1\eta_2)^2+\ve^{3/\beta}(\eta_1\eta_2)^3<4\ve^{2/\beta}(\eta_1\eta_2)^2<\ve(\eta_1\eta_2)^2$.
%By Lemma \ref{Lemma-Pesin-chart} and (\ref{equation-parallelization-C^2}),
%the $C^{1+\beta/2}$ norm of the third summand above is at most
%$$
%\|C_2^{-1}\|\mathfrak K d(x_1,x_2)8\mathfrak K^3<\tfrac{\ve^{1/4}}{\eta_2}\mathfrak K(\eta_1\eta_2)^4
%8\mathfrak K^3
%<8\ve^2\mathfrak K^4(\eta_1\eta_2)^2<\tfrac{\ve}{2}(\eta_1\eta_2)^2
%$$
%whenever $\ve<\tfrac{1}{16\mathfrak K^4}$. The $C^{1+\beta/2}$ norm of the second summand
%is at most $\|C_2^{-1}\|\|C_2-C_1\|\leq \tfrac{\ve^{1/4}}{\eta_2}(\eta_1\eta_2)^4<\ve^2(\eta_1\eta_2)^2
%<\tfrac{\ve}{2}(\eta_1\eta_2)^2$.
%This completes the proof.
\end{proof}

\subsection{The map $f_{x,y}$}

Let $x,y\in{\rm NUH}_\chi$, and assume that $\Psi_{f(x)}^{\eta}\overset{\ve}{\approx}\Psi_y^{\eta'}$.
We want to change $\Psi_{f(x)}$ by $\Psi_y$ in $f_x$ and obtain a result
similar to Theorem \ref{Thm-non-linear-Pesin}.

\medskip
\noindent
{\sc The maps $f_{x,y}$ and $f_{x,y}^{-1}$:} If $\Psi_{f(x)}^{\eta}\overset{\ve}{\approx}\Psi_y^{\eta'}$,
define the map $f_{x,y}:=\Psi_y^{-1}\circ f\circ \Psi_x$.
If $\Psi_{x}^{\eta}\overset{\ve}{\approx}\Psi_{f^{-1}(y)}^{\eta'}$, define
$f_{x,y}^{-1}:=\Psi_x^{-1}\circ f^{-1}\circ \Psi_y$. 

\medskip
Any meaningful estimate of the regularity of $f_{x,y}$ in the $C^{1+\beta/2}$ norm cannot be better than
that of Theorem \ref{Thm-non-linear-Pesin}. In order to keep estimates of size $\ve$, we
consider the $C^{1+\beta/3}$ norm.

\begin{theorem}\label{Thm-non-linear-Pesin-2}
The following holds for all $\ve>0$ small enough:
If $x,y\in{\rm NUH}_\chi$ and $\Psi_{f(x)}^{\eta}\overset{\ve}{\approx}\Psi_{y}^{\eta'}$, then
$f_{x,y}$ is well-defined in $R[10Q_\ve(x)]$ and can be written as
$f_{x,y}(v_1,v_2)=(Av_1+h_1(v_1,v_2),Bv_2+h_2(v_1,v_2))$ where:
\begin{enumerate}[{\rm (a)}]
\item $|A|<e^{-\chi}$, $|B|>e^{\chi}$, cf. Lemma \ref{Lemma-linear-reduction}.
\item $\|h_i(0)\|<\ve\eta$, $\|\nabla h_i(0)\|<\ve\eta^{\beta/3}$, and
$\Hol{\beta/3}(\nabla h_i)<\ve$ where the norm is taken in $R[10Q_\ve(x)]$.
\end{enumerate}
If $\Psi_{x}^{\eta}\overset{\ve}{\approx}\Psi_{f^{-1}(y)}^{\eta'}$
then a similar statement holds for $f_{x,y}^{-1}$.
\end{theorem}

\begin{proof}
We write $f_{x,y}=(\Psi_y^{-1}\circ\Psi_{f(x)})\circ f_x=:g\circ f_x$ and see it as a
small perturbation of $f_x$.
By Theorem \ref{Thm-non-linear-Pesin}(2--3),
$$
f_x(0)=0,\ \|d(f_x)\|_0< \tfrac{2(1+e^{2\chi})}{\rho(x)^a},\ \|d(f_x)_v-d(f_x)_w\|\leq \ve\|v-w\|^{\beta/2}
$$
for $v,w\in R[10Q_\ve(x)]$, where the $C^0$ norm is taken in $R[10Q_\ve(x)]$,
and by Proposition \ref{Lemma-overlap}(4) we have
$$
\|g-{\rm Id}\|<\ve(\eta\eta')^2,\ \|d(g-{\rm Id})\|_0<\ve(\eta\eta')^2,\ \|dg_v-dg_w\|\leq\ve(\eta\eta')^2\|v-w\|^{\beta/2}
$$
for $v,w\in R[d(f(x),\mathfs{D})^{2a}]$, where the $C^0$ norm is taken in this same domain.

\medskip
We first prove that $f_{x,y}$ is well-defined in $R[10Q_\ve(x)]$.
We have
$$
f_x(R[10Q_\ve(x)])\subset B(0,40(1+e^{2\chi})\rho(x)^{-a}Q_\ve(x))\subset R[d(f(x),\mathfs D)^{2a}]
$$
since
$40(1+e^{2\chi})\rho(x)^{-a}Q_\ve(x)<40(1+e^{2\chi})\ve^{3/\beta}d(f(x),\mathfs D)^{2a}<d(f(x),\mathfs D)^{2a}$
for $\ve>0$ small enough. By Proposition \ref{Lemma-overlap}(4), $f_{x,y}$ is well-defined.
 
\medskip
Now we prove (b). Let $h:=(h_1,h_2)=g\circ f_x-d(f_x)_0$.
Then $\|h(0)\|=\|g(0)\|<\ve(\eta\eta')^2<\ve\eta$
and for $\ve>0$ small enough:
\begin{align*}
&\|\nabla h(0)\|\leq \|dg_0\circ d(f_x)_0-d(f_x)_0\|\leq \|d(g-{\rm Id})_0\|\|d(f_x)_0\|\\
&<\ve(\eta\eta')^2 2(1+e^{2\chi})\rho(x)^{-a}<\ve\eta\eta' 2\ve^{3/\beta}(1+e^{2\chi})<\ve\eta^{\beta/3}.
\end{align*}
%where in the fourth passage we used that $\eta\eta'^2\tfrac{2(1+e^{2\chi})}{\rho(x)^a}<2\ve^{6/\beta}(1+e^{2\chi})<1$
%for small $\ve>0$. 
Finally, since $f_x(R[10Q_\ve(x)])\subset R[d(f(x),\mathfs D)^{2a}]$, if $\ve>0$ is small enough then
for all $v,w\in R[10Q_\ve(x)]$ it holds:
\begin{align*}
&\|dh_v-dh_w\|=\|dg_{f_x(v)}\circ d(f_x)_v-dg_{f_x(w)}\circ d(f_x)_w\|\\
&\leq \|dg_{f_x(v)}-dg_{f_x(w)}\|\|d(f_x)_v\|+\|dg_{f_x(w)}\|\|d(f_x)_v-d(f_x)_w\|\\
&\leq \ve(\eta\eta')^2\|f_x(v)-f_x(w)\|^{\beta/2}\|d(f_x)\|_0+\ve\|dg\|_0\|v-w\|^{\beta/2}\\
&\leq (\ve(\eta\eta')^2\|d(f_x)\|_0^{1+\beta/2}+40\ve\|dg\|_0Q_\ve(x)^{\beta/6})\|v-w\|^{\beta/3}\\
&\leq (4\eta^2(1+e^{2\chi})^2\rho(x)^{-2a}+ 80Q_\ve(x)^{\beta/6})\ve\|v-w\|^{\beta/3}\\
&\leq (4\ve^{6/\beta}(1+e^{2\chi})^2+ 80\ve^{1/2})\ve\|v-w\|^{\beta/3}<\ve\|v-w\|^{\beta/3}.
\end{align*}
%since $(\eta\eta')^2\|dL_0\|_{C^0}^2<4\ve^{6/\beta}(1+e^{2\chi})^2<\tfrac{1}{2}$
%and $\|dJ\|_{C^0}(3\eta)^{\beta/6}<6\ve^{1/2}<\tfrac{1}{2}$ for small enough $\ve>0$.
\end{proof}

\section{Double charts and the graph transform method}

We now define $\ve$--double charts.
For $\ve>0$ small, define $\delta_\ve:=e^{-\ve n}\in I_\ve$ where $n$ is the unique positive integer s.t.
$e^{-\ve n}<\ve\leq e^{-\ve(n-1)}$. In particular, $\delta_\ve<\ve$.

\medskip
\noindent
{\sc $\ve$--double chart:} An {\em $\ve$--double chart} is a pair of Pesin charts
$\Psi_x^{p^s,p^u}=(\Psi_x^{p^s},\Psi_x^{p^u})$ where $p^s,p^u\in I_\ve$
with $0<p^s,p^u\leq \delta_\ve Q_\ve(x)$.

\medskip
The parameters $p^s/p^u$ control the local forward/backward hyperbolicity at $x$.
They are a way of separating the future and past dynamics. This will be better explained
below, when we introduce the parameters $q_\ve,q_\ve^s,q_\ve^u$.

\medskip
\noindent
{\sc Edge $v\overset{\ve}{\rightarrow}w$:} Given $\ve$--double charts $v=\Psi_x^{p^s,p^u}$
and $w=\Psi_y^{q^s,q^u}$, we draw an edge from $v$ to $w$ if the following conditions are
satisfied:
\begin{enumerate}[iii\,]
\item[(GPO1)] $\Psi_{f(x)}^{q^s\wedge q^u}\overset{\ve}{\approx}\Psi_y^{q^s\wedge q^u}$
and $\Psi_{f^{-1}(y)}^{p^s\wedge p^u}\overset{\ve}{\approx}\Psi_x^{p^s\wedge p^u}$.
\item[(GPO2)] $p^s=\min\{e^\ve q^s,\delta_\ve Q_\ve(x)\}$ and $q^u=\min\{e^\ve p^u,\delta_\ve Q_\ve(y)\}$.
\end{enumerate}

\medskip
(GPO1) allows to pass from an $\ve$--double chart at $x$
to an $\ve$--double chart at $y$ and vice-versa. (GPO2) is a greedy recursion
that implies that the local hyperbolicity parameters are the largest as possible.
It implies that $\tfrac{p^s\wedge p^u}{q^s\wedge q^u}=e^{\pm\ve}$.
(GPO2) will be crucial in the proof of the inverse theorem (Theorem \ref{Thm-inverse}).

\medskip
\noindent
{\sc $\ve$--generalized pseudo-orbit ($\ve$--gpo):} An {\em $\ve$--generalized pseudo-orbit ($\ve$--gpo)}
is a sequence $\un{v}=\{v_n\}_{n\in\Z}$ of $\ve$--double charts
s.t. $v_n\overset{\ve}{\rightarrow}v_{n+1}$ for all $n\in\Z$.

\subsection{The parameters $q_\ve(x),q_\ve^s(x),q_\ve^u(x)$}

A transition between Pesin charts only makes sense if their sizes $\eta,\eta'$ satisfy $\tfrac{\eta}{\eta'}=e^{\pm\ve}$
(see Theorem \ref{Thm-non-linear-Pesin-2}). Since the ratio $\tfrac{Q_\ve(f(x))}{Q_\ve(x)}$ 
might be different from $e^{\pm\ve}$, we introduce the parameter $q_\ve(x)$ below.
%Take $n_0>0$ s.t. $e^{-\ve n_0}<\ve\leq e^{-\ve(n_0+1)}$. 

\medskip
\noindent
{\sc Parameter $q_\ve(x)$:} For $x\in{\rm NUH}_\chi^*$, let
$q_\ve(x):=\delta_\ve\min\{e^{\ve|n|}Q_\ve(f^n(x)):n\in\Z\}$.

\medskip
The above minimum is the greedy way of defining values in $I_\ve$ smaller than $\ve Q_\ve$
with the required regularity property.

\begin{lemma}\label{Lemma-q}
For all $x\in{\rm NUH}_\chi^*$, $0<q_\ve(x)<\ve Q_\ve(x)$
and $\tfrac{q_\ve(f(x))}{q_\ve(x)}=e^{\pm\ve}$.
\end{lemma}

\begin{proof}
By Lemma \ref{Lemma-temperedness}, $\inf\{e^{\ve|n|}Q_\ve(f^n(x)):n\in\Z\}>0$.
Since zero is the only accumulation point of $I_\ve$,
$q_\ve(x)$ is well-defined and positive.
It is clear that $q_\ve(x)\leq \delta_\ve Q_\ve(x)<\ve Q_\ve(x)$. Since
$$
\min\{e^{\ve|n|}Q_\ve(f^{n+1}(x)):n\in\Z\}\leq e^{\ve}\min\{e^{\ve|n+1|}Q_\ve(f^{n+1}(x)):n\in\Z\},
$$
we have $q_\ve(f(x))\leq e^{\ve}q_\ve(x)$. Reversely,
$$
 e^{-\ve}\min\{e^{\ve|n+1|}Q_\ve(f^{n+1}(x)):n\in\Z\}\leq \min\{e^{\ve|n|}Q_\ve(f^{n+1}(x)):n\in\Z\}
$$
therefore $e^{-\ve}q_\ve(x)\leq q_\ve(f(x))$.
\end{proof}

We want to separate the dependence of $q_\ve(x)$
on the future from its dependence on the past, hence we define the one-sided versions of $q_\ve(x)$.

\medskip
\noindent
{\sc Parameters $q_\ve^s(x),q_\ve^u(x)$:}  For $x\in{\rm NUH}_\chi^*$, define
\begin{align*}
q_\ve^s(x)&:=\delta_\ve\min\{e^{\ve|n|}Q_\ve(f^n(x)):n\geq 0\}\\
q_\ve^u(x)&:=\delta_\ve\min\{e^{\ve|n|}Q_\ve(f^n(x)):n\leq 0\}.
\end{align*}

\begin{lemma}\label{Lemma-q^s}
For all $x\in{\rm NUH}_\chi^*$, the following holds:
\begin{enumerate}[{\rm (1)}]
\item {\sc Good definition:} $0<q_\ve^s(x),q_\ve^u(x)<\ve Q_\ve(x)$ and $q_\ve^s(x)\wedge q_\ve^u(x)=q_\ve(x)$.
\item {\sc Greedy algorithm:} For all $n\in\Z$ it holds
\begin{align*}
q_\ve^s(f^n(x))&=\min\{e^\ve q_\ve^s(f^{n+1}(x)),\delta_\ve Q_\ve(f^n(x))\}\\
q_\ve^u(f^n(x))&=\min\{e^\ve q_\ve^u(f^{n-1}(x)),\delta_\ve Q_\ve(f^n(x))\}.
\end{align*}
\end{enumerate}
\end{lemma}

\begin{proof}
As in the proof of Lemma \ref{Lemma-q}, $q^s_\ve(x)$ and $q^u_\ve(x)$ are well-defined and positive,
and by definition $q_\ve^s(x)\wedge q_\ve^u(x)=q_\ve(x)$. This proves (1).
We prove the first equality in (2): for a fixed $n\in\Z$ we have
\begin{align*}
&q_\ve^s(f^n(x))=\delta_\ve\min\{e^{\ve|m|}Q_\ve(f^m(f^n(x))):m\geq 0\}\\
&=\min\{\delta_\ve \min\{e^{\ve|m|}Q_\ve(f^{m+n}(x)):m\geq 1\},\delta_\ve Q_\ve(f^n(x))\}\\
&=\min\{e^\ve\delta_\ve\min\{e^{\ve|m|}Q_\ve(f^m(f^{n+1}(x))):m\geq 0\},\delta_\ve Q_\ve(f^n(x))\}\\
&=\min\{e^\ve q_\ve^s(f^{n+1}(x)),\delta_\ve Q_\ve(f^n(x))\}.
\end{align*}
The second equality is proved similarly.
\end{proof}

\medskip
\noindent
{\sc The set ${\rm NUH}_\chi^\#$:} It is the set of $x\in{\rm NUH}_\chi^*$ s.t.
$$
\limsup_{n\to\infty}q_\ve^s(f^n(x))>0\text{ and }\limsup_{n\to-\infty}q_\ve^u(f^n(x))>0.
$$

\subsection{ The graph transform method}

Let $v=\Psi_x^{p^s,p^u}$ be an $\ve$--double chart.

\medskip
\noindent
{\sc Admissible manifolds:} We define an {\em $s$--admissible manifold at $v$} as a set
of the form $\Psi_x\{(t,F(t)):|t|\leq p^s\}$ where $F:[-p^s,p^s]\to\R$ is a $C^{1+\beta/3}$ function
s.t.:
\begin{enumerate}
\item[(AM1)] $|F(0)|\leq 10^{-3}(p^s\wedge p^u)$.
\item[(AM2)] $|F'(0)|\leq \tfrac{1}{2}(p^s\wedge p^u)^{\beta/3}$.
\item[(AM3)] $\|F'\|_0+\Hol{\beta/3}(F')\leq\tfrac{1}{2}$ where the norms are taken in $[-p^s,p^s]$.
\end{enumerate}
Similarly, a {\em $u$--admissible manifold at $v$} is a set
of the form $\Psi_x\{(G(t),t):|t|\leq p^u\}$ where $G:[-p^u,p^u]\to\R$ is a $C^{1+\beta/3}$ function
satisfying (AM1)--(AM3), where the norms are taken in $[-p^u,p^u]$.

\medskip
The functions $F,G$ are called the {\em representing functions}.
We let $\mathfs M^s(v)$ (resp. $\mathfs M^u(v)$) denote the set of all $s$--admissible
(resp. $u$--admissible) manifolds at $v$.

\begin{lemma}\label{Lemma-admissible-manifolds}
The following holds for $\ve>0$ small enough. If $v=\Psi_x^{p^s,p^u}$ is an $\ve$--double chart, then for
every $V^s\in\mathfs M^s(v)$ and $V^u\in\mathfs M^u(v)$ it holds:
\begin{enumerate}[{\rm (1)}]
\item $V^s$ and $V^u$ intersect at a single point $P=\Psi_x(w)$, and $\|w\|_\infty<10^{-2}(p^s\wedge p^u)$.
\item $\tfrac{\sin\angle(V^s,V^u)}{\sin\alpha(x)}=e^{\pm(p^s\wedge p^u)^{\beta/4}}$
and $|\cos\angle(V^s,V^u)-\cos\alpha(x)|<2(p^s\wedge p^u)^{\beta/4}$, where
$\angle(V^s,V^u)=$ angle of intersection of the tangents to $V^s$ and $V^u$ at $P$.
%\footnote{If $F,G$ are the representing functions of
%$V^s,V^u$ and $P=\Psi_x(w_1,w_2)$, then 
%$\angle(V^s,V^u)=\angle ((d\Psi_x)_{(w_1,w_2)}(1,F'(w_1)),(d\Psi_x)_{(w_1,w_2)}(G'(w_2),1))$.} 
\end{enumerate}
\end{lemma}

When $M$ is compact and $f$ is a $C^{1+\beta}$ diffeomorphism, this is \cite[Prop. 4.11]{Sarig-JAMS}.
The same proof works almost verbatim, see the appendix for the necessary adaptations.

\medskip
Let $v=\Psi_x^{p^s,p^u}$, $w=\Psi_y^{q^s,q^u}$ be $\ve$--double charts with
$v\overset{\ve}{\rightarrow}w$. We now define the {\em graph transforms}: these are two maps
that work in different directions of the edge $v\overset{\ve}{\rightarrow}w$, one of them
sends $u$--admissible manifolds at $v$ to $u$--admissible manifolds at $w$, the other
sends $s$--admissible manifolds at $w$ to $s$--admissible manifolds at $v$. 

\medskip
\noindent
{\sc Graph transforms $\mathfs F_{v,w}^s$ and $\mathfs F_{v,w}^u$:} The {\em graph transform}
$\mathfs F_{v,w}^s:\mathfs M^s(w)\to\mathfs M^s(v)$ is the map that sends
an $s$--admissible manifold at $w$ with representing function $F:[-q^s,q^s]\to\R$ to the unique
$s$--admissible  manifold at $v$ with representing function $G:[-p^s,p^s]\to\R$ s.t.
$\{(t,G(t)):|t|\leq p^s\}\subset f_{x,y}^{-1}\{(t,F(t)):|t|\leq q^s\}$.  
Similarly, the {\em graph transform} $\mathfs F_{v,w}^u:\mathfs M^u(v)\to\mathfs M^u(w)$
is the map sending a $u$--admissible manifold at $v$ with representing function $F:[-p^u,p^u]\to\R$
to the unique $u$--admissible  manifold at $w$ with representing function $G:[-q^u,q^u]\to\R$ s.t.
$\{(G(t),t):|t|\leq q^u\}\subset f_{x,y}\{(F(t),t):|t|\leq p^u\}$.

\medskip
In other words, the representing functions of $s,u$--admissible manifolds
change by the application of $f_{x,y}^{-1},f_{x,y}$ respectively.
For $V_1,V_2\in\mathfs M^s(v)$ with representing functions
$F_1,F_2$ and for $i\geq 0$, define $ d_{C^i}(V_1,V_2):=\|F_1-F_2\|_i$ where the
norm is taken in $[-p^s,p^s]$. A similar definition
holds in $\mathfs M^u(v)$.

\begin{proposition}\label{Prop-graph-transform}
The following holds for $\ve>0$ small enough. If $v\overset{\ve}{\rightarrow}w$ then
$\mathfs F_{v,w}^s$ and $\mathfs F_{v,w}^u$ are well-defined. Furthermore, if
$V_1,V_2\in \mathfs M^u(v)$ then:
\begin{enumerate}[{\rm (1)}]
\item $ d_{C^0}(\mathfs F_{v,w}^u(V_1),\mathfs F_{v,w}^u(V_2))\leq e^{-\chi/2} d_{C^0}(V_1,V_2)$.
\item $ d_{C^1}(\mathfs F_{v,w}^u(V_1),\mathfs F_{v,w}^u(V_2))\leq
e^{-\chi/2}( d_{C^1}(V_1,V_2)+ d_{C^0}(V_1,V_2)^{\beta/3})$.
\item $f(V_i)$ intersects every element of $\mathfs M^u(w)$ at exactly one point.
\end{enumerate}
An analogous statement holds for $\mathfs F_{v,w}^s$.
\end{proposition}

When $M$ is compact and $f$ is a $C^{1+\beta}$ diffeomorphism,
this is \cite[Prop. 4.12 and 4.14]{Sarig-JAMS}.
The proof in our case requires some minor changes, see Appendix B.

\subsection{Stable and unstable manifolds of $\ve$--gpo's}

Call a sequence ${\un v}^+=\{v_n\}_{n\geq 0}$ a {\em positive $\ve$--gpo} if $v_n\overset{\ve}{\to}v_{n+1}$
for all $n\geq 0$. Similarly, a {\em negative $\ve$--gpo} is a sequence ${\un v}^-=\{v_n\}_{n\leq 0}$
s.t. $v_{n-1}\overset{\ve}{\to}v_n$ for all $n\leq 0$.

\medskip
\noindent
{\sc Stable/unstable manifold of positive/negative $\ve$--gpo:} The {\em stable manifold}
of a positive $\ve$--gpo ${\un v}^+=\{v_n\}_{n\geq 0}$ is 
$$
V^s[{\un v}^+]:=\lim_{n\to\infty}
(\mathfs F_{v_0,v_1}^s\circ\cdots\circ\mathfs F_{v_{n-2},v_{n-1}}^s\circ\mathfs F_{v_{n-1},v_n}^s)(V_n)
$$
for some (any) choice of $(V_n)_{n\geq 0}$ with $V_n\in\mathfs M^s(v_n)$.
The {\em unstable manifold} of a negative $\ve$--gpo ${\un v}^-=\{v_n\}_{n\leq 0}$ is 
$$
V^u[{\un v}^-]:=\lim_{n\to-\infty}
(\mathfs F_{v_{-1},v_0}^u\circ\cdots\circ\mathfs F_{v_{n+1},v_{n+2}}^u\circ\mathfs F_{v_n,v_{n+1}}^u)(V_n)
$$
for some (any) choice of $(V_n)_{n\leq 0}$ with $V_n\in\mathfs M^u(v_n)$.

\medskip
For an $\ve$--gpo $\un{v}=\{v_n\}_{n\in\Z}$, let
$V^s[\un v]:=V^s[\{v_n\}_{n\geq 0}]$
and $V^u[\un v]:=V^u[\{v_n\}_{n\leq 0}]$.

\begin{proposition}\label{Prop-stable-manifolds}
The following holds for all $\ve>0$ small enough.
\begin{enumerate}[{\rm (1)}]
\item {\sc Admissibility:} $V^s[{\un v}^+],V^s[{\un v}^-]$ are well-defined admissible manifolds at $v_0$.
\item {\sc Invariance:}
$$
f(V^s[\{v_n\}_{n\geq 0}])\subset V^s[\{v_n\}_{n\geq 1}]\text{ and }
f^{-1}(V^u[\{v_n\}_{n\leq 0}])\subset V^u[\{v_n\}_{n\leq -1}].
$$
\item {\sc Shadowing:} If ${\un v}^+=\{\Psi_{x_n}^{p^s_n,p^u_n}\}_{n\geq 0}$ then
$$
V^s[{\un v}^+]=\{x\in \Psi_{x_0}(R[p^s_0]):f^n(x)\in \Psi_{x_n}(R[10Q_\ve(x_n)]),\,\forall n\geq 0\}.
$$
An analogous statement holds for $V^u[{\un v}^-]$.
\item {\sc Hyperbolicity:} If $x,y\in V^s[{\un v}^+]$ then $d(f^n(x),f^n(y))\xrightarrow[n\to\infty]{}0$,
if $x,y\in V^u[{\un v}^-]$ then $d(f^n(x),f^n(y))\xrightarrow[n\to-\infty]{}0$, and the rates are exponential.
\item {\sc H\"older property:} The map $\un v^+\mapsto V^s[\un v^+]$ is H\"older continuous,
i.e. there exists $K>0$ and $\theta<1$ s.t. for all $N\geq 0$, if $\un v^+,\un w^+$ are positive $\ve$--gpo's
with $v_n=w_n$ for $n=0,\ldots,N$
then $ d_{C^1}(V^s[\un v^+],V^s[\un w^+])\leq K\theta^N$. The same holds for the map
$\un v^-\mapsto V^u[\un v^-]$.
\end{enumerate}
\end{proposition}

When $M$ is compact and $f$ is a $C^{1+\beta}$ diffeomorphism,
this is \cite[Prop. 4.15]{Sarig-JAMS}. The same proof works in our case:
it uses the hyperbolicity of $f_{x,y}$ (Theorem \ref{Thm-non-linear-Pesin-2}),
and the contracting properties of the graph transforms (Proposition \ref{Prop-graph-transform}).
Proposition \ref{Prop-stable-manifolds} ensures that every $\ve$--gpo is associated to a unique point.

\medskip
\noindent
{\sc Shadowing:} We say that an $\ve$--gpo $\{\Psi_{x_n}^{p^s_n,p^u_n}\}$ {\em shadows}
a point $x\in M$ when $f^n(x)\in \Psi_{x_n}(R[p^s_n\wedge p^u_n])$ for all $n\in\Z$.

\begin{lemma}\label{Lemma-shadowing}
Every $\ve$--gpo shadows a unique point.
\end{lemma}

\begin{proof}
Let $\un v=\{v_n\}_{n\in\Z}$ be an $\ve$--gpo. By Proposition \ref{Prop-stable-manifolds}(3),
any point shadowed by $\un v$ must lie in $V^s[\{v_n\}_{n\geq 0}]\cap V^u[\{v_n\}_{n\leq 0}]$.
By Lemma \ref{Lemma-admissible-manifolds}(1), this intersection consists of a singleton $\{x\}$.
Write $v_n=\Psi_{x_n}^{p^s_n,p^u_n}$. By Proposition \ref{Prop-stable-manifolds}(2),
for all $n\geq 0$ we have $f^n(x)\in V^s[\{v_{n+k}\}_{k\geq 0}]\subset \Psi_{x_n}(R[10Q_\ve(x_n)])$,
and for all $n\leq 0$ we have $f^n(x)\in V^u[\{v_{n+k}\}_{k\leq 0}]\subset\Psi_{x_n}(R[10Q_\ve(x_n)])$,
hence $\un v$ shadows $x$.
\end{proof}

\section{Coarse graining}\label{Section-coarse-graining}

We now pass to a countable set of $\ve$--double charts that define a topological Markov shift that
shadows all relevant orbits.

\begin{theorem}\label{Thm-coarse-graining}
For all $\ve>0$ sufficiently small, there exists a countable family $\mathfs A$ of $\ve$--double charts
with the following properties:
\begin{enumerate}[{\rm (1)}]
\item {\sc Discreteness}: For all $t>0$, the set $\{\Psi_x^{p^s,p^u}\in\mathfs A:p^s,p^u>t\}$ is finite.
\item {\sc Sufficiency:} If $x\in {\rm NUH}_\chi^*$ then there is a sequence $\un v\in{\mathfs A}^{\Z}$
that shadows $x$.
\item {\sc Relevance:} For all $v\in \mathfs A$ there is an $\ve$--gpo $\un{v}\in\mathfs A^\Z$
with $v_0=v$ that shadows a point in ${\rm NUH}_\chi^*$.
\end{enumerate}
\end{theorem}

Parts (1) and (3) will be crucial to prove the inverse theorem (Theorem \ref{Thm-inverse}).
Part (2) says that the $\ve$--gpo's in $\mathfs A$ shadow a.e. point with respect to every
$f$--adapted $\chi$--hyperbolic measure, see Lemma \ref{Lemma-adaptedness}.

\begin{remark}
In part (2) we only assume that $x\in{\rm NUH}_\chi^*$, while \cite{Lima-Sarig,Sarig-JAMS}
require the stronger assumption $x\in{\rm NUH}_\chi^\#$. The reason of the improvement
is that here $q_\ve(x)$ is defined as a minimum instead of a sum,
and hence Lemma \ref{Lemma-q^s}(1) holds.
\end{remark}

\begin{proof}
When $M$ is compact and $f$ is a diffeomorphism, the above statement is consequence
of Propositions 3.5, 4.5 and Lemmas 4.6, 4.7 of \cite{Sarig-JAMS}. When $M$ is compact (with boundary)
and $f$ is a local diffeomorphism with bounded derivatives, this is Proposition 4.3 of \cite{Lima-Sarig}.
We follow the same strategy, adapted to our context.

\medskip
For $t>0$, let $M_t=\{x\in M: d(x,\mathfs D)\geq t\}$.
Since $M$ has finite diameter (remember we are even assuming it is smaller than one), each $M_t$
is precompact\footnote{$M_t$ might not be compact, since $M$ might have boundaries.}.
Let $\N_0=\N\cup\{0\}$. Fix a countable open cover $\mathfs P=\{D_i\}_{i\in\N_0}$ of $M\backslash\mathfs D$ s.t.:
\begin{enumerate}[$\circ$]
\item $D_i:=D_{z_i}=B(z_i,2\mathfrak r(z_i))$ for some $z_i\in M$.
\item For every $t>0$, $\{D\in\mathfs P:D\cap M_t\neq\emptyset\}$ is finite.
\end{enumerate}
%These are the domains of the parallelization of $M$ we will consider in the construction to come.

\medskip
Let $X:=M^3\times {\rm GL}(2,\R)^3\times (0,1]$.
For $x\in{\rm NUH}_\chi^*$, let
$\Gamma(x)=(\un x,\un C,\un Q)\in X$ with
\begin{align*}
\un x=(f^{-1}(x),x,f(x)),\ \un C=(C_\chi(f^{-1}(x)),C_\chi(x),C_\chi(f(x))),\ \un Q=Q_\ve(x).
\end{align*}
Let $Y=\{\Gamma(x):x\in{\rm NUH}_\chi^*\}$. We want to construct a countable dense subset
of $Y$. Since the maps $x\mapsto C_\chi(x),Q_\ve(x)$ are usually just measurable,
we apply a precompactness argument.
For each triple of vectors $\un{k}=(k_{-1},k_0,k_1)$, $\un{\ell}=(\ell_{-1},\ell_0,\ell_1)$,
$\un a=(a_{-1},a_0,a_1)\in\N_0^3$ and $m\in\N_0$, define
$$
Y_{\un k,\un \ell,\un a,m}:=\left\{\Gamma(x)\in Y:
\begin{array}{cl}
e^{-k_i-1}\leq d(f^i(x),\mathfs D)< e^{-k_i},& -1\leq i\leq 1\\
e^{\ell_i}\leq\|C_\chi(f^i(x))^{-1}\|<e^{\ell_i+1},&-1\leq i\leq 1\\
f^i(x)\in D_{a_i},&-1\leq i\leq 1\\
e^{-m-1}\leq Q_\ve(x)< e^{-m}&\\
\end{array}
\right\}.
$$

\medskip
\noindent
{\sc Claim 1:} $Y=\bigcup_{\un k,\un\ell,\un a\in\N_0^3\atop{m\in\N_0}}Y_{\un k,\un\ell,\un a,m}$, and each
$Y_{\un k,\un\ell,\un a,m}$ is precompact in $X$.

\medskip
\noindent
{\em Proof of claim $1$.}
The first statement is clear. We focus on the second.
Fix $\un k,\un\ell,\un a\in \N_0^3$, $m\in\N_0$. Take $\Gamma(x)\in Y_{\un k,\un\ell,\un a,m}$. Then
$$
\un x\in M_{e^{-k_{-1}-1}}\times M_{e^{-k_0-1}}\times M_{e^{-k_1-1}},$$
a precompact subset of $M^3$.
For $|i|\leq 1$, $C_\chi(f^i(x))$ is an element of ${\rm GL}(2,\R)$ with norm $\leq 1$ and
inverse norm $\leq e^{\ell_i+1}$, hence it belongs to a compact subset of ${\rm GL}(2,\R)$.
This guarantees that $\un C$ belongs to a compact subset of ${\rm GL}(2,\R)^3$. Also,
$\un Q\in [e^{-m-1},1]$, a compact subinterval of $(0,1]$. Since the product of precompact sets
is precompact, the claim is proved.

\medskip
Let $j\geq 0$. By claim 1, there exists a finite set
$Y_{\un k,\un\ell,\un a,m}(j)\subset Y_{\un k,\un\ell,\un a,m}$
s.t. for every $\Gamma(x)\in Y_{\un k,\un\ell,\un a,m}$
there exists $\Gamma(y)\in Y_{\un k,\un\ell,\un a,m}(j)$
s.t.:
\begin{enumerate}[{\rm (a)}]
\item $ d(f^i(x),f^i(y))+\|\widetilde{C_\chi(f^i(x))}-\widetilde{C_\chi(f^i(y))}\|<e^{-8(j+2)}$
for $-1\leq i\leq 1$.
\item $\tfrac{Q_\ve(x)}{Q_\ve(y)}=e^{\pm \ve/3}$.
\end{enumerate}

\medskip
\noindent
{\sc The alphabet $\mathfs A$:} Let $\mathfs A$ be the countable family of $\Psi_x^{p^s,p^u}$ s.t.:
\begin{enumerate}[i i)]
\item[(CG1)] $\Gamma(x)\in Y_{\un k,\un\ell,\un a,m}(j)$ for some
$(\un k,\un\ell,\un a,m,j)\in\N_0^3\times\N_0^3\times\N_0^3\times \N_0\times \N_0$.
\item[(CG2)] $0<p^s,p^u\leq \delta_\ve Q_\ve(x)$ and $p^s,p^u\in I_\ve$.
\item[(CG3)] $e^{-j-2}\leq p^s\wedge p^u\leq e^{-j+2}$.
\end{enumerate}

\medskip
\noindent
{\em Proof of discreteness.}
We will use the following fact, whose proof is in the appendix:
\begin{equation}\label{inequality-C}
\|C_\chi(f^{-1}(x))^{-1}\|\leq 2\rho(x)^{-2a}(1+e^\chi\rho(x)^{-a})\|C_\chi(x)^{-1}\|.
\end{equation}

Fix $t>0$, and let $\Psi_x^{p^s,p^u}\in\mathfs A$ with $p^s,p^u>t$.
Note that $\rho(x)>\rho(x)^{2a}>Q_\ve(x)>p^s,p^u>t$.
If $\Gamma(x)\in Y_{\un k,\un\ell,\un a,m}(j)$ then:
\begin{enumerate}[$\circ$]
\item Finiteness of $\un k$: for $|i|\leq 1$, $e^{-k_i}> d(f^i(x),\mathfs D)\geq\rho(x)>t$, hence $k_i< |\log t|$.
\item Finiteness of $\un\ell$: for $i=0,1$, $e^{\ell_i}\leq \|C_\chi(f^i(x))^{-1}\|<Q_\ve(x)^{-1}<t^{-1}$,
hence $\ell_i<|\log t|$. By inequality (\ref{inequality-C}) above,
$$
e^{\ell_{-1}}\leq \|C_\chi(f^{-1}(x))^{-1}\|<2t^{-1}(1+e^\chi t^{-1})t^{-1}<4e^\chi t^{-3},
$$
hence $\ell_{-1}<\log 4+\chi+3|\log t|=:T_t$, which is bigger than $|\log t|$.
\item Finiteness of $\un a$: $f^i(x)\in D_{a_i}\cap M_t$, hence $D_{a_i}$ belongs to the finite set
$\{D\in\mathfs P:D\cap M_t\neq\emptyset\}$.
\item Finiteness of $m$: $e^{-m}>Q_\ve(x)>t$, hence $m<|\log t|$.
\item Finiteness of $j$: $t<p^s\wedge p^u\leq e^{-j+2}$, hence $j\leq |\log t|+2$.
\item Finiteness of $(p^s,p^u)$: $t< p^s,p^u$, hence $\#\{(p^s,p^u):p^s,p^u>t\}\leq \#(I_\ve\cap (t,1])^2$ is finite.
\end{enumerate}
The first five items above give that, for $\un a\in\N_0^3$ and $t>0$,
\begin{align*}
\#\left\{\Gamma(x):
\begin{array}{c}
\Psi_x^{p^s,p^u}\in\mathcal A\text{ s.t. }p^s,p^u>t\\
\text{and }f^i(x)\in D_{a_i}, |i|\leq 1
\end{array}
\right\}\leq\sum_{j=0}^{\lceil |\log t|\rceil+2}\sum_{m=0}^{\lceil |\log t|\rceil}
\sum_{-1\leq i\leq 1\atop{k_i,\ell_i=0}}^{T_t}
\# Y_{\un k,\un\ell,\un a,m}(j)
\end{align*}
is the finite sum of finite terms, hence finite. Together with the last item above,
we conclude that
\begin{align*}
\#\left\{\Psi_x^{p^s,p^u}\in\mathcal A:p^s,p^u>t\right\}&\leq 
\sum_{j=0}^{\lceil |\log t|\rceil+2}\sum_{m=0}^{\lceil |\log t|\rceil}\sum_{-1\leq i\leq 1\atop{k_i,\ell_i=0}}^{T_t}
\# Y_{\un k,\un\ell,\un a,m}(j)\\
&\ \ \ \times (\#\{D\in\mathfs P:D\cap M_t\neq\emptyset\})^3\times (\#(I_\ve\cap (t,1]))^2
\end{align*}
is finite. This proves the discreteness of $\mathfs A$.

\medskip
\noindent
{\em Proof of sufficiency.}
Let $x\in {\rm NUH}_\chi^*$. Take $(k_i)_{i\in\Z},(\ell_i)_{i\in\Z},(m_i)_{i\in\Z},(a_i)_{i\in\Z},(j_i)_{i\in\Z}$  s.t.:
\begin{align*}
& d(f^i(x),\mathfs D)\in [e^{-k_i-1},e^{-k_i}), \|C_\chi(f^i(x))^{-1}\|\in [e^{\ell_i},e^{\ell_i+1}),\\
&Q_\ve(f^i(x))\in [e^{-m_i-1},e^{-m_i}),f^i(x)\in D_{a_i}, q_\ve(f^i(x))\in[e^{-j_i-1},e^{-j_i+1}).
\end{align*}
For $n\in\Z$, define
$$
\un k^{(n)}=(k_{n-1},k_n,k_{n+1}),\ \un\ell^{(n)}=(\ell_{n-1},\ell_n,\ell_{n+1}),\ \un a^{(n)}=(a_{n-1},a_n,a_{n+1}).
$$
Then $\Gamma(f^n(x))\in Y_{\un k^{(n)},\un\ell^{(n)},\un a^{(n)},m_n}$.
Take $\Gamma(x_n)\in Y_{\un k^{(n)},\un\ell^{(n)},\un a^{(n)},m_n}(j_n)$
s.t.:
\begin{enumerate}[aaa)]
\item[(${\rm a}_n$)] $ d(f^i(f^n(x)),f^i(x_n))+
\|\widetilde{C_\chi(f^i(f^n(x)))}-\widetilde{C_\chi(f^i(x_n))}\|<e^{-8(j_n+2)}$
for $|i|\leq 1$.
\item[(${\rm b}_n$)] $\tfrac{Q_\ve(f^n(x))}{Q_\ve(x_n)}=e^{\pm\ve/3}$.
\end{enumerate}
Define $p^s_n=\delta_\ve\min\{e^{\ve|k|}Q_\ve(x_{n+k}):k\geq 0\}$ and
$p^u_n=\delta_\ve\min\{e^{\ve|k|}Q_\ve(x_{n+k}):k\leq 0\}$.
We claim that $\{\Psi_{x_n}^{p^s_n,p^u_n}\}_{n\in\Z}$ is an $\ve$--gpo
in $\mathfs A^\Z$ that shadows $x$.

\medskip
\noindent
{\sc Claim 2:} $\Psi_{x_n}^{p^s_n,p^u_n}\in\mathfs A$ for all $n\in\Z$.

\medskip
\noindent
(CG1) By definition, $\Gamma(x_n)\in Y_{\un k^{(n)},\un\ell^{(n)},\un a^{(n)},m_n}(j_n)$.

\noindent
(CG2) By (${\rm b}_n$), $\inf\{e^{\ve|k|}Q_\ve(x_{n+k}):k\geq 0\}=
e^{\pm\ve/3}\inf\{e^{\ve|k|}Q_\ve(f^{n+k}(x)):k\geq 0\}$ is positive. Since the only accumulation
point of $I_\ve$ is zero, it follows that $p^s_n,p^u_n$ are well-defined and positive.
The other conditions are clear from the definition.

%0<p^s_n\wedge p^u_n\leq e^{\ve/3} q_\ve(f^n(x))<
%\ve e^{\ve/3}Q_\ve(f^n(x))\leq \ve e^{2\ve/3}Q_\ve(x_n)<Q_\ve(x_n)$$
%for $\ve>0$ small enough.

\noindent
(CG3) Again by (${\rm b}_n$), we have
$$
\min\{e^{\ve|k|}Q_\ve(x_{n+k}):k\geq 0\}=e^{\pm\ve/3}\min\{e^{\ve|k|}Q_\ve(f^{n+k}(x)):k\geq 0\}
$$
hence $\tfrac{p^s_n}{q_\ve^s(f^n(x))}=e^{\pm\ve/3}$, and analogously 
$\tfrac{p^u_n}{q_\ve^u(f^n(x))}=e^{\pm\ve/3}$. By Lemma \ref{Lemma-q^s}(1),
$p^s_n\wedge p^u_n=e^{\pm\ve/3}q_\ve(f^n(x))\in [e^{-j_n-2},e^{-j_n+2})$.

\medskip
\noindent
{\sc Claim 3:} $\Psi_{x_n}^{p^s_n,p^u_n}\overset{\ve}{\rightarrow}\Psi_{x_{n+1}}^{p^s_{n+1},p^u_{n+1}}$
for all $n\in\Z$.

\medskip
\noindent
(GPO1) We have $f(x_n),x_{n+1}\in D_{a_{n+1}}$, and by (${\rm a}_n$) with $i=1$
and (${\rm a}_{n+1}$) with $i=0$, we have
\begin{align*}
& d(f(x_n),x_{n+1})+\|\widetilde{C_\chi(f(x_n))}-\widetilde{C_\chi(x_{n+1})}\|\\
&\leq  d(f^{n+1}(x),f(x_n))+
\|\widetilde{C_\chi(f^{n+1}(x))}-\widetilde{C_\chi(f(x_n))}\|\\
&\ \ \ \,+ d(f^{n+1}(x),x_{n+1})+
\|\widetilde{C_\chi(f^{n+1}(x))}-\widetilde{C_\chi(x_{n+1})}\|\\
&<e^{-8(j_n+2)}+e^{-8(j_{n+1}+2)}\leq e^{-8}\left(q_\ve(f^n(x))^8+q_\ve(f^{n+1}(x))^8\right)\\
&\overset{!}{\leq} e^{-8}(1+e^{8\ve})q_\ve(f^{n+1}(x))^8\leq e^{-8+8\ve/3}(1+e^{8\ve})(p^s_{n+1}\wedge p^u_{n+1})^8
\overset{!!}{<}(p^s_{n+1}\wedge p^u_{n+1})^8,
\end{align*}
where in $\overset{!}{\leq}$ we used Lemma \ref{Lemma-q} and in $\overset{!!}{<}$
we used that $e^{-8+8\ve/3}(1+e^{8\ve})<1$ when $\ve>0$ is sufficiently small. This proves
that $\Psi_{f(x_n)}^{p^s_{n+1}\wedge p^u_{n+1}}\overset{\ve}{\approx}\Psi_{x_{n+1}}^{p^s_{n+1}\wedge p^u_{n+1}}$.
Similarly, we prove that
$\Psi_{f^{-1}(x_{n+1})}^{p^s_n\wedge p^u_n}\overset{\ve}{\approx}\Psi_{x_n}^{p^s_n\wedge p^u_n}$.

\medskip
\noindent
(GPO2) The very definitions of $p^s_n,p^u_n$ guarantee that
$p^s_n=\min\{e^\ve p^s_{n+1},\delta_\ve Q_\ve(x_n)\}$ and
$p^u_{n+1}=\min\{e^\ve p^u_n,\delta_\ve Q_\ve(x_{n+1})\}$.

\medskip
\noindent
{\sc Claim 4:} $\{\Psi_{x_n}^{p^s_n,p^u_n}\}_{n\in\Z}$ shadows $x$.

\medskip
By (${\rm a}_n$) with $i=0$, we have
$\Psi_{f^n(x)}^{p^s_n\wedge p^u_n}\overset{\ve}{\approx}\Psi_{x_n}^{p^s_n\wedge p^u_n}$, hence 
by Proposition \ref{Lemma-overlap}(3) we have $f^n(x)=\Psi_{f^n(x)}(0)\in \Psi_{x_n}(R[p^s_n\wedge p^u_n])$,
thus $\{\Psi_{x_n}^{p^s_n,p^u_n}\}_{n\in\Z}$ shadows $x$.

\medskip
This concludes the proof of sufficiency.

\medskip
\noindent
{\em Proof of relevance.} The alphabet $\mathfs A$ might not a priori satisfy
the relevance condition, but we can easily reduce it to a sub-alphabet $\mathfs A'$ satisfying (1)--(3).
Call $v\in\mathfs A$ relevant if there is $\un v\in\mathfs A^\Z$ with $v_0=v$ s.t. $\un{v}$ shadows
a point in ${\rm NUH}_\chi^*$. Since ${\rm NUH}_\chi^*$ is $f$--invariant, every $v_i$ is relevant.
Then $\mathfs A'=\{v\in\mathfs A:v\text{ is relevant}\}$ is discrete
because $\mathfs A'\subset\mathfs A$, it is sufficient and relevant by definition.
%because ${\rm NUH}_\chi^*\subset {\rm NUH}_\chi$,
%and it is relevant by definition.
\end{proof}

Let $\Sigma$ be the TMS associated to the graph with vertex set $\mathfs A$ given by
Theorem \ref{Thm-coarse-graining} and
edges $v\overset{\ve}{\to}w$. An element of $\Sigma$ is an $\ve$--gpo, hence
we define $\pi:\Sigma\to M$ by
$$
\{\pi[\{v_n\}_{n\in\Z}]\}:=V^s[\{v_n\}_{n\geq 0}]\cap V^u[\{v_n\}_{n\leq 0}].
$$
Here are the main properties of the triple $(\Sigma,\sigma,\pi)$.

\begin{proposition}\label{Prop-pi}
The following holds for all $\ve>0$ small enough.
\begin{enumerate}[{\rm (1)}]
\item Each $v\in\mathfs A$ has finite ingoing and outgoing degree, hence $\Sigma$ is locally compact.
\item $\pi:\Sigma\to M$ is H\"older continuous.
\item $\pi\circ\sigma=f\circ\pi$.
\item $\pi[\Sigma]\supset{\rm NUH}_\chi^*$.
\end{enumerate} 
\end{proposition}

Part (1) follows from (GPO2), part (2) follows from Proposition \ref{Prop-graph-transform},
part (3) is obvious, and part (4) follows from Theorem \ref{Thm-coarse-graining}(2).
It is important noting that $(\Sigma,\sigma,\pi)$ does {\em not} satisfy Theorem \ref{Thm-main},
since $\pi$ might be (and usually is) infinite-to-one. We use $\pi$ to induce a locally
finite cover of ${\rm NUH}_\chi^\#$, which will then be refined to a partition of ${\rm NUH}_\chi^\#$
that will lead to the proof of Theorem \ref{Thm-main}.

\section{The inverse problem}

Our goal is to analyze when $\pi$ loses injectivity. More specifically, given that
$\pi(\un{v})=\pi(\un{w})$ we want to compare $v_n$ and $w_n$, and show that they
are uniquely defined ``up to bounded error''. We do this under the additional assumption
that $\un{v},\un{w}\in\Sigma^\#$. Remind that $\Sigma^\#$ is the {\em recurrent set} of $\Sigma$:
$$
\Sigma^\#:=\left\{\un v\in\Sigma:\exists v,w\in V\text{ s.t. }\begin{array}{l}v_n=v\text{ for infinitely many }n>0\\
v_n=w\text{ for infinitely many }n<0
\end{array}\right\}.
$$
The main result is the following.

\begin{theorem}[Inverse theorem]\label{Thm-inverse}
The following holds for $\ve>0$ small enough.
If $\{\Psi_{x_n}^{p^s_n,p^u_n}\}_{n\in\Z},\{\Psi_{y_n}^{q^s_n,q^u_n}\}_{n\in\Z}\in\Sigma^\#$ satisfy
$\pi[\{\Psi_{x_n}^{p^s_n,p^u_n}\}_{n\in\Z}]=\pi[\{\Psi_{y_n}^{q^s_n,q^u_n}\}_{n\in\Z}]$ then:
\begin{enumerate}[{\rm (1)}]
\item $d(x_n,y_n)<25^{-1}\max\{p^s_n\wedge p^u_n,q^s_n\wedge q^u_n\}$.
\item $\tfrac{\sin\alpha(x_n)}{\sin\alpha(y_n)}=e^{\pm\sqrt{\ve}}$ and
$|\cos\alpha(x_n)-\cos\alpha(y_n)|<\sqrt{\ve}$.
\item $\tfrac{s(x_n)}{s(y_n)}=e^{\pm 4\sqrt{\ve}}$ and $\tfrac{u(x_n)}{u(y_n)}=e^{\pm 4\sqrt{\ve}}$.
\item $\tfrac{Q_\ve(x_n)}{Q_\ve(y_n)}=e^{\pm \sqrt[3]{\ve}}$.
\item $\tfrac{p^s_n}{q^s_n}=e^{\pm\sqrt[3]{\ve}}$ and $\tfrac{p^u_n}{q^u_n}=e^{\pm\sqrt[3]{\ve}}$.
\item $(\Psi_{y_n}^{-1}\circ\Psi_{x_n})(v)=(-1)^{\sigma_n}v+\delta_n+\Delta_n(v)$ for $v\in R[10Q_\ve(x_n)]$,
where $\sigma_n\in\{0,1\}$, $\delta_n$ is a vector with $\|\delta_n\|<10^{-1}(q^s_n\wedge q^u_n)$ and
$\Delta_n$ is a vector field s.t. $\Delta_n(0)=0$ and $\|d\Delta_n\|_0<\sqrt[3]{\ve}$ on $R[10Q_\ve(x_n)]$.
\end{enumerate}
\end{theorem}

The difference from Theorem \ref{Thm-inverse} to \cite[Thm 5.2]{Sarig-JAMS} is
that the estimate on our part (6) holds only in the smaller rectangle $R[10Q_\ve(x_n)]$. Part (1) is proved as
in \cite[Prop. 5.3]{Sarig-JAMS}. 
Here is one of its consequences. We have
$d(x_n,y_n)<25^{-1}(p^s_n\wedge p^u_n+q^s_n\wedge q^u_n)<\ve[d(x_n,\mathfs D)^a+d(y_n,\mathfs D)^a]$,
hence
$$
d(x_n,\mathfs D)=d(y_n,\mathfs D)\pm d(x_n,y_n)=d(y_n,\mathfs D)\pm\ve[d(x_n,\mathfs D)^a+d(y_n,\mathfs D)^a].
$$
These estimates have two consequences. The first is that
\begin{equation}\label{equation-distances}
\frac{1-\ve}{1+\ve}\leq \frac{d(x_n,\mathfs D)}{d(y_n,\mathfs D)}\leq \frac{1+\ve}{1-\ve}
\end{equation}
and so, for $\ve>0$ is sufficiently small, it holds
$\tfrac{1}{2}\leq \tfrac{d(x_n,\mathfs D)^a}{d(y_n,\mathfs D)^a}\leq 2$.
The second consequence is that $x_n\in D_{y_n}$ and $y_n\in D_{x_n}$, since
\begin{align*}
&d(x_n,y_n)<\ve[d(x_n,\mathfs D)^a+d(y_n,\mathfs D)^a]<3\ve\min\{d(x_n,\mathfs D)^a,d(y_n,\mathfs D)^a\}\\
&<\min\{\mathfrak r(x_n),\mathfrak r(y_n)\}.
\end{align*}
Therefore we can take parallel transport with respect to either $x_n$ or $y_n$.

\medskip
The proofs of parts (2)--(6) use, as in \cite{Sarig-JAMS}, some auxiliary facts about admissible manifolds. Let
$\un v^+=\{v_n\}_{n\geq 0}$ be a positive $\ve$--gpo with $v_n=\Psi_{x_n}^{p^s_n,p^u_n}$.
By Proposition \ref{Prop-stable-manifolds}, $V^s[\un v^+]$ has the following property:
$f^n(V^s[\un v^+])\subset V^s[\{v_k\}_{k\geq n}]\subset \Psi_{x_n}(R[10Q_\ve(x_n)])$.
This motivates the definition of {\em staying in windows} as in \cite{Sarig-JAMS}:
given an $\ve$--double chart, say that $V^s\in\mathfs M^s(v)$ stays in windows if
there exists a positive $\ve$--gpo $\un v^+$ with $v_0=v$ and $s$--admissible
manifolds $W^s_n\in\mathfs M^s(v_n)$ s.t. $f^n(V^s)\subset W^s_n$ for all $n\geq 0$.
In particular, every $V^s[\un v^+]$ stays in windows, and the reverse statement is also true.
An analogous definition holds for $u$--admissible manifolds. Given $V^s\in \mathfs M^s[v]$
and $x\in V^s$, let $e^s_x\in T_xM$ denote the positively oriented vector tangent to $V^s$ at $x$.

\begin{proposition}\label{Prop-stay-window}
The following holds for all $\ve>0$ small enough.
\begin{enumerate}[{\rm (1)}]
\item  If $V^s\in\mathfs M^s[\Psi_x^{p^s,p^u}]$ stays in windows then for all $y,z\in V^s$ and $n\geq 0$: 
\begin{enumerate}[{\rm (a)}]
\item $d(f^n(y),f^n(z))<6p^s e^{-\frac{\chi}{2}n}$.
\item $\|df^n_y e^s_y\|\leq 6\|C_\chi(x)^{-1}\|e^{-\frac{\chi}{2}n}$.
\item $|\log\|df^n_y e^s_y\|-\log\|df^n_z e^s_z\||<Q_\ve(x)^{\beta/4}$.
\end{enumerate}
\item If $V^s\in\mathfs M^s[\Psi_x^{p^s,p^u}], U^s\in\mathfs M^s[\Psi_x^{q^s,q^u}]$ stay in windows
then either $V^s\subset U^s$ or $U^s\subset V^s$.
\end{enumerate}
Analogous statements hold for $u$--admissible manifolds that stay in windows.
\end{proposition}

When $M$ is compact and $f$ is a $C^{1+\beta}$ diffeomorphism, this is \cite[Prop. 6.3 and 6.4]{Sarig-JAMS}.
The only adaptation we need to make is in part (1)(c), see Appendix B.
Because of part (1)(c), if $y,z\in V^s$ then $\tfrac{s(y)}{s(z)}=e^{\pm Q_\ve(x)^{\beta/4}}$,
therefore we can define $s(V^s):=s(\Psi_x(0,F(0)))$, where $F$ is the representing function of $V^s$.
Note that $s(V^s)$ might be infinite, in which case $s(y)$ is infinite for all $y\in V^s$.
A similar definition holds for $u$--admissible manifolds that stay in windows. 

\medskip
The proof of part (2) of Theorem \ref{Thm-inverse} is analogous to \cite[Prop. 6.5]{Sarig-JAMS}.
In the sequel we adapt the methods of \cite{Sarig-JAMS} to prove parts (3)--(6).

\subsection{Control of $s(x_n)$ and $u(x_n)$}
As in \cite{Sarig-JAMS}, the hyperbolicity of $f$ induces an improvement for $s$ and $u$.
Because of symmetry, we only state the result for $s$.

\begin{lemma}[Improvement lemma]\label{Lemma-improvement}
The following holds for all $\ve>0$ small enough. Let $v\overset{\ve}{\to}w$ with
$v=\Psi_x^{p^s,p^u},w=\Psi_y^{q^s,q^u}$, and assume $V^s\in\mathfs M^s[w]$ stays in windows.
\begin{enumerate}[{\rm (1)}]
\item If $s(V^s)<\infty$ then $s[\mathfs F^s_{v,w}(V^s)]<\infty$.
\item For $\xi\geq {\sqrt{\ve}}$, if $s(V^s)<\infty$ and $\tfrac{s(V^s)}{s(y)}=e^{\pm\xi}$
then $\tfrac{s(\mathfs F^s_{v,w}(V^s))}{s(x)}=e^{\pm(\xi-Q_\ve(x)^{\beta/4})}$.
\end{enumerate} 
\end{lemma}

Note that the ratio improves. 

\begin{proof}
When $M$ is compact and $f$ is a $C^{1+\beta}$ diffeomorphism,
this is \cite[Lemma 7.2]{Sarig-JAMS}, and the proof of part (1) is identical.
Part (2) requires some finer estimates.

\medskip
Let $F,G$ be the representing functions of $V^s,\mathfs F^s_{v,w}(V^s)$,
and let $q:=\Psi_y(0,F(0))$, $p:=\Psi_x(0,G(0))$. Then
$\tfrac{s(\mathfs F^s_{v,w}(V^s))}{s(x)}=\tfrac{s(p)}{s(x)}=\tfrac{s(p)}{s(f^{-1}(q))}\cdot\tfrac{s(f^{-1}(q))}{s(f^{-1}(y))}\cdot\tfrac{s(f^{-1}(y))}{s(x)}$. We have:
\begin{enumerate}[$\circ$]
\item $p,f^{-1}(q)\in \mathfs F^s_{v,w}(V^s)$, hence Proposition \ref{Prop-stay-window}(1)(c) implies
$\tfrac{s(p)}{s(f^{-1}(q))}=e^{\pm Q_\ve(x)^{\beta/4}}$.
\item Since $(p^s\wedge p^u)^3(q^s\wedge q^u)^3\ll Q_\ve(x)^{\beta/4}$,
Proposition \ref{Lemma-overlap}(1) implies $\tfrac{s(f^{-1}(y))}{s(x)}=e^{\pm Q_\ve(x)^{\beta/4}}$.
\end{enumerate}
Thus it is enough to show that $\tfrac{s(f^{-1}(q))}{s(f^{-1}(y))}=e^{\pm(\xi-3Q_\ve(x)^{\beta/4})}$.
We show one side of the inequality (the other is similar).
Note that this is the term that gives the improvement. As in \cite[pp. 375]{Sarig-JAMS}, we have
$$
\tfrac{s(f^{-1}(q))^2}{s(f^{-1}(y))^2}\leq
\underbrace{\left(\tfrac{2+e^{2\xi+2\chi}s(y)^2\|df e^s_{f^{-1}(y)}\|^2}{2+e^{2\chi}s(y)^2\|df e^s_{f^{-1}(y)}\|^2}\right)}_{= \text{ I}}\ 
\underbrace{{\rm exp}\left(2|\log \|df e^s_{f^{-1}(q)}\|-\log \|df e^s_{f^{-1}(y)}\||\right)}_{=\text{ II}}.
$$
We estimate I as in \cite[pp. 376]{Sarig-JAMS}: I $\leq e^{2\xi-7Q_\ve(x)^{\beta/4}}$.
Therefore it suffices to show that $\text{II}\leq e^{Q_\ve(x)^{\beta/4}}$.
Since $\|df e^s_{f^{-1}(z)}\|=\|df^{-1}e^s_z\|^{-1}$,
$\text{II}={\rm exp}(2|\log \|df^{-1}e^s_q\|-\log \|df^{-1}e^s_{y}\||)$, hence
by the claim in the proof of Proposition \ref{Prop-stay-window} (Appendix B):
\begin{equation}\label{estimate-II}
\log(\text{II})\leq2\mathfrak K\rho(y)^{-2a}[d(q,y)^\beta+\|e^s_q-P_{y,q}e^s_y\|].
\end{equation}
Since $q=\Psi_y(0,G(0))$ and $y=\Psi_y(0,0)$, Lemma \ref{Lemma-Pesin-chart}(1) implies that
$d(q,y)\leq 2|G(0)|\leq 500^{-1}(q^s\wedge q^u)\leq 500^{-1}e^\ve(p^s\wedge p^u)$, therefore
$d(q,y)<Q_\ve(x),Q_\ve(y)$. Hence for small $\ve>0$:
\begin{align*}
&2\mathfrak K\rho(y)^{-2a}d(q,y)^\beta\leq 2\mathfrak K\rho(y)^{-2a}Q_\ve(y)^{3\beta/4}Q_\ve(x)^{\beta/4}
\leq 2\mathfrak K\rho(y)^{-2a}Q_\ve(y)^{\beta/36}Q_\ve(x)^{\beta/4}\\
&\leq 2\mathfrak K\ve^{1/12}Q_\ve(x)^{\beta/4}<\tfrac{1}{2}Q_\ve(x)^{\beta/4}.
\end{align*}
To bound the second term of (\ref{estimate-II}), we first estimate $\sin\angle(e^s_q,P_{y,q}e^s_y)$.
Since $e^s_y$ is the unitary vector in the direction of
$d(\Psi_y)_0\colvec{1\\0}=d(\exp{y})_0\circ C_\chi(y)\colvec{1\\0}$
and $e^s_q$ is the unitary vector in the direction of
$d(\Psi_y)_{(0,G(0))}\colvec{1\\ G'(0)}=d(\exp{y})_{C_\chi(y)\colvec[.6]{0\\G(0)}}\circ C_\chi(y)\colvec{1\\ G'(0)}$,
the angles they define are the same. In other words, if
$$
A=\widetilde{d(\exp{y})_0\circ C_\chi(y)},B=\widetilde{d(\exp{y})_{C_\chi(y)\colvec[.6]{0\\G(0)}}\circ C_\chi(y)},
v_1=\colvec{1\\0},v_2=\colvec{1\\ G'(0)}
$$
then $\sin\angle(e^s_q,P_{y,q}e^s_y)=\sin\angle(Av_1,Bv_2)$. Using (\ref{gen-ineq-angles}) 
with $L=A$, $v=v_1$, $w=A^{-1}Bv_2$, we get
\begin{align*}
&|\sin\angle(Av_1,Bv_2)|\leq \|A\|\|A^{-1}\||\sin\angle(v_1,A^{-1}Bv_2)|\\
&\leq \|C_\chi(y)^{-1}\|[|\sin\angle(v_1,v_2)|+|\sin\angle(v_2,A^{-1}Bv_2)|].
\end{align*}
We have $|\sin\angle(v_1,v_2)|\leq |G'(0)|\leq \tfrac{1}{2}(q^s\wedge q^u)^{\beta/3}\leq
\tfrac{e^{\frac{\beta\ve}{3}}}{2}(p^s\wedge p^u)^{\beta/3}$, therefore for small $\ve>0$ it holds
$|\sin\angle(v_1,v_2)|\leq Q_\ve(x)^{\beta/3},Q_\ve(y)^{\beta/3}$. In particular
$|\sin\angle(v_1,v_2)|\leq Q_\ve(y)^{\beta/12}Q_\ve(x)^{\beta/4}$. Also, by (A3):
\begin{align*}
&\|A^{-1}B-{\rm Id}\|\leq \|A^{-1}\|\|A-B\|\leq
\|C_\chi(y)^{-1}\| \left\|\widetilde{d(\exp{y})_0}-\widetilde{d(\exp{y})_{C_\chi(y)\colvec[.6]{0\\ G(0)}}}\right\|\\
&\leq \|C_\chi(y)^{-1}\|\rho(y)^{-a}|G(0)|\leq  \|C_\chi(y)^{-1}\|\rho(y)^{-a}Q_\ve(y)^{1-\frac{\beta}{4}}Q_\ve(x)^{\beta/4}\\
&\leq Q_\ve(y)^{1-\frac{11\beta}{36}}Q_\ve(x)^{\beta/4}<\tfrac{1}{4}Q_\ve(y)^{\beta/12}Q_\ve(x)^{\beta/4}\ll 1.
\end{align*}
This implies that $v_2,A^{-1}Bv_2$ are almost unitary vectors, therefore
$$
|\sin\angle(v_2,A^{-1}Bv_2)|\leq 2\|v_2-A^{-1}Bv_2\|\leq 4\|A^{-1}B-{\rm Id}\|<Q_\ve(y)^{\beta/12}Q_\ve(x)^{\beta/4},$$
thus $|\sin\angle(e^s_q,P_{y,q}e^s_y)|<2\|C_\chi(y)^{-1}\|Q_\ve(y)^{\beta/12}Q_\ve(x)^{\beta/4}$.
Since $\|e^s_q\|=\|P_{y,q}e^s_y\|=1$ and the angle between them is small,
$\|e^s_q-P_{y,q}e^s_y\|\leq 2|\sin\angle(e^s_q,P_{y,q}e^s_y)|<4\|C_\chi(y)^{-1}\|Q_\ve(y)^{\beta/12}Q_\ve(x)^{\beta/4}$.
The conclusion is that for small $\ve>0$:
\begin{align*}
&2\mathfrak K\rho(y)^{-2a}\|e^s_q-P_{y,q}e^s_y\|\leq
8\mathfrak K\|C_\chi(y)^{-1}\|\rho(y)^{-2a}Q_\ve(y)^{\beta/12}Q_\ve(x)^{\beta/4}\\
&\leq 8\mathfrak K\|C_\chi(y)^{-1}\|Q_\ve(y)^{\beta/24}\rho(y)^{-2a}Q_\ve(y)^{\beta/36}Q_\ve(x)^{\beta/4}\\
&\leq 8\mathfrak K\ve^{5/24}Q_\ve(x)^{\beta/4}<\tfrac{1}{2}Q_\ve(x)^{\beta/4}.
\end{align*}
Hence (\ref{estimate-II}) implies that $\text{II}<e^{Q_\ve(x)^{\beta/4}}$.
\end{proof}

We are now ready to prove part (3) of Theorem \ref{Thm-inverse}.
\begin{proposition}
The following holds for all $\ve>0$ small enough.
If $\{\Psi_{x_n}^{p^s_n,p^u_n}\}_{n\in\Z}$, $\{\Psi_{y_n}^{q^s_n,q^u_n}\}_{n\in\Z}\in\Sigma^\#$ satisfy
$\pi[\{\Psi_{x_n}^{p^s_n,p^u_n}\}_{n\in\Z}]=\pi[\{\Psi_{y_n}^{q^s_n,q^u_n}\}_{n\in\Z}]$ then for all $n\in\Z$:
$$
\tfrac{s(x_n)}{s(y_n)}=e^{\pm 4\sqrt{\ve}}\text{ and }\tfrac{u(x_n)}{u(y_n)}=e^{\pm 4\sqrt{\ve}}.
$$
\end{proposition}

When $M$ is compact and $f$ is a $C^{1+\beta}$ diffeomorphism,
this is \cite[Prop. 7.3]{Sarig-JAMS}, and the proof is identical.
Let $\un v=\{\Psi_{x_n}^{p^s_n,p^u_n}\}_{n\in\Z}$ and $\un w=\{\Psi_{y_n}^{q^s_n,q^u_n}\}_{n\in\Z}$.
We sketch the proof for the first estimate:
\begin{enumerate}[$\circ$]
\item If $\pi(\un v)=x$ then $s(x)<\infty$: this follows from the relevance of $\mathfs A$ 
(Thm. \ref{Thm-coarse-graining}(3)).
\item Apply Lemma \ref{Lemma-improvement} along $\un v$ and the orbit of $x$: if
$v_n=v$ for infinitely many $n>0$, then the ratio improves at each of these indices.
The conclusion is that $\tfrac{s(V^s[\{v_k\}_{k\geq n}])}{s(x_n)}=e^{\pm\sqrt{\ve}}$, and
analogously $\tfrac{s(V^s[\{w_k\}_{k\geq n}])}{s(y_n)}=e^{\pm\sqrt{\ve}}$.
\item Since $f^n(x)\in V^s[\{v_k\}_{k\geq n}]\cap V^s[\{w_k\}_{k\geq n}]$, Proposition \ref{Prop-stay-window}(1)(c)
implies that $\tfrac{s(V^s[\{v_k\}_{k\geq n}])}{s(f^n(x))}=e^{\pm\sqrt{\ve}}$
and $\tfrac{s(V^s[\{w_k\}_{k\geq n}])}{s(f^n(x))}=e^{\pm\sqrt{\ve}}$.
\end{enumerate}
Hence $\tfrac{s(x_n)}{s(y_n)}=\tfrac{s(x_n)}{s(V^s[\{v_k\}_{k\geq n}])}\cdot\tfrac{s(V^s[\{v_k\}_{k\geq n}])}{s(f^n(x))}
\cdot\tfrac{s(f^n(x))}{s(V^s[\{w_k\}_{k\geq n}])}\cdot\tfrac{s(V^s[\{w_k\}_{k\geq n}])}{s(y_n)}=e^{\pm4\sqrt{\ve}}$.

\subsection{Control of $Q_\ve(x_n)$}

Remind that $Q_\ve(x):=\max\{q\in I_\ve:q\leq \widetilde Q_\ve(x)\}$ where
$$
\widetilde Q_\ve(x)=\ve^{3/\beta}
\min\left\{\|C_\chi(x)^{-1}\|_{\rm Frob}^{-24/\beta},\|C_\chi(f(x))^{-1}\|^{-12/\beta}_{\rm Frob}\rho(x)^{72a/\beta}\right\},
$$
so we first control $\widetilde Q_\ve(x_n)$.
By parts (2)--(3), $\tfrac{\|C_\chi(x_n)^{-1}\|_{\rm Frob}}{\|C_\chi(y_n)^{-1}\|_{\rm Frob}}=e^{\pm 5\sqrt{\ve}}$.
Using that $\Psi_{f(x_n)}^{p^s_{n+1}\wedge p^u_{n+1}}\overset{\ve}{\approx}\Psi_{x_{n+1}}^{p^s_{n+1}\wedge p^u_{n+1}}$,
Proposition \ref{Lemma-overlap}(1)--(2) implies that
$\tfrac{\|C_\chi(f(x_n))^{-1}\|_{\rm Frob}}{\|C_\chi(x_{n+1})^{-1}\|_{\rm Frob}}=e^{\pm\sqrt{\ve}}$,
and similarly $\tfrac{\|C_\chi(f(y_n))^{-1}\|_{\rm Frob}}{\|C_\chi(y_{n+1})^{-1}\|_{\rm Frob}}=e^{\pm\sqrt{\ve}}$.
Hence
$$\tfrac{\|C_\chi(f(x_n))^{-1}\|_{\rm Frob}}{\|C_\chi(f(y_n))^{-1}\|_{\rm Frob}}=
\tfrac{\|C_\chi(f(x_n))^{-1}\|_{\rm Frob}}{\|C_\chi(x_{n+1})^{-1}\|_{\rm Frob}}\cdot
\tfrac{\|C_\chi(x_{n+1})^{-1}\|_{\rm Frob}}{\|C_\chi(y_{n+1})^{-1}\|_{\rm Frob}}\cdot
\tfrac{\|C_\chi(y_{n+1})^{-1}\|_{\rm Frob}}{\|C_\chi(f(y_n))^{-1}\|_{\rm Frob}}=e^{\pm 7\sqrt{\ve}}.$$

\medskip
We now estimate the ratio $\tfrac{\rho(x_n)}{\rho(y_n)}$. For that we obtain estimates
similar to (\ref{equation-distances}) for $f^{\pm 1}(x_n),f^{\pm 1}(y_n)$. By symmetry,
we only need to get the inequalities for $f(x_n),f(y_n)$. Start by noting that
$d(f(x_n),x_{n+1})\leq (p^s_{n+1}\wedge p^u_{n+1})^8<\ve d(x_{n+1},\mathfs D)$, hence
$d(f(x_n),\mathfs D)=d(x_{n+1},\mathfs D)\pm d(f(x_n),x_{n+1})=(1\pm\ve)d(x_{n+1},\mathfs D)$
and thus $d(f(x_n),x_{n+1})<2\ve d(f(x_n),\mathfs D)$. Similarly $d(f(y_n),y_{n+1})<2\ve d(f(y_n),\mathfs D)$.
Using part (1), $d(x_{n+1},y_{n+1})<\ve[d(x_{n+1},\mathfs D)+d(y_{n+1},\mathfs D)]<
2\ve[d(f(x_n),\mathfs D)+d(f(y_n),\mathfs D)]$, therefore
\begin{align*}
d(f(x_n),f(y_n))&\leq d(f(x_n),x_{n+1})+d(x_{n+1},y_{n+1})+d(y_{n+1},f(y_n))\\
&<4\ve[d(f(x_n),\mathfs D)+d(f(y_n),\mathfs D)].
\end{align*}
This implies that $d(f(x_n),\mathfs D)=d(f(y_n),\mathfs D)\pm 4\ve[d(f(x_n),\mathfs D)+d(f(y_n),\mathfs D)]$
and so
$\tfrac{1-4\ve}{1+4\ve}\leq \tfrac{d(f(x_n),\mathfs D)}{d(f(y_n),\mathfs D)}\leq\tfrac{1+4\ve}{1-4\ve}$.
The same estimate holds for $f^{-1}$. Together with (\ref{equation-distances}), we get that
$\tfrac{1-4\ve}{1+4\ve}\leq\tfrac{\rho(x_n)}{\rho(y_n)}\leq\tfrac{1+4\ve}{1-4\ve}$.
If $\ve>0$ is small enough then
$e^{-\sqrt{\ve}}<\left(\tfrac{1-4\ve}{1+4\ve}\right)^{\frac{72a}{\beta}}
<\left(\tfrac{1+4\ve}{1-4\ve}\right)^{\frac{72a}{\beta}}<e^{\sqrt{\ve}}$, hence
$\tfrac{\rho(x_n)^{72a/\beta}}{\rho(y_n)^{72a/\beta}}=e^{\pm\sqrt{\ve}}$.
The conclusion is that
$\tfrac{\widetilde Q_\ve(x_n)}{\widetilde Q_\ve(y_n)}={\rm exp}[\pm(\tfrac{120}{\beta}\sqrt{\ve})]$,
which implies that $\tfrac{Q_\ve(x_n)}{Q_\ve(y_n)}={\rm exp}[\pm(\frac{2}{3}\ve+\tfrac{120}{\beta}\sqrt{\ve})]$.
Hence if $\ve>0$ is small enough it holds $\tfrac{Q_\ve(x_n)}{Q_\ve(y_n)}=e^{\pm\sqrt[3]{\ve}}$.

\subsection{Control of $p^s_n$ and $p^u_n$}

As in \cite[Prop. 8.3]{Sarig-JAMS}, (GPO2) implies the lemma below.

\begin{lemma}\label{Lemma-maximality}
If $\un v=\{\Psi_{x_n}^{p^s_n,p^u_n}\}_{n\in\Z}\in\Sigma^\#$ then 
$p^s_n=\delta_\ve Q_\ve(x_n)$ for infinitely many $n>0$ and
$p^u_n=\delta_\ve Q_\ve(x_n)$ for infinitely many $n<0$.
\end{lemma}

We now prove the first half of part (5) (the other half is analogous).
By symmetry, it is enough to prove that $p^s_n\geq e^{-\sqrt[3]{\ve}}q^s_n$ for all $n\in\Z$.
\begin{enumerate}[$\circ$]
\item If $p^s_n=\delta_\ve Q_\ve(x_n)$ then part (4) gives
$p^s_n=\delta_\ve Q_\ve(x_n)\geq e^{-\sqrt[3]{\ve}}\delta_\ve Q_\ve(y_n)
\geq e^{-\sqrt[3]{\ve}}q^s_n$.
\item If $p^s_n\geq e^{-\sqrt[3]{\ve}}q^s_n$ then (GPO2) and part (4) give:
$$p^s_{n-1}=\min\{e^\ve p^s_n,\delta_\ve Q_\ve(x_{n-1})\}\geq
e^{-\sqrt[3]{\ve}}\min\{e^\ve q^s_n,\delta_\ve Q_\ve(y_{n-1})\}=e^{-\sqrt[3]{\ve}}q^s_{n-1}.$$
%Using induction, $p^s_m\geq e^{-\sqrt[3]{\ve}}q^s_m$ for all $m\leq n$.
\end{enumerate}
By Lemma \ref{Lemma-maximality}, it follows that $p^s_n\geq e^{-\sqrt[3]{\ve}}q^s_n$ for all $n\in\Z$.

\subsection{Control of $\Psi_{y_n}^{-1}\circ\Psi_{x_n}$}

%We finish this section with the proof of part (6).
For $z_n=x_n,y_n$, the calculations in the
proof of Lemma \ref{Lemma-linear-reduction} give that
$$
\widetilde{C_\chi(z_n)}=R_{z_i}\left[\begin{array}{cc}\tfrac{1}{s(z_n)}& \tfrac{\cos\alpha(z_n)}{u(z_n)}\\
0 & \tfrac{\sin\alpha(z_n)}{u(z_n)}\end{array}\right]
$$
where $R_{z_n}$ is the rotation that takes $e_1$ to $\iota_{z_n}e^s_{z_n}$.

\begin{lemma}\label{Lemma-rotations}
Under the conditions of Theorem \ref{Thm-inverse}, for all $n\in\Z$ it holds
$$
R_{y_n}^{-1}R_{x_n}=(-1)^{\sigma_n}{\rm Id}+
\left[\begin{array}{cc}\ve_{11}&\ve_{12}\\ \ve_{21}&\ve_{22}\end{array}\right]
$$
where $\sigma_n\in\{0,1\}$ and $|\ve_{jk}|<(p^s_n\wedge p^u_n)^{\beta/5}+(q^s_n\wedge q^u_n)^{\beta/5}<\sqrt{\ve}$.
\end{lemma}

When $M$ is compact and $f$ is a $C^{1+\beta}$ diffeomorphism,
this is \cite[Prop. 6.7]{Sarig-JAMS}. See Apendix B for the proof in our context.

\medskip
Now we establish part (6). It is enough to prove the case $n=0$.
Write $\Psi_{x_0}^{p^s_0,p^u_0}=\Psi_{x}^{p^s,p^u}$,
$\Psi_{y_0}^{q^s_0,q^u_0}=\Psi_{y}^{q^s,q^u}$, $p=p^s\wedge p^u$, $q=q^s\wedge q^u$,
$\sigma=\sigma_0$.
Write $\widetilde{C_\chi(x)}=R_xC_x$, $\widetilde{C_\chi(y)}=R_yC_y$.
As in \cite[\S9]{Sarig-JAMS}, Lemma \ref{Lemma-rotations} gives
$\|C_y^{-1}C_x-(-1)^\sigma{\rm Id}\|<14\sqrt{\ve}$ and hence for small $\ve>0$:
\begin{align*}
&\|\widetilde{C_\chi(x)}-\widetilde{C_\chi(y)}\|\leq \|R_xC_x-(-1)^\sigma R_xC_y\|+\|R_xC_y-(-1)^\sigma R_yC_y\|\\
&\leq \|C_y^{-1}\|\|C_y^{-1}C_x-(-1)^\sigma{\rm Id}\|+\|R_y^{-1}R_x-(-1)^\sigma{\rm Id}\|<
16\sqrt{\ve}\|C_y^{-1}\|<\|C_y^{-1}\|.
\end{align*}
We use this to show that $\Psi_y^{-1}\circ\Psi_x$ is well-defined in $R[10Q_\ve(x)]$.
The argument is very similar to the proof of Proposition \ref{Lemma-overlap}(3).
For $v\in R[10Q_\ve(x)]$, (A2) and part (4) imply that for small $\ve>0$:
\begin{align*}
&d(\Psi_x(v),\Psi_y(v))\leq 2\Sas(C_\chi(x)v,C_\chi(y)v)\leq 4(d(x,y)+\|\widetilde{C_\chi(x)}-\widetilde{C_\chi(y)}\|\|v\|)\\
&< 4(q+\|C_y^{-1}\|\|v\|)<100\|C_y^{-1}\|Q_\ve(y).
\end{align*}
hence $\Psi_x(v)\in B(\Psi_y(v),100\|C_y^{-1}\|Q_\ve(y))\subset \Psi_y[B]$ where
$B\subset\R^2$ is the ball with center $v$ and radius $200\|C_y^{-1}\|^2Q_\ve(y)$.
If $\ve>0$ is small then for $w\in B$ we have
\begin{align*}
&\|w\|\leq \|v\|+200\|C_y^{-1}\|^2Q_\ve(y)<20Q_\ve(y)+200\ve^{1/4}Q_\ve(y)^{1-\beta/12}\\
&<20\ve^{3/\beta}d(y,\mathfs D)^a+200\ve^{1/4}d(y,\mathfs D)^a<d(y,\mathfs D)^a<2\mathfrak r(y),
\end{align*}
therefore $\Psi_y^{-1}\circ\Psi_x$ is well-defined in $R[10Q_\ve(x)]$.

\medskip
It remains to estimate $\Psi_y^{-1}\circ\Psi_x-(-1)^\sigma{\rm Id}$. Write
$\Psi_y^{-1}\circ\Psi_x=(-1)^\sigma{\rm Id}+\delta+\Delta$, where $\delta\in\R^2$ is a constant vector
and $\Delta:R[10Q_\ve(x)]\to\R^2$. Let $v\in R[10Q_\ve(x)]$.
Proceeding as in \cite[pp. 382]{Sarig-JAMS} and applying (A4) we get for small $\ve>0$ that:
\begin{align*}
&\|d(\Delta)_v\|\leq 2\|C_y^{-1}\|d(y,\mathfs D)^{-a}d(x,y)+14\sqrt{\ve}
<2\|C_y^{-1}\|d(y,\mathfs D)^{-a}Q_\ve(y)+14\sqrt{\ve}\\
&<2\sqrt{\ve}\|C_y^{-1}\|Q_\ve(y)^{\beta/24}d(y,\mathfs D)^{-a}Q_\ve(y)^{\beta/72}+14\sqrt{\ve}<16\sqrt{\ve}<\sqrt[3]{\ve}.
\end{align*}
The estimate of $\|\delta\|$ is identical to \cite[pp. 383]{Sarig-JAMS}. This completes the proof
of part (6), and hence of Theorem \ref{Thm-inverse}.

\section{Symbolic dynamics}

\subsection{A countable Markov partition}

Let $(\Sigma,\sigma)$ be the TMS
constructed in Theorem \ref{Thm-coarse-graining}, and let
$\pi:\Sigma\to M$ as defined in the end of section \ref{Section-coarse-graining}.
In the sequel we use Theorem \ref{Thm-inverse} to construct a cover of ${\rm NUH}_\chi^\#$
that is locally finite and satisfies a (symbolic) Markov property.

\medskip
\noindent
{\sc The Markov cover $\mathfs Z$:} Let $\mathfs Z:=\{Z(v):v\in\mathfs A\}$, where
$$
Z(v):=\{\pi(\un v):\un v\in\Sigma^\#\text{ and }v_0=v\}.
$$

\medskip
In other words, $\mathfs Z$ is the family defined by the natural partition of $\Sigma^\#$ into
cylinder at the zeroth position. Admissible manifolds allow us to
define {\em invariant fibres} inside each $Z\in\mathfs Z$. Let $Z=Z(v)$.

\medskip
\noindent
{\sc $s$/$u$--fibres in $\mathfs Z$:} Given $x\in Z$, let $W^s(x,Z):=V^s[\{v_n\}_{n\geq 0}]\cap Z$
be the {\em $s$--fibre} of $x$ in $Z$ for some (any) $\un v=\{v_n\}_{n\in\Z}\in\Sigma^\#$
s.t. $\pi(\un v)=x$ and $v_0=v$. Similarly, let $W^u(x,Z):=V^u[\{v_n\}_{n\leq 0}]\cap Z$ be
the {\em $u$--fibre} of $x$ in $Z$.

\medskip
By Proposition \ref{Prop-stay-window}(2), the definitions above do not depend on the choice of $\un v$, 
and any two $s$--fibres ($u$--fibres) either coincide or are disjoint. We also
define $V^s(x,Z):=V^s[\{v_n\}_{n\geq 0}]$ and $V^u(x,Z):=V^u[\{v_n\}_{n\leq 0}]$.
Below we collect the main properties of $\mathfs Z$.

\begin{proposition}\label{Prop-Z}
The following are true.
\begin{enumerate}[{\rm (1)}]
\item {\sc Covering property:} $\mathfs Z$ is a cover of ${\rm NUH}_\chi^\#$.
\item {\sc Local finiteness:} For every $Z\in\mathfs Z$, $\#\{Z'\in\mathfs Z:Z\cap Z'\neq\emptyset\}<\infty$.
\item {\sc Product structure:} For every $Z\in\mathfs Z$ and every $x,y\in Z$, the intersection
$W^s(x,Z)\cap W^u(y,Z)$ consists of a single point of $Z$.
\item {\sc Symbolic Markov property:} If $x=\pi(\un v)$ with $\un v\in\Sigma^\#$, then
$$
f(W^s(x,Z(v_0)))\subset W^s(f(x),Z(v_1))\, \text{ and }\, f^{-1}(W^u(f(x),Z(v_1)))\subset W^u(x,Z(v_0)).
$$
\end{enumerate}
\end{proposition}

Part (1) follows from Theorem \ref{Thm-coarse-graining}(2),
part (2) follows from Theorem \ref{Thm-inverse}(5), part (3) follows from
Lemma \ref{Lemma-admissible-manifolds}(1), and part (4) is proved as in \cite[Prop. 10.9]{Sarig-JAMS}.
For $x,y\in Z$, let $[x,y]_Z:=$ intersection point of $W^s(x,Z)$ and $W^u(y,Z)$, and
call it the {\em Smale bracket} of $x,y$ in $Z$.

\begin{lemma}
The following holds for all $\ve>0$ small enough.
\begin{enumerate}[{\rm (1)}]
\item {\sc Compatibility:} If $x,y\in Z(v_0)$ and $f(x),f(y)\in Z(v_1)$ with
$v_0\overset{\ve}{\to} v_1$ then $f([x,y]_{Z(v_0)})=[f(x),f(y)]_{Z(v_1)}$.
\item {\sc Overlapping charts properties:} If $Z=Z(\Psi_x^{p^s,p^u}),Z'=Z(\Psi_y^{q^s,q^u})\in\mathfs Z$
with $Z\cap Z'\neq \emptyset$ then:
\begin{enumerate}[{\rm (a)}]
\item $Z\subset \Psi_y(R[q^s\wedge q^u])$.
\item If $x\in Z\cap Z'$ then $W^{s/u}(x,Z)\subset V^{s/u}(x,Z')$. 
\item If $x\in Z,y\in Z'$ then $V^s(x,Z)$ and $V^u(y,Z')$ intersect at a unique point. 
\end{enumerate}
\end{enumerate}
\end{lemma}

When $M$ is compact and $f$ is a diffeomorphism, part (1) is \cite[Lemma 10.7]{Sarig-JAMS}
and part (2) is \cite[Lemmas 10.8 and 10.10]{Sarig-JAMS}. The same proofs work in our case,
since all calculations are made in the rectangle $R[10Q_\ve(x)]$, and in this domain
we have Theorem \ref{Thm-inverse}(6). 

\medskip
Now we apply a refinement method to destroy non-trivial intersections in $\mathfs Z$. 
The result is a partition of ${\rm NUH}_\chi^\#$ with the (geometrical) Markov property.
This idea, originally developed by Sina{\u\i} and Bowen
for finite covers \cite{Sinai-Construction-of-MP,Sinai-MP-U-diffeomorphisms,Bowen-LNM},
works equally well for countable covers with the local finiteness property \cite{Sarig-JAMS}.
Write $\mathfs Z=\{Z_1,Z_2,\ldots\}$.

\medskip
\noindent
{\sc The Markov partition $\mathfs R$:} For every $Z_i,Z_j\in\mathfs Z$, define a partition of $Z_i$ by:
\begin{align*}
T_{ij}^{su}&=\{x\in Z_i: W^s(x,Z_i)\cap Z_j\neq\emptyset,
W^u(x,Z_i)\cap Z_j\neq\emptyset\}\\
T_{ij}^{s\emptyset}&=\{x\in Z_i: W^s(x,Z_i)\cap Z_j\neq\emptyset,
W^u(x,Z_i)\cap Z_j=\emptyset\}\\
T_{ij}^{\emptyset u}&=\{x\in Z_i: W^s(x,Z_i)\cap Z_j=\emptyset,
W^u(x,Z_i)\cap Z_j\neq\emptyset\}\\
T_{ij}^{\emptyset\emptyset}&=\{x\in Z_i: W^s(x,Z_i)\cap Z_j=\emptyset,
W^u(x,Z_i)\cap Z_j=\emptyset\}.
\end{align*}
Let $\mathfs T:=\{T_{ij}^{\alpha\beta}:i,j\geq 1,\alpha\in\{s,\emptyset\},\beta\in\{u,\emptyset\}\}$,
and let $\mathfs R$ be the partition generated by $\mathfs T$.

%\medskip
%The atom of $\mathfs R$ containing $x$ is $R(x)=\bigcap_{T\in\mathfs T\atop{T\ni x}}T$.
%The equivalence relation defining $\mathfs R$ is:
%\begin{align*}
%x\sim y\iff \forall Z,Z'\in\mathfs Z:\left(\begin{array}{c}
%x\in Z\Leftrightarrow y\in Z\\
%W^s(x,Z)\cap Z'\neq\emptyset\Leftrightarrow W^s(y,Z)\cap Z'\neq\emptyset\\
%W^u(x,Z)\cap Z'\neq\emptyset\Leftrightarrow W^u(y,Z)\cap Z'\neq\emptyset
%\end{array}\right).
%\end{align*}

%If $Z_i\cap Z_j=\emptyset$ then the only non-empty element of $\mathfs T_{ij}$ is $T_{ij}^{s\emptyset}=Z_i$. 
\medskip
Since $T_{ii}^{su}=Z_i$, $\mathfs R$ is a partition of ${\rm NUH}_\chi^\#$.
Clearly, $\mathfs R$ is a refinement of $\mathfs Z$. Theorem \ref{Thm-inverse}
implies two local finiteness properties for $\mathfs R$:
\begin{enumerate}[$\circ$]
\item For every $Z\in\mathfs Z$, $\#\{R\in\mathfs R:R\subset Z\}<\infty$.
\item For every $R\in\mathfs R$, $\#\{Z\in\mathfs Z:Z\supset R\}<\infty$.
\end{enumerate}

\medskip
Now we show that $\mathfs R$ is a Markov partition in the sense of Sina{\u\i} \cite{Sinai-MP-U-diffeomorphisms}. 

\medskip
\noindent
{\sc $s$/$u$--fibres in $\mathfs R$:} Given $x\in R\in\mathfs R$, we define the {\em $s$--fibre}
and {\em $u$--fibre} of $x$ by:
\begin{align*}
W^s(x,R):=
\bigcap_{T_{ij}^{\alpha\beta}\in\mathfs T\atop{T_{ij}^{\alpha\beta}\supset R}} W^s(x,Z_i)\cap T_{ij}^{\alpha\beta}
\, \text{ and }\, W^u(x,R):=
\bigcap_{T_{ij}^{\alpha\beta}\in\mathfs T\atop{T_{ij}^{\alpha\beta}\supset R}}  W^u(x,Z_i)\cap T_{ij}^{\alpha\beta}.
\end{align*}

Any two $s$--fibres ($u$--fibres) either coincide or are disjoint.

\begin{proposition}\label{Prop-R}
The following are true.
\begin{enumerate}[{\rm (1)}]
\item {\sc Product structure:} For every $R\in\mathfs R$ and every $x,y\in R$, the intersection
$W^s(x,R)\cap W^u(y,R)$ consists of a single point of $R$. Denote it by $[x,y]$.
\item {\sc Hyperbolicity:} If $z,w\in W^s(x,R)$ then $d(f^n(z),f^n(w))\xrightarrow[n\to\infty]{}0$, and
if $z,w\in W^u(x,R)$ then $d(f^n(z),f^n(w))\xrightarrow[n\to-\infty]{}0$. The rates are exponential.
\item {\sc Geometrical Markov property:} Let $R_0,R_1\in\mathfs R$. If $x\in R_0$ and $f(x)\in R_1$ then 
$$
f(W^s(x,R_0))\subset W^s(f(x),R_1)\, \text{ and }\, f^{-1}(W^u(f(x),R_1))\subset W^u(x,R_0).
$$
\end{enumerate}
\end{proposition}

When $M$ is compact and $f$ is a diffeomorphism, this is \cite[Prop. 11.5 and 11.7]{Sarig-JAMS}
and the same proof works in our case.

\subsection{A finite-to-one Markov extension}

%For $R,S\in\mathfs R$, draw an edge from $R$ to $S$ if $f(R)\cap S\neq\emptyset$.
We construct a new symbolic coding of $f$.
Let $\widehat{\mathfs G}=(\widehat V,\widehat E)$ be the oriented graph with vertex set
$\widehat V=\mathfs R$ and edge set $\widehat E=\{R\to S:R,S\in\mathfs R\text{ s.t. }f(R)\cap S\neq\emptyset\}$,
and let $(\widehat\Sigma,\widehat\sigma)$ be the TMS induced by $\widehat{\mathfs G}$.
The ingoing and outgoing degree of every vertex in $\widehat\Sigma$ is finite.

\medskip
For $\ell\in\Z$ and a path $R_m\to\cdots\to R_n$ on $\widehat{\mathfs G}$ define
$_\ell[R_m,\ldots,R_n]:=f^{-\ell}(R_m)\cap\cdots\cap f^{-\ell-(n-m)}(R_n)$, the set of points whose itinerary
from $\ell$ to $\ell+(n-m)$ visits the rectangles $R_m,\ldots,R_n$. The crucial property that
gives the new coding is that $_\ell[R_m,\ldots,R_n]\neq\emptyset$. This follows by induction, using the
Markov property of $\mathfs R$ (Proposition \ref{Prop-R}(3)).

\medskip
The map $\pi$ defines similar sets: for $\ell\in\Z$ and a path
$v_m\overset{\ve}{\to}\cdots\overset{\ve}{\to}v_n$ on $\Sigma$ let
$
Z_\ell[v_m,\ldots,v_n]:=\{\pi(\un w):\un w\in\Sigma^\#\text{ and }w_\ell=v_m,\ldots,w_{\ell+(n-m)}=v_n\}$.
There is a relation between $\Sigma$ and $\widehat\Sigma$ in terms of these sets:
if $\{R_n\}_{n\in\Z}\in\widehat\Sigma$ then there exists $\{v_n\}_{n\in\Z}\in\Sigma$
s.t. $_{-n}[R_{-n},\ldots,R_n]\subset Z_{-n}[v_{-n},\ldots,v_n]$ for all $n\geq 0$ (in particular $R_n\subset Z(v_n)$
for all $n\in\Z$). This fact is proved as in \cite[Lemma 12.2]{Sarig-JAMS}.
By Proposition \ref{Prop-R}(2), $\bigcap_{n\geq 0}\overline{_{-n}[R_{-n},\ldots,R_n]}$
is the intersection of a descending chain of nonempty closed sets with
diameters converging to zero.

\medskip
\noindent
{\sc The map $\widehat\pi:\widehat\Sigma\to M$:} Given $\un R=\{R_n\}_{n\in\Z}\in\widehat\Sigma$,
$\widehat\pi(\un R)$ is defined by the identity
$$
\{\widehat\pi(\un R)\}:=\bigcap_{n\geq 0}\overline{_{-n}[R_{-n},\ldots,R_n]}.
$$

\medskip
The triple $(\widehat\Sigma,\widehat\sigma,\widehat\pi)$ is the one that satisfies Theorem \ref{Thm-main}.

\begin{theorem}\label{Thm-widehat-pi}
The following holds for all $\ve>0$ small enough.
\begin{enumerate}[{\rm (1)}]
\item $\widehat\pi:\widehat\Sigma\to M$ is H\"older continuous.
\item $\widehat\pi\circ\widehat\sigma=f\circ\widehat\pi$.
\item $\widehat\pi[\widehat\Sigma^\#]\supset {\rm NUH}_\chi^\#$, hence
$\pi[\widehat\Sigma^\#]$ carries all $f$--adapted $\chi$--hyperbolic measures. 
\item Every point of $\widehat\pi[\widehat\Sigma^\#]$ has finitely many pre-images in $\widehat\Sigma^\#$.
\end{enumerate}
\end{theorem}

When $M$ is compact and $f$ is a diffeomorphism, parts (1)--(3) are \cite[Thm. 12.5]{Sarig-JAMS}
and part (4) is \cite[Thm. 5.6(5)]{Lima-Sarig}.
%\footnote{The statement of \cite[Thm. 12.8]{Sarig-JAMS}
%has a small gap because it does not restrict attention to the recurrent set $\widehat\Sigma^\#$.
%This gap was fixed in \cite[Thm. 5.6(5)]{Lima-Sarig}, and it is this proof applies to us.}
The same proofs work in our case, and the bound
on the number of pre-images is exactly the same: there is a function 
$N:\mathfs R\to\N$ s.t. if $x=\widehat\pi(\un R)$ with $R_n=R$ for infinitely many $n>0$ and $R_n=S$
for infinitely many $n<0$ then $\#\{\un S\in\widehat\Sigma^\#:\widehat\pi(\un S)=x\}\leq N(R)N(S)$.

\section*{Appendix A: Underlying assumptions}

Remember the definition of $\widetilde{A}\in\mathfs L_{x,x'}$ for $A\in\mathfs L_{y,z}$ and
$y\in D_x,z\in D_{x'}$. Remember also the definition of $\tau=\tau_x:D_x\times D_x\to \mathfs L_x$
by $\tau(y,z)=\widetilde{d(\exp{y}^{-1})_z}$.
Throughout the text, we assume that there are constants $\mathfrak K,a>1$ s.t. for all
$x\in M\backslash\mathfs D$ there is $d(x,\mathfs D)^a<\mathfrak r(x)<1$ 
s.t. for $D_x:=B(x,2\mathfrak r(x))$ it holds:
\begin{enumerate}[ii]
\item[(A1)] If $y\in D_x$ then $\inj(y)\geq 2\mathfrak r(x)$, $\exp{y}^{-1}:D_x\to T_yM$
is a diffeomorphism onto its image, and
$\tfrac{1}{2}(d(x,y)+\|v-P_{y,x}w\|)\leq \Sas(v,w)\leq 2(d(x,y)+\|v-P_{y,x} w\|)$ for all $y\in D_x$ and
$v\in T_xM,w\in T_yM$ s.t. $\|v\|,\|w\|\leq 2\mathfrak r(x)$, where 	
$P_{y,x}:=P_\gamma$ is the radial geodesic $\gamma$ joining $y$ to $x$.
\item[(A2)] If $y_1,y_2\in D_x$ then
$d(\exp{y_1}v_1,\exp{y_2}v_2)\leq 2\Sas(v_1,v_2)$ for $\|v_1\|$, $\|v_2\|\leq 2\mathfrak r(x)$,
and $\Sas(\exp{y_1}^{-1}z_1,\exp{y_2}^{-1}z_2)\leq 2[d(y_1,y_2)+d(z_1,z_2)]$
for $z_1,z_2\in D_x$ where the expression makes sense.
In particular $\|d(\exp{x})_v\|\leq 2$ for $\|v\|\leq 2\mathfrak r(x)$,
and $\|d(\exp{x}^{-1})_y\|\leq 2$ for $y\in D_x$.
\item[(A3)] If $y_1,y_2\in D_x$ then
$$
\|\widetilde{d(\exp{y_1})_{v_1}}-\widetilde{d(\exp{y_2})_{v_2}}\|
\leq d(x,\mathfs D)^{-a}\Sas(v_1,v_2)\leq \rho(x)^{-a}\Sas(v_1,v_2)
$$
for all $\|v_1\|,\|v_2\|\leq 2\mathfrak r(x)$ and 
\begin{align*}
\|\tau(y_1,z_1)-\tau(y_2,z_2)\|&\leq d(x,\mathfs D)^{-a}[d(y_1,y_2)+d(z_1,z_2)]\\
&\leq \rho(x)^{-a}[d(y_1,y_2)+d(z_1,z_2)]
\end{align*}
for all $z_1,z_2\in D_x$.
\item[(A4)] If $y_1,y_2\in D_x$ then the map $\tau(y_1,\cdot)-\tau(y_2,\cdot):D_x\to \mathfs L_x$
has Lipschitz constant $\leq d(x,\mathfs D)^{-a}d(y_1,y_2)\leq \rho(x)^{-a}d(y_1,y_2)$.
\item[(A5)] If $y\in D_x$ then $\|df_y^{\pm 1}\|\leq d(x,\mathfs D)^{-a}\leq \rho(x)^{-a}$.
\item[(A6)] If $y_1,y_2\in D_x$ and $f(y_1),f(y_2)\in D_{x'}$ then
$\|\widetilde{df_{y_1}}-\widetilde{df_{y_2}}\|\leq \mathfrak Kd(y_1,y_2)^\beta$,
and if $y_1,y_2\in D_x$ and $f^{-1}(y_1),f^{-1}(y_2)\in D_{x''}$ then
$\|\widetilde{df_{y_1}^{-1}}-\widetilde{df_{y_2}^{-1}}\|\leq \mathfrak Kd(y_1,y_2)^\beta$.
\item[(A7)] $\|df^{\pm 1}_x\|\geq m(df^{\pm 1}_x)\geq \rho(x)^a$.
\end{enumerate}

\section*{Appendix B: Standard proofs and adaptations of \cite{Sarig-JAMS}}\label{Appendix-standard-proofs}

In this appendix we prove some statements claimed throughout the text, most of them consisting
of adaptations of proofs in \cite{Sarig-JAMS}. The main issue is the lack of higher
regularity of the exponential map. The results of \cite{Sarig-JAMS} are technical but extremely
well-written, so rewriting it to our context would probably increase the technicalities and decrease
the clarity. Hence we decided to write this appendix as a tutorial:
we follow the proofs of \cite{Sarig-JAMS} as most as possible, mentioning the necessary changes. 
The main changes are in the geometrical estimates on $M$:
some Lipschitz constants of \cite{Sarig-JAMS} are substituted by terms of
the form $d(x,\mathfs D)^{-a}$. We then show that our
definition of $Q_\ve(x)$ is strong enough to cancel out these terms.
Since the proofs of \cite{Sarig-JAMS} have freedom in the choice of exponents,
we obtain the same final results and therefore (almost always) the same statements of \cite{Sarig-JAMS}.

\begin{proof}[Proof of Lemma \ref{Lemma-admissible-manifolds}.]
Part (1) is proved exactly as in \cite[Prop. 4.11(1)--(2)]{Sarig-JAMS}.
We concentrate on part (2). Let $\eta=p^s\wedge p^u$.
The estimate of $\tfrac{\sin\angle(V^s,V^u)}{\sin\alpha(x)}$ in \cite{Sarig-JAMS} is
divided into the analysis of four factors. The estimate of the first
two factors is identical; the difference is in the estimates of the remaining two factors.

\medskip
By (A3), if $x\in M\backslash\mathfs D$ and $\|v\|\leq 2\mathfrak r(x)$ then
$|{\rm det}[d(\exp{x})_v]-1|\leq 4d(x,\mathfs D)^{-a}\|v\|$,
i.e. we substitute $K_1$ in \cite[pp. 407]{Sarig-JAMS} by $4d(x,\mathfs D)^{-a}$.
With this notation,
$K_1\eta<4d(x,\mathfs D)^{-a}Q_\ve(x)^{\beta/72}\eta^{1-\beta/72}<4\ve^{1/24}\eta^{1-\beta/72}<\eta^{2\beta/3}$
for $\ve>0$ small, then the third factor is $e^{\pm 2\eta^{2\beta/3}}$. To estimate the fourth
factor, note that again by (A3) if $x\in M\backslash\mathfs D$ and $\|v\|\leq 2\mathfrak r(x)$
then $\|\widetilde{d(\exp{x})_v}-{\rm Id}\|\leq d(x,\mathfs D)^{-a}\|v\|$,
i.e. we substitute $K_2$ in \cite[pp. 407]{Sarig-JAMS} by $d(x,\mathfs D)^{-a}$. Noting as above
that $3K_2\eta<\eta^{2\beta/3}$, we get that the fourth factor is $e^{\pm\tfrac{1}{3}\eta^{\beta/4}}$
as in \cite[pp. 408]{Sarig-JAMS}.

\medskip
The estimates of $|\cos\angle(V^s,V^u)-\cos\alpha(x)|$ work as in \cite{Sarig-JAMS} after using again
that $K_2\eta<\eta^{2\beta/3}$, in which case $K_3=24$. 
\end{proof}

\begin{proof}[Proof of Proposition \ref{Prop-graph-transform}.]
We follow the proofs of \cite[Prop. 4.12 and 4.14]{Sarig-JAMS},
with the modifications below.
\begin{enumerate}[$\circ$]
\item Pages 411--412: in claim 3, it is enough to have $|G'(0)|<\tfrac{1}{2}(q^s\wedge q^u)^{\beta/3}$.
Proceed as in \cite{Sarig-JAMS} to get that
$$
|G'(0)|< e^{-\chi+\ve}\left[|A||F'(0)|+\tfrac{2}{3}\ve^{\beta/3}(p^s\wedge p^u)^{\beta/3}
+6\ve (p^s\wedge p^u)^{\beta/3}\right]
$$
and then note that for $\ve>0$ small enough this is at most
\begin{align*}
&e^{-\chi+\ve}\left[\tfrac{1}{2}e^{-\chi}+\tfrac{2}{3}\ve^{\beta/3}+6\ve\right](p^s\wedge p^s)^{\beta/3}\\
&\leq e^{-\chi+\ve+\ve\beta/3}\left[\tfrac{1}{2}e^{-\chi}+\tfrac{2}{3}\ve^{\beta/3}+6\ve\right](q^s\wedge q^s)^{\beta/3}
<\tfrac{1}{2}(q^s\wedge q^u)^{\beta/3}.
\end{align*}
\item Page 412: in claim 4, it is enough to have $\|G'\|_0+\Hol{\beta/3}(G')<\tfrac{1}{2}$.
Proceed as in \cite{Sarig-JAMS} to get that
$\|G'\|_0+\Hol{\beta/3}(G')<e^{-\chi+3\ve}\left[\tfrac{1}{2}e^{-\chi}+\tfrac{3}{2}\ve\right]$.
This is $<\tfrac{1}{2}$ when $\ve>0$ is small.
\item Pages 414--415: in the proof of part 2, proceed as in \cite{Sarig-JAMS} to get that
$$
\|G_1-G_2\|_0\leq (|A|+3\ve^2)(1+\ve^2+3\ve^3)\|F_1-F_2\|_0
$$
and note that $ (|A|+3\ve^2)(1+\ve^2+3\ve^3)<(e^{-\chi}+3\ve^2)(1+\ve^2+3\ve^3)<e^{-\chi/2}$ when $\ve>0$ is small enough.
\end{enumerate}
\end{proof}

\begin{proof}[Proof of inequality {\rm (\ref{inequality-C})}.]
We will use assumption (A3) as stated in section \ref{Section-introduction}:
\begin{enumerate}[ii]
\item[(A3)] $\|df_x\|< d(x,\mathfs D)^{-a}$ and
$\|df^{-1}_x\|< d(x,\mathfs D)^{-a}$ for all $x\in M\backslash\mathfs D$.
\end{enumerate}
We have:
\begin{align*}
&s(f^{-1}(x))^2=2\sum_{n\geq 0}e^{2n\chi}\|df^ne^s_{f^{-1}(x)}\|^2=
2+2e^{2\chi}\|dfe^s_{f^{-1}(x)}\|^2\sum_{n\geq 0}e^{2n\chi}\|df^ne^s_x\|^2\\
&=
2+e^{2\chi}\|dfe^s_{f^{-1}(x)}\|^2s(x)^2\leq (1+e^{2\chi}\|dfe^s_{f^{-1}(x)}\|^2)s(x)^2.
\end{align*}
By (A3),
$\tfrac{s(f^{-1}(x))^2}{s(x)^2}\leq 1+e^{2\chi} d(f^{-1}(x),\mathfs D)^{-2a}\leq 1+e^{2\chi}\rho(x)^{-2a}$.
We also have that
\begin{align*}
&u(f^{-1}(x))^2=2\sum_{n\geq 0}e^{2n\chi}\|df^{-n}e^u_{f^{-1}(x)}\|^2=2\|df^{-1}e^u_x\|^{-2}
\sum_{n\geq 0}e^{2n\chi}\|df^{-(n+1)}e^u_x\|^2\\
&=2e^{-2\chi}\|df^{-1}e^u_x\|^{-2}\sum_{n\geq 1}e^{2n\chi}\|df^{-n}e^u_x\|^2=
e^{-2\chi}\|df^{-1}e^u_x\|^{-2}(u(x)^2-2)\\
&<\|df^{-1}e^u_x\|^{-2} u(x)^2,
\end{align*}
hence by (A6) we get that
$\tfrac{u(f^{-1}(x))^2}{u(x)^2}\leq \rho(x)^{-2a}<1+e^{2\chi}\rho(x)^{-2a}$.
Finally, applying (\ref{gen-ineq-angles}) for $L=df^{-1}_x$, $v=e^s_x$, $w=e^u_x$ and
using (A3), we have
$$
\tfrac{\sin\alpha(x)}{\sin\alpha(f^{-1}(x))}=\tfrac{\sin\angle(e^s_x,e^u_x)}{\sin\angle(df^{-1}_x e^s_x,df^{-1}_xe^u_x)}
\leq \|df^{-1}_x\|\|df_{f^{-1}(x)}\|<\rho(x)^{-2a}.
$$
Since $\|\cdot \|\leq \|\cdot\|_{\rm Frob}\leq \sqrt{2}\|\cdot\|$, the above inequalities and
Lemma \ref{Lemma-linear-reduction} give that
\begin{align*}
&\|C_\chi(f^{-1}(x))^{-1}\|\leq \|C_\chi(f^{-1}(x))^{-1}\|_{\rm Frob}
\leq\rho(x)^{-2a}\sqrt{1+e^{2\chi}\rho(x)^{-2a}}\|C_\chi(x)^{-1}\|_{\rm Frob}\\
&\leq 2\rho(x)^{-2a}(1+e^\chi\rho(x)^{-a})\|C_\chi(x)^{-1}\|.
\end{align*}
\end{proof}

\begin{proof}[Proof of Proposition \ref{Prop-stay-window}.]
The proof of part (2) is identical to the proof of \cite[Prop. 6.4]{Sarig-JAMS},
and the proof of part (1)(a)--(b) is identical to the proof of \cite[Prop. 6.3(1)--(2)]{Sarig-JAMS}.
To prove (1)(c), we make some modifications in the proof of \cite[Prop. 6.3(3))]{Sarig-JAMS}.
We start with the claim below.

\medskip
\noindent
{\sc Claim:} If $y,z\in D_x$ and $v\in T_yM,w\in T_zM$ with $\|v\|=\|w\|=1$ then
\begin{align*}
&|\|df_y^{\pm 1}(v)\|-\|df_z^{\pm 1}(w)\||\leq \mathfrak K\rho(x)^{-a}[d(y,z)^\beta+\|v-P_{z,y}w\|]\hspace{.2cm}\text{and}\\
&\left|\frac{\|df_y^{\pm 1}(v)\|}{\|df_z^{\pm 1}(w)\|}-1\right|\leq \mathfrak K\rho(x)^{-2a}[d(y,z)^\beta+\|v-P_{z,y}w\|].
\end{align*}
In particular
$\left|\log\|df_y^{\pm 1}(v)\|-\log\|df_z^{\pm 1}(w)\|\right|\leq \mathfrak K\rho(x)^{-2a}[d(y,z)^\beta+\|v-P_{z,y}w\|]$.

\medskip
\noindent
{\em Proof of the claim.} The inequalities are consequences of (A5)--(A7). Since
these assumptions are symmetric on $f$ and $f^{-1}$, we only prove the claim for $f$.
Note that:
\begin{align*}
&|\|df_y(v)\|-\|df_z(w)\||\leq \|\widetilde{df_y}(P_{y,x}v)-\widetilde{df_z}(P_{z,x}w)\|\\
%&\leq \|\widetilde{df_y}(P_{y,x}v)-\widetilde{df_z}(P_{y,x}v)\|+\|\widetilde{df_z}(P_{y,x}v)-\widetilde{df_z}(P_{z,x}w)\|\\
&\leq \|\widetilde{df_y}-\widetilde{df_z}\|+\|\widetilde{df_z}\|\|v-P_{z,y}w\|\leq
\mathfrak Kd(y,z)^\beta+\rho(x)^{-a}\|v-P_{z,y}w\|\\
&\leq \mathfrak K\rho(x)^{-a}[d(y,z)^\beta+\|v-P_{z,y}w\|].
\end{align*}
The second inequality follows from the first one and from (A7).

\medskip
Let us now prove part (1)(c). Write $V^s=V^s[\{\Psi_{x_n}^{p^s_n,p^u_n}\}_{n\geq 0}]$. By the claim,
\begin{align*}
&|\log\|df^n e^s_y\|-\log\|df^n e^s_z\||\leq \sum_{k=0}^{n-1}|\log\|df e^s_{f^k(y)}\|-\log\|df e^s_{f^k(z)}\||\\
&\leq \sum_{k=0}^{n-1}\mathfrak K\rho(x_k)^{-2a}[d(f^k(y),f^k(z))^\beta+
\|e^s_{f^k(y)}-P_{f^k(z),f^k(y)}e^s_{f^k(z)}\|].
\end{align*}
By part (1)(a) and the definition of $Q_\ve(x_k)$,
\begin{align*}
&\rho(x_k)^{-2a}d(f^k(y),f^k(z))^\beta<\ve^{1/12}Q_\ve(x_k)^{-\beta/36}6e^{-\frac{\beta\chi}{2} k}(p^s_0)^\beta\\
&<6\ve^{1/12}(p^s_k)^{-\beta/36}e^{-\frac{\beta\chi}{2} k}(p^s_0)^\beta.
\end{align*}
By (GPO2) we have $p^s_0\leq e^{\ve k}p^s_k$, then for small $\ve>0$ the last expression above is
\begin{align*}
\leq 6\ve^{1/12}(p^s_0)^{-\beta/36}e^{-\frac{\beta\chi}{2}k+\frac{\beta\ve}{36}k}(p^s_0)^\beta
< 6\ve^{1/12}e^{-\frac{\beta\chi}{3}k}(p^s_0)^{\beta/4}
\end{align*}
and thus
$$
\sum_{k=0}^{n-1}\mathfrak K\rho(x_k)^{-2a}d(f^k(y),f^k(z))^\beta\leq
\tfrac{6\mathfrak K\ve^{1/12}}{1-e^{-\frac{\beta\chi}{3}}}(p^s_0)^{\beta/4}<\tfrac{1}{2}(p^s_0)^{\beta/4}.
$$
We now estimate the second sum. Call $N_k:=\|e^s_{f^k(y)}-P_{f^k(z),f^k(y)}e^s_{f^k(z)}\|$.
Write $f^k(y)=\Psi_{x_k}(\un y_k)=\Psi_{x_k}(y_k,F_k(y_k))$ and
$f^k(z)=\Psi_{x_k}(\un z_k)=\Psi_{x_k}(z_k,F_k(z_k))$, where $F_k$ is the representing function
of $V^s[\{\Psi_{x_n}^{p^s_n,p^u_n}\}_{n\geq k}]$. In part (1), it is proved
that $\|\un y_k-\un z_k\|\leq 3p^s_0e^{-\frac{\chi}{2}k}$.
%, and let $\un v'_k=\left[\begin{array}{c}1\\F_k'(z_k)\end{array}\right]$. 
As in \cite[pp. 418--419]{Sarig-JAMS}, 
\begin{align*}
&N_k\leq 2\|C_\chi(x_k)^{-1}\|\|\un y_k-\un z_k\|^{\beta/3}\\
&\hspace{.85cm}+4\|C_\chi(x_k)^{-1}\|
\left\|\widetilde{d(\exp{x_k})_{C_\chi(x_k)\un y_k}\circ C_\chi(x_k)}-\widetilde{d(\exp{x_k})_{C_\chi(x_k)\un z_k}\circ C_\chi(x_k)}\right\|
\end{align*}
which, by (A3), is 
$\leq 2\|C_\chi(x_k)^{-1}\|\|\un y_k-\un z_k\|^{\beta/3}+4\|C_\chi(x_k)^{-1}\|\rho(x_k)^{-a}\|\un y_k-\un z_k\|$.
For $\ve>0$ small enough
\begin{align*}
&4\rho(x_k)^{-a}\|\un y_k-\un z_k\|^{\beta/72}\leq 12\rho(x_k)^{-a}(p^s_0)^{\beta/72}e^{-\frac{\beta\chi}{144}k}\\
&\leq 12\rho(x_k)^{-a}(p^s_k)^{\beta/72}e^{-\frac{\beta\chi}{144}k+\frac{\beta\ve}{72}k}
\leq 12\ve^{1/24}e^{-\frac{\beta\chi}{144}k+\frac{\beta\ve}{72}k}<1,
\end{align*}
thus $N_k\leq 3\|C_\chi(x_k)^{-1}\|\|\un y_k-\un z_k\|^{\beta/3}$. Hence for small $\ve>0$
\begin{align*}
&\rho(x_k)^{-2a}N_k\leq 3\|C_\chi(x_k)^{-1}\|\rho(x_k)^{-2a}\|\un y_k-\un z_k\|^{\beta/3}\\
&\leq 9\|C_\chi(x_k)^{-1}\|\rho(x_k)^{-2a}(p^s_0)^{\beta/3}e^{-\frac{\beta\chi}{6}k}\\
&\leq 9\|C_\chi(x_k)^{-1}\|\rho(x_k)^{-2a}(p^s_0)^{\beta/12}e^{-\frac{\beta\chi}{6}k}(p^s_0)^{\beta/4}\\
&\leq 9\|C_\chi(x_k)^{-1}\|\rho(x_k)^{-2a}(p^s_k)^{\beta/12}e^{-\frac{\beta\chi}{6}k+\frac{\beta\ve}{12}k}(p^s_0)^{\beta/4}\\
&\leq 9\|C_\chi(x_k)^{-1}\|(p^s_k)^{\beta/24}\rho(x_k)^{-2a}(p^s_k)^{\beta/36}e^{-\frac{\beta\chi}{6}k+\frac{\beta\ve}{12}k}(p^s_0)^{\beta/4}\\
&\leq 9\ve^{5/24}e^{-\frac{\beta\chi}{7}k}(p^s_0)^{\beta/4}
\end{align*}
and therefore
$$
\sum_{k=0}^{n-1}\mathfrak K\rho(x_k)^{-2a}\|e^s_{f^k(y)}-P_{f^k(z),f^k(y)}e^s_{f^k(z)}\|\leq
\tfrac{9\mathfrak K\ve^{5/24}}{1-e^{-\beta\chi/7}}(p^s_0)^{\beta/4}<\tfrac{1}{2}(p^s_0)^{\beta/4}.
$$
The conclusion is that $|\log\|df^n e^s_y\|-\log\|df^n e^s_z\||<(p^s_0)^{\beta/4}<Q_\ve(x)^{\beta/4}$.
\end{proof}

\begin{proof}[Proof of Lemma \ref{Lemma-rotations}]
It is enough to prove the case $n=0$. Write $\Psi_{x_0}^{p^s_0,p^u_0}=\Psi_{x}^{p^s,p^u}$,
$\Psi_{y_0}^{q^s_0,q^u_0}=\Psi_{y}^{q^s,q^u}$, $p=p^s\wedge p^u$, $q=q^s\wedge q^u$.
Write $\widetilde{C_\chi(x)}=R_xC_x$, $\widetilde{C_\chi(y)}=R_yC_y$.
Since $R_y^{-1}R_x$ is a rotation matrix, it is enough to estimate its angle.
As in \cite[pp. 372]{Sarig-JAMS}, $\exists\lambda\neq 0$ s.t.
$C_x\un a=\lambda[\widetilde{d(\exp{x})_{C_\chi(x)\un\zeta}}]^{-1}
[\widetilde{d(\exp{y})_{C_\chi(y)\un\eta}}]C_y\un b$ where:
\begin{enumerate}[$\circ$]
\item $\un\zeta\in R[10^{-2}p]$, $\un a=\colvec{1 \\ a}$ and $|a|<p^{\beta/3}$.
\item $\un\eta\in R[10^{-2}q]$, $\un b=\colvec{1 \\ b}$ and $|b|<q^{\beta/3}$.
\end{enumerate}
The proof is based on three claims. Write $\vec{v}\propto\vec{w}$ if $\vec{v}=t\vec{w}$ for some $t\neq 0$.

\medskip
\noindent
{\sc Claim 1:} $C_x\un a\propto R_x\colvec{1\pm p^{\beta/4}\\ \pm p^{\beta/4}}$ and
$C_y\un a\propto R_y\colvec{1\pm q^{\beta/4}\\ \pm q^{\beta/4}}$.

\medskip
The proof is the same as in \cite[pp. 372]{Sarig-JAMS}.

\medskip
\noindent
{\sc Claim 2:} If $x,y\in D_z$ and $\|v\|,\|w\|\leq \mathfrak r(z)$ then
$$
\|[\widetilde{d(\exp{x})_v}]^{-1}[\widetilde{d(\exp{y})_w}]-{\rm Id}\|
<2d(z,\mathfs D)^{-a}\Sas(v,w).
$$

\medskip
The proof is a direct consequence of (A2)--(A3). In particular, if we write
$E:=[\widetilde{d(\exp{x})_{C_\chi(x)\un\zeta}}]^{-1}[\widetilde{d(\exp{y})_{C_\chi(y)\un\eta}}]-{\rm Id}$
then
\begin{align*}
&\|E\|<2d(y,\mathfs D)^{-a}\Sas(C_\chi(x)\un\zeta,C_\chi(y)\un\eta)
\leq 4d(y,\mathfs D)^{-a}[d(x,y)+\|\un\zeta-\un\eta\|]\\
&<4d(y,\mathfs D)^{-a}(p+q)<8d(x,\mathfs D)^{-a}p+8d(y,\mathfs D)^{-a}q\ll p^{\beta/3}+q^{\beta/3}
\end{align*}
since $d(x,y)<25^{-1}(p+q)$ and $\|\un\zeta\|+\|\un\eta\|<10^{-2}(p+q)$.

\medskip
\noindent
{\sc Claim 3:} $R_x\colvec{1\\0}+\un\ve_1\propto R_y\colvec{1\\0}+\un\ve_2$ where
$\|\un\ve_1\|,\|\un\ve_2\|<3(p^{\beta/4}+q^{\beta/4})\leq 6\ve^{3/4}$.

\medskip
To see this, note that since $C_x\un a\propto (E+I)C_y\un b$, claim 1 gives that
$$
R_x\colvec{1\\0}+\underbrace{R_x\colvec{\pm p^{\beta/4}\\ \pm p^{\beta/4}}}_{=\un\ve_1}
\propto R_x\colvec{1\\0}+
\underbrace{R_y\colvec{\pm q^{\beta/4}\\ \pm q^{\beta/4}}+EC_y\un b}_{=\un\ve_2}
$$
and that $\|\un\ve_1\|\leq 2p^{\beta/4}$ and
$\|\un\ve_2\|\leq 2q^{\beta/4}+2(p^{\beta/3}+q^{\beta/3})<3(p^{\beta/4}+q^{\beta/4})$.
The remainder of the proof is identical to \cite[pp. 373]{Sarig-JAMS}.
\end{proof}

\bibliographystyle{alpha}
\bibliography{bibliography}{}

\end{document}